\documentclass[a4paper,leqno,11pt]{amsart}

\usepackage[utf8]{inputenc}
\usepackage[margin=2.54cm]{geometry}
\usepackage{amsmath, amsthm, mathrsfs}
\usepackage{amssymb}
\usepackage{mathtools}
\usepackage{etoolbox}
\usepackage{dsfont}
\usepackage{scalerel}
\usepackage{tikz-cd}
\usepackage[colorlinks,citecolor=red,linkcolor=blue]{hyperref}
\usepackage{stmaryrd}
\usepackage{esint}
\usepackage[style=numeric,sortcites,giveninits=true,url=false,doi=false,isbn=false,backref=true,maxbibnames=99,backend=biber]{biblatex}
\theoremstyle{definition}
\newtheorem{thm}{Theorem}[section]

\newtheorem{lem}[thm]{Lemma}
\newtheorem{cor}[thm]{Corollary}
\newtheorem{prop}[thm]{Proposition}
\newtheorem{rem}[thm]{Remark}

\newcommand{\C}{\mathbb{C}}
\newcommand{\R}{\mathbb{R}}
\newcommand{\N}{\mathbb{N}}
\newcommand{\Z}{\mathbb{Z}}
\newcommand{\Hil}{\mathcal{H}}
\renewcommand*\d{\mathop{}\!\mathrm{d}}
\newcommand{\test}{\mathcal{C}_c^{\infty}}
\newcommand{\Wv}{\mathcal{V}}
\newcommand{\Wvo}{\dot{\mathcal{V}}}
\newcommand{\Wvc}{\mathcal{V}_0}
\newcommand{\V}{|V|^{1/2}}
\newcommand{\gradx}{\nabla_{\|}}
\newcommand{\gradV}{\nabla_{\!\mu}}
\newcommand{\gradVt}{\nabla_{\!\mu(t)}}
\newcommand{\gradVp}{\nabla_{\!\mu}^{\|}}
\newcommand{\gradIV}{\nabla_{\!\mu}}
\newcommand{\gradAV}{\ola{A}\nabla_{\!\mu}}
\newcommand{\p}[1][t]{\partial_#1}
\newcommand{\HAaV}{\mathop{H_{A,a,V}}}
\newcommand{\oo}{\mu}
\newcommand{\cA}{\mathcal{A}} 
 
\newcommand{\App}{A_{\perp \perp}}
\newcommand{\Apv}{A_{\perp \parallel}}
\newcommand{\Avp}{A_{\parallel \perp}}
\newcommand{\Avv}{A_{\parallel \parallel}}
\newcommand{\Neu}{(\mathcal{N})_2}
\newcommand{\Reg}{(\mathcal{R})_2}
\newcommand{\SNeu}{(\overline{\mathcal{N}})_2}
\newcommand{\SReg}{(\overline{\mathcal{R}})_2}
\newcommand{\D}{\mathsf{D}}
\newcommand{\Nul}{\mathsf{N}}
\newcommand{\Ran}{\mathsf{R}}
\newcommand{\bigminus}{\scaleobj{2}{-}}

\let\div\relax
\DeclareMathOperator{\div}{div}

\let\Re\relax
\DeclareMathOperator{\Re}{Re}

\let\supp\relax
\DeclareMathOperator{\supp}{supp}

\let\Loc\relax
\DeclareMathOperator{\Loc}{loc}

\let\dist\relax
\DeclareMathOperator{\dist}{dist}

\let\sgn\relax
\DeclareMathOperator{\sgn}{sgn}

\newcommand\mtiny[1]{\mbox{\tiny\ensuremath{#1}}}
\newcommand{\Nlim}{\mathop{{\mtiny{\textit{N}_*}}\textrm{-lim}}\nolimits}

\DeclarePairedDelimiterXPP\norm[3]{}\lVert\rVert{\ifblank{#2}{}{_#2}\ifblank{#3}{}{^#3}}{#1}

\def\Xint#1{\mathchoice
   {\XXint\displaystyle\textstyle{#1}}%
   {\XXint\textstyle\scriptstyle{#1}}%
   {\XXint\scriptstyle\scriptscriptstyle{#1}}%
   {\XXint\scriptscriptstyle\scriptscriptstyle{#1}}%
   \!\int}
\def\XXint#1#2#3{{\setbox0=\hbox{$#1{#2#3}{\int}$}
     \vcenter{\hbox{$#2#3$}}\kern-.5\wd0}}

\def\dashint{\Xint-}

\def\Xiint#1{\mathchoice
   {\XXiint\displaystyle\textstyle{#1}}%
   {\XXiint\textstyle\scriptstyle{#1}}%
   {\XXiint\scriptstyle\scriptscriptstyle{#1}}%
   {\XXiint\scriptscriptstyle\scriptscriptstyle{#1}}%
   \!\iint}
\def\XXiint#1#2#3{{\setbox0=\hbox{$#1{#2#3}{\iint}$}
     \vcenter{\hbox{$#2#3$}}\kern-.5\wd0}}

\def\dashiint{\Xiint{\bigminus}}

\newcommand*\conj[1]{\overline{#1}}
\newcommand*\hhat[1]{\hspace{4.25pt} \widehat{\hspace{-4.25pt}\widehat{#1}}}
\newcommand*\ola[1]{\hspace{2.5pt}\overline{\hspace{-2.5pt}\mathcal{#1}\hspace{1pt}}\hspace{0.5pt}}
\newcommand*\ula[1]{\hspace{1pt}\underline{\hspace{-1pt}\mathcal{#1}\hspace{-0.5pt}}\hspace{1.5pt}}

\numberwithin{equation}{section}

\title[Solvability for non-smooth Schr\"{o}dinger equations]{Solvability for non-smooth Schr\"{o}dinger equations with singular potentials and square integrable data}
\author{Andrew J. Morris}
\author{Andrew J. Turner}
\address{A. J. Morris, School of Mathematics, University of Birmingham, Edgbaston, B15 2TT, UK}
\email{a.morris.2@bham.ac.uk}
\address{A. J. Turner, School of Mathematics, University of Birmingham, Edgbaston, B15 2TT, UK}
\email{a.j.turner@bham.ac.uk}
\date{\today} 

\begin{document}

\begin{abstract}
We develop a holomorphic functional calculus for first-order operators $DB$ to solve boundary value problems for Schr\"odinger equations $- \div A \nabla u + a V u = 0$ in the upper half-space $\R^{n+1}_+$ with $n\in\N$. This relies on  quadratic estimates for $DB$, which are proved for coefficients $A,a,V$ that are independent of the transversal direction to the boundary, and comprised of a complex-elliptic pair $(A,a)$ that are bounded and measurable, and a singular potential $V$ in either $L^{n/2}(\R^n)$ or the reverse H\"older class $B^{q}(\R^n)$ with $q\geq\max\{\tfrac{n}{2},2\}$. In the latter case, square function bounds are also shown to be equivalent to non-tangential maximal function bounds. This allows us to prove that the (Dirichlet) Regularity and Neumann boundary value problems with $L^2(\R^n)$-data are well-posed if and only if certain boundary trace operators defined by the functional calculus are isomorphisms. We prove this property when the principal coefficient matrix $A$ has either a Hermitian or block structure. More generally, the set of all complex coefficients for which the boundary value problems are well-posed is shown to be open.
\end{abstract}

\maketitle

\setcounter{tocdepth}{2}
\tableofcontents


\section{Introduction}\label{sec:intro}


Many exceptional advances in our understanding of elliptic boundary value problems under minimal regularity constraints have been driven by the proof of the Kato square root conjecture obtained by Auscher--Hofmann--Lacey--McIntosh--Tchamitchian in~\cite{AHLMcT02}. The connection between square roots of accretive operators and elliptic boundary value problems has an even longer history, of course, and an outstanding exposition is given by Kenig in~\cite[Section~2.5]{KenigCBMS}. We build on recent progress for free-space Laplace-type equations to solve Neumann and (Dirichlet) Regularity boundary value problems for Schr\"{o}dinger-type equations with singular potentials in certain reverse H\"{o}lder classes. The Kato square root results obtained for Schr\"{o}dinger operators by Gesztesy--Hofmann--Nichols in \cite{GHN16}  is a primary motivation for this work.

A key point in the transference from  Kato square root-type results is their connection with square function bounds which can often be used to derive non-tangential maximal function bounds to solve boundary value problems. This is especially the case when equation coefficients are independent of the transversal direction to the boundary. We consider such a case by treating boundary value problems on the upper half-space $\R_+^{n+1} \coloneqq \{ (t,x) : t \in (0,\infty),\ x\in\R^n\}$, which has the boundary $\partial\R^{n+1}_+\cong\R^n$, and equation coefficients that are $t$-independent. This is a standard prototype with the potential to allow for more general Lipschitz domains and transversal-dependent coefficients. Indeed, a standard pullback transforms the domain above a Lipschitz graph to the upper half-space, whilst transversal regularity is typically introduced as a perturbation of the $t$-independent case, given that further regularity must be imposed transversely in accordance with the results of Caffarelli--Fabes--Kenig in~\cite{CFK}.

Most generally, we consider Schr\"{o}dinger equations on $\R^{n+1}_+$ in dimension $n \in \N$ with a  locally integrable potential $V \in L^1_{\Loc}(\R^n)$ in the divergence form 
\begin{equation} \label{eqn:Schrodinger Equation}
    \HAaV  u (t,x) \coloneqq 
    -\sum\nolimits_{i=0}^n\sum\nolimits_{j=0}^n \partial_i (A_{i,j}\partial_j u)(t,x) + a (x) V (x) u (t,x)
    = 0
\end{equation}
where $\partial_0=\partial_t$ and $(\partial_1,\ldots,\partial_n)=\nabla_x$. A potential $V$ always refer to a measurable function that is either $\C$-valued or $[0,\infty)$-valued. For most of our results, the potential belongs to either $L^p(\R^n) \coloneqq L^p(\R^n;\C)$ for some $p\in[1,\infty)$, or the reverse H\"{o}lder class $B^q(\R^n) \coloneqq B^q(\R^n;[0,\infty))$ for some $q\in(1,\infty)$, so that either
\begin{align}\begin{split}
\label{eq:VLpBqdef}
    \|V\|_p
    \coloneqq  &\left( \int_{\R^n} |V|^p \right)^{1/p}
    <\infty \\
    &\textrm{or } \quad
    \llbracket V\rrbracket_q
    \coloneqq \inf\left\{ C \geq 0 : \left(\dashint_{Q} V^q\right)^{1/q} \leq C \dashint_Q V \text{ for all cubes $Q$ in $\R^n$}\right\}
    < \infty,
\end{split}\end{align}
where a cube $Q$ in $\R^n$ refers to an $n$-dimensional cube with side-length $\ell(Q)$, Lebesgue measure $|Q|$, and $\dashint_Q f \coloneqq  |Q|^{-1} \int_Q f$ for all $f\in L^1_{\Loc}(\R^n)$.

Meanwhile, the equation coefficients $A$ and $a$ are given by $t$-independent, complex-valued measurable functions, so that $A(t,x)=A(x) \in \mathcal{L} (\C^{n+1})$ and $a(t,x)=a(x)\in\C$ for almost every $(t,x)\in\R^{n+1}_+$, where $\mathcal{L} (\C^{n+1})$ is the space of $(n+1)\times(n+1)$ matrices with complex coefficients. We assume that there exist constants $0<\lambda\leq\Lambda<\infty$ such that the bound
\begin{equation}\label{eq:bddcoeff}
\max\left\{\|A\|_{L^{\infty} (\R^n; \mathcal{L} (\C^{n+1}))},
\|a\|_{L^{\infty}(\R^n)}\right\} \leq \Lambda
\end{equation}
and ellipticity
\begin{align}\label{eq:elliptest}\begin{split}
    \Re\int_{\R^n} \bigg(\left(A(x) \begin{bmatrix}  f(x)\\ \nabla g(x) \end{bmatrix} , \begin{bmatrix}  f(x)\\ \nabla g(x) \end{bmatrix}\right) 
    & + a(x)V(x)|g(x)|^2\bigg) \mathrm{d}x \\ 
    &\geq \lambda
    \int_{\R^n} \left(|f(x)|^2 + |\nabla g(x)|^2 + |V(x)||g(x)|^2\right) \mathrm{d}x
\end{split}\end{align}
hold for all $f,\, g \in \test(\R^n)$, where $(\xi,\zeta)\coloneqq\xi\cdot\bar\zeta$ is the Euclidean inner-product for $\xi,\zeta\in\C^{n+1}$.

There is an extensive literature devoted to these equations in the free-space case, that is when $V\equiv0$, so we restrict our attention here to works that consider nonzero singular potentials $V$. The pioneering work of Shen in~\cite{Shen94} used the method of layer potentials to uniquely solve the Neumann problem for $L^p$ boundary data when $p \in (1, 2]$ for the Schr\"{o}dinger equation $-\Delta u + V u = 0$ on the domain above a Lipschitz graph. The result allowed for ($t$-dependent) potentials in the reverse H\"{o}lder class $V \in B^{\infty} (\R^{n+1})$, whereby $\|V\|_{L^\infty(Q)} \lesssim \dashint_Q V$ uniformly for all cubes $Q\subset\R^{n+1}$. Tao--Wang extended those results in~\cite{TW04} by allowing for singular potentials $V \in B^{n+1}(\R^{n+1})$, which permits greater reverse H\"{o}lder singularity, and data in the Hardy space $H^p(\R^n)$ for $p \in (1-\varepsilon , 1]$ and some $\varepsilon \in (0,1/n)$. Tao also solved the corresponding Regularity problem in \cite{Tao12} for data in a Hardy-type space adapted to the potential. In this context, the work of Yang in \cite{Yang18} shows that certain weak reverse H\"older estimates for solutions actually characterises the unique solvability of these boundary problems when $p\in(2,\infty)$.

The study of boundary value problems for equations $-\div (A \nabla u + b \cdot u) + c \cdot \nabla u + d u=0$ that have lower-order terms with Lebesgue integrable singularities has also received considerable recent attention. For real-valued coefficients, Sakellaris proved unique solvability in~\cite{Sakellaris19} for the Regularity problem for $L^2$-data on bounded Lipschitz domains assuming that the principal coefficients are H\"older continuous and the lower-order coefficients satisfy $d\geq \min\{\div b,\div c\}$. These results use the Green function estimates obtained Kim--Sakellaris in~\cite{KS19} (see also the scale-invariant results obtained by Sakellaris in~\cite{sakellaris19b}). In a similar direction, Mourgoglou considers variational Dirichlet problems in~\cite{mourgoglou19} assuming that the lower-order coefficients are in a local Stummel--Kato class, as well as developing a theory for the associated Green functions. We also mention the results for $t$-dependent coefficients obtained by Kenig--Pipher in~\cite{KP01} and Dindos--Petermichl--Pipher in~\cite{DPP07} for equations  $-\div (A \nabla u + b \cdot u)=0$ with singular drift terms, which built on the work by Hofmann--Lewis for parabolic equations with drift terms in~\cite{HL01}.

The most recent developments for $t$-independent, complex-valued coefficients on the upper half-space are related to the Kato square root results obtained for Schr\"{o}dinger operators by Gesztesy--Hofmann--Nichols in \cite{GHN16}. The perturbative approach they developed requires the lower-order coefficients to have sufficiently small norm in Lebesgue spaces that are critical for certain Sobolev embeddings. The results of Bortz--Hofmann--Luna--Mayboroda--Poggi  in~\cite{BHLMP22} show that the Dirichlet, Neumann and Regularity problems for $L^2$-data are well-posed in this context when the coefficients are either Hermitian or in the block form made precise below. There are now also estimates available for fundamental solutions and Green functions in this setting, including the work by Davey--Hill--Mayboroda in \cite{DHM18} under De Giorgi--Nash--Moser-type bounds on solutions, the results for magnetic Schr\"{o}dinger operators obtained by Mayboroda--Poggi in \cite{MP19}, and those obtained by Shen for generalized Schr\"{o}dinger operators in \cite{Shen99}.

The preceding results typically used layer potential representations for solutions and relied on adapting kernel estimates to account for the potential $V$. Instead, we obtain semigroup representations for solutions via the holomorphic functional calculus for a bisectorial operator $DB$, where $D$ is a self-adjoint first-order differential operator and $B$ is a multiplication operator with an accretivity that encodes the ellipticity of the second-order equation. This allows us to obtain representations for solutions even when kernel estimates are not available. This `first-order approach' was developed in earnest by Auscher--Axelsson--McIntosh in \cite{AAMc10-1} and~\cite{AAMc10-2} for $t$-independent free-space equations $\div A \nabla u = 0$. It has since been used widely to treat square integrable data, including $t$-dependent systems by Auscher--Axelsson in~\cite{AA11}, triangular coefficient structures by Auscher--McIntosh--Mourgoglou in~\cite{AMcM13}, degenerate elliptic equations by Auscher--Ros\'{e}n--Rule in~\cite{ARR15}, and parabolic equations by Auscher--Egert--Nystr\"{o}m in~\cite{AEN16-1}. There is also a growing literature devoted to $L^p$-type data.

In the first-order approach, square function bounds manifest as quadratic estimates
\[
\int_0^{\infty} \|t DB (I + (tDB)^2)^{-1} u \|_2^2 \frac{\d t}{t} \eqsim \|u\|_2^2
\]
for all $u$ in the closure of the range $\overline{\Ran(DB)}$. These guarantee that a holomorphic functional calculus for $DB$ is bounded (see Appendix~\ref{sec:app}), which underpins the development of the theory. Our approach to Schr\"{o}dinger operators is motivated by homogeneity considerations (see \eqref{eq:VLn} and \eqref{eq:VBn}) which suggest that such singular potentials $V$ should be encoded in the operator $D$ via the adapted gradient $\gradV\coloneqq (\nabla,\V)$. This leads us to consider a zeroth-order additive perturbation of the free-space operator. When the potential is not in $L^\infty$, then this perturbation is not even a bounded operator in $L^2$, so new ideas were needed to completely rework the proof of the free-space estimates. We obtain the associated quadratic estimates in Theorem~\ref{thm:QE} when $V \in B^{q}(\R^n)$ with $q\geq\max\{\tfrac{n}{2},2\}$ and $n\geq 3$, as well as when $V\in L^{n/2}(\R^n)$ with sufficiently small norm and $n\geq 5$. The dimensional restrictions result from the Sobolev embedding \eqref{eq:Sobolev} that is applied once when $n\geq3$ and twice when $n\geq5$ (see Lemmas~\ref{lem:Riesz Transform Estimates with small norm}--\ref{lem:Interpolation Lemma}). Alternative Sobolev embeddings can be applied to extend the results in this paper to the case when $V\in B^2 (\R^n)$ and $n\in\{1,2\}$, and that case will be treated in a forthcoming paper by Dumont--Morris.

Quadratic estimates for homogeneous $DB$-type operators were introduced in the work by Auscher--McIntosh--Nahmod in~\cite{AMcN97}, prior to the complete resolution of the Kato conjecture, and then developed for Dirac-type operators by Axelsson--Keith--McIntosh in \cite{AKMc06} and more generally by Auscher--Axelsson--McIntosh in \cite{AAMc10-1}. A key difference with the operators treated here is that the coercivity estimate $\|\nabla^2 f\|_2 \lesssim \|\Delta f\|_2$ and cancellation $\int \nabla g= 0$ (for $g$ with compact support) which were essential in the free-space case (see, for instance, \cite[Theorem~3.4]{AAMc10-1} and \cite[(H7)--(H8) and Theorem~3.1(iii)]{AKMc06}) need to be modified to account for the potential. 

If we interpret this coercivity as a bound on the Riesz transform $\nabla^2\Delta^{-1}$, then $L^2$-estimates for the Riesz transforms $V(-\Delta + V)^{-1}$, $\nabla^2(-\Delta + V)^{-1}$ and $\V \nabla(-\Delta + V)^{-1}$ are natural substitutes. These were obtained by Shen in~\cite{Shen95} and Auscher--Ben Ali in~\cite{AB07} when $V\in B^{q}(\R^n)$ with $q\geq\max\{\tfrac{n}{2},2\}$ and follow from Sobolev embedding in the case $V\in L^{n/2}(\R^n)$ with small norm (see Lemma~\ref{lem:Riesz Transform Estimates with small norm}). The reduction to a Carleson measure estimate then needs to be arranged carefully so that when the Poincar\'{e} inequality is applied it does not introduce derivatives on components of $DB$ that contain the potential term $\V$. The idea, introduced by Bailey in~\cite{Bailey18}, of using a separate reduction for such components is essential for this purpose.

We overcome the lack of cancellation in the case $V\in B^{q}(\R^n)$ with $q\geq\max\{\tfrac{n}{2},2\}$ by introducing a dichotomy amongst the Euclidean dyadic cube structure. A cube $Q$ in $\R^n$ is called \textit{homogeneous} when $\ell(Q)^2 \dashint_Q V \leq 1$ and it is called \textit{inhomogeneous} otherwise. This terminology is chosen to reflect the fact that when $V\equiv 0$ and the Schr\"{o}dinger equation reduces to the homogeneous Laplace equation, then all cubes are classified as homogeneous, whereas when $V\equiv 1$ and the equation is inhomogeneous, then cubes of side-length $\ell(Q)> 1$ become classified as inhomogeneous. In the latter case, Axelsson--Keith--McIntosh introduced a local cancellation property suited to quadratic estimates at small ($t\leq1$) scales, and a coercivity estimate suited to quadratic estimates at large ($t>1$) scales (see~\cite[(H7)--(H8)]{AKMc06-2}). This dichotomy of scales also features in the work by Morris on manifolds in~\cite{Morris12}.

We obtain a local cancellation property for homogeneous cubes in the case $V\in B^{q}(\R^n)$ with $q\geq\max\{\tfrac{n}{2},2\}$ which generalises these results (see Lemma~\ref{lem:Interpolation Lemma}). A substitute for coercivity on inhomogeneous cubes is then provided by the local Fefferman--Phong inequality
\[
\min\left\{1,\ell(Q)^2 \dashint_Q V\right\} \int_Q |f|^2
\lesssim \ell(Q)^2 \int_Q |\nabla_\mu f|^2
\]
obtained by Shen in~\cite{Shen95} and Auscher--Ben Ali in~\cite{AB07}. The consideration of homogeneous and inhomogeneous cubes thus generalises the simpler dichotomy of small and large side-length cubes that is relevant in the case $V\equiv1$. The case when $V\in L^{n/2}(\R^n)$ is simpler in this respect, since Sobolev embedding provides a substitute for the cancellation property on all cubes.

As an initial application, we obtain new Kato square root-type estimates, extending results of Bailey in~\cite{Bailey18} by allowing for bounded elliptic coefficients $b$ in multiplicative perturbations $b\mathscr{H}_{A,a,V}$ of operators $\mathscr{H}_{A,a,V}\coloneqq -\div A \nabla + aV$ acting on the boundary $\partial\R^{n+1}_+\cong\R^n$. If $V\in B^{q}(\R^n)$ with $q\geq\max\{\tfrac{n}{2},2\}$ and $n\geq 3$, then we characterise the domain of the accretive square root of $b\mathscr{H}_{A,a,V}$ by establishing the \textit{a priori} estimate
\[
\|(b\mathscr{H}_{A,a,V})^{1/2} f \|_2 \eqsim \|\nabla f\|_2 + \|\V f \|_2
\]
for all $f \in \test(\R^n)$ (see Corollary~\ref{cor:Kato for RH}). This extension is possible because we consider both $DB$ and $BD$-type operators in which the component structure of $B$ is less limited than that permitted by the Dirac-type structure considered by Bailey. We also establish the analogous result when $V\in L^{n/2}(\R^n)$ and $n\geq 5$ without any smallness condition on the norm (see Corollary~\ref{cor:Kato for Ln/2}), extending results of Gesztesy--Hofmann--Nichols in \cite{GHN16}. This is achieved by using a scaling argument to reduce matters to the quadratic estimates obtained for potentials with sufficiently small norm. The implicit constants then depend naturally on the size of the norm $\|V\|_{n/2}$. A more detailed comparison with known results in this regard is given at the start of Section~\ref{ssec:KBE}.

Once the quadratic estimates are established, we show that weak solutions of the Schr\"{o}dinger equation are characterised by weak solutions $F \in L^2_{\Loc} (\R_+ ; L^2 (\R^{n} ; \C^{n+2}))$ of Cauchy--Riemann-type systems $\p F + DB F = 0$ acting on $t$ in $\R_+$ (see Proposition~\ref{prop:Reduction to first-order}). The operator $DB$ is a bisectorial perturbation of the first-order operator $D$, which can have spectrum across the entire real line, so it does not generate an analytic semigroup. Instead, the operator $[DB]\coloneqq\sqrt{(DB)^2}$ defined via the functional calculus for $DB$ on $\Hil\coloneqq\overline{\Ran(D)}$ generates an analytic semigroup $e^{-t[DB]}$, which we can use to solve  initial value problems for the Cauchy--Riemann-type system. The quadratic estimates for $DB$ then provide a crucial topological decomposition $\Hil=\Hil^+\oplus\Hil^-$ into spectral subspaces $\Hil^\pm$ for $DB$. This allows us to show that solutions of $\p F + DB F= 0$ on $\R_+$ have a boundary trace $f$ in $\Hil^+$ such that $F(t)=e^{-t[DB]}f$ for all $t\geq0$ (see Theorem~\ref{Thm:First order solutions are semigroups}).

We follow the approach of Auscher--Axelsson in~\cite{AA11} to obtain these results, but we are forced to avoid their mollification arguments because it is not clear whether $\mathcal{C}^\infty_c(\R^n;\C^{n+1})$ is dense in the domain $\D(\gradV^*)$ of the adjoint operator. We achieve this by enlarging the class of test functions associated with the weak formulation of the Cauchy--Riemann-type system. Instead of the usual distributional space, our test functions belong to $L^2(\mathbb{R}_+;\D(D)) \cap \test(\R_+ ; L^2 (\R^{n} ; \C^{n+2}))$, which resembles the test function space considered in related work by Auscher--Axelsson in \cite{AA2011Remarks}. This weakens the tangential smoothness required on the test functions so that the mollification can be avoided, but the equivalence between first- and second-order solutions is preserved, since $\D(D)=\D(\gradV)\oplus \D(\gradV^*)$. The key insight is that density can be used in the $\D(\gradV)$ component whilst the equivalence in the $\D(\gradV^*)$ component only requires the $t$-independence of the coefficients.

We show that the Neumann and Regularity problems for the Schr\"{o}dinger equation are well-posed precisely when certain isomorphisms exist between the first-order trace space $\Hil^+$ and the $L^2$-spaces of second-order boundary data (see  Proposition~\ref{prop:WP equivalence}). The structure of the principal elliptic coefficient matrix $A$ only becomes a limiting consideration when one attempts to prove that these mappings are isomorphisms. The coefficient matrix $A$ is called \textit{Hermitian} when its conjugate transpose satisfies $A^*(x)=A(x)$ for almost every $x\in\R^n$. It is called \textit{block} when $A_{i,0}(x)=A_{0,j}(x)=0$ for all $i,j\in\{1,\ldots,n\}$ and almost every $x\in\R^n$. We follow the strategy of Auscher--Axelsson--McIntosh in \cite{AAMc10-2} and establish Rellich-type estimates to obtain the required boundary isomorphisms for both such coefficient structures in Section~\ref{sec:ADD}.

The passage from quadratic estimates to the non-tangential maximal function bounds needed to solve these boundary value problems is achieved when $V\in B^{q}(\R^n)$ with $q\geq\max\{\tfrac{n}{2},2\}$ in Theorem~\ref{thm:First-Order NT Control}. We adapt the method developed by Auscher--Axelsson--Hofmann in~\cite[Proposition~2.56]{AAH08} for this purpose. This requires us to establish reverse H\"{o}lder estimates for solutions of the Schr\"{o}dinger equation, and $L^p$-type resolvent bounds for $DB$ when $p$ belongs to a small neighbourhood around $2$. In contrast to the free-space equation, if $c\in\R$ and $H_{A,a,V}u=0$, then $u-c$ is a solution of the inhomogeneous equation $H_{A,a,V}(u-c)=-aVc$. This leads us to derive a Caccioppoli inequality for such solutions, which we combine with a Fefferman--Phong inequality for $\V$ to obtain the required reverse H\"{o}lder estimates for $\gradV u$. Meanwhile, the $L^p$-type resolvent bounds are obtained from interpolation theory for a scale of Sobolev spaces $\Wv^{1,p}(\R^n)$ adapted to the potential. It is important for us to prove that the interpolation constants depend on the potential only in terms of its reverse H\"{o}lder constant $\llbracket V\rrbracket_{q}$, since this quantity is preserved by the scaling of the potential that is used to obtain the resolvent bounds.

We now introduce some notation and  preliminaries to state the main results of the paper. The non-tangential maximal operator ${N}_*$ is defined for $F \in L^2_{\Loc} (\R_+^{n+1})$ and $x\in\R^n$ by
\begin{equation}\label{eq:N*def}
    ({N}_* F) (x) \coloneqq \sup_{t > 0} \left( \dashiint_{W(x,t)} |F (s,y)|^2 \d y \d s \right)^{1/2},
\end{equation}
where $W(x,t) \coloneqq (t , 2t) \times Q(x,t)$ denotes the \textit{Whitney cube} in $\R^{n+1}_+$ of scale $t>0$ above the cube $Q(x,t)\coloneqq\{y\in\R^n: |y-x|_\infty<t\}$ in $\R^n$. We shall write that \textit{the Whitney averages of $F$ converge to $f$ almost everywhere on $\R^n$}, or simply $\Nlim_{t\to 0^+} F(t,\cdot)=f$, if $f \in L^2_{\Loc} (\R^n)$ and
\begin{equation} \label{eqn:convergence on Whitney averages}
    \lim_{t \to 0^+} \dashiint_{W(x,t)} | F(s,y) - f(x) |^2 \d y \d s = 0
\end{equation}
for almost every $x \in \R^n$.

For weakly differentiable functions $u:\R^{n+1}_+\to\C$ and $f:\R^{n}\to\C$, there is also the notation
\[
\gradIV u(t,x) \coloneqq 
    \begin{bmatrix}
    \p u(t,x) \\
    \gradVp u(t,x)\\
    \end{bmatrix},
\quad
\gradVp u(t,x) \coloneqq 
    \begin{bmatrix}
    \nabla_x u(t,x)\\
    |V(x)|^{1/2} u(t,x)
    \end{bmatrix}
\ \text{ and }\ 
\gradV f(x) \coloneqq \begin{bmatrix}
    \nabla f(x)\\ |V(x)|^{1/2} f(x)
    \end{bmatrix}
\]
for $t>0$ and $x\in\R^n$. We adapt the standard Sobolev space $W^{1,2}(\R^{n+1}_+)$ to account for these adapted gradient operators in Section \ref{section:SobSpa}. For an open set $\Omega\subseteq \R^{n+1}_+$, the space $\Wv^{1,2}_{\Loc}(\Omega)$ consists of $u$ in the local Sobolev space $W^{1,2}_{\Loc}(\Omega)$ such that 
\[
\|u\|_{\Wv^{1,2}(\Omega')}^2
    \coloneqq 
    \iint_{\Omega'} (|\p u(t,x)|^2 + |\nabla_x u(t,x)|^2 + |V(x)| |u(t,x)|^2) \d t \d x < \infty
\]
for all open sets $\Omega'$ whose closure $\overline{\Omega'}$ is compact and contained in $\Omega$ (henceforth denoted $\Omega'\Subset\Omega$). Using this space, we shall write that \textit{$u$ is a weak solution of $- \div A \nabla u + a V u = 0$ in $\Omega$}, or simply \textit{$\HAaV  u  = 0$ in $\Omega$}, if $u \in \Wv^{1,2}_{\Loc}(\Omega)$ and 
\begin{equation}\label{eq:weak2nd}
    \int_\Omega (A \nabla u \cdot \overline{\nabla v} + a V u \overline{v}) = 0
\end{equation}
for all smooth compactly supported functions $v \in \test (\Omega)$. A homogeneous Sobolev-type space $\Wvo^{1,2} (\R^n)$ is also defined on the boundary as the completion of $\mathcal{C}_c^{\infty}(\R^n)$ with the norm
\[
\|f\|_{\Wvo^{1,2}(\R^n)}^2
    \coloneqq 
    \int_{\R^n} (|\nabla f(x)|^2 + |V(x)| |f(x)|^2) \d x
\]
and identified as a subspace of $L^{2^*}(\R^n)$, where $2^*\coloneqq 2n/(n-2)$ is the usual Sobolev exponent.

Given $h \in L^2 (\R^n)$ and $g \in \Wvo^{1,2} (\R^n)$, the Neumann and Regularity boundary value problems require functions $u$ in $\mathcal{V}^{1,2}_{\Loc}(\R^{n+1}_+)$ with the following properties:
\begin{equation*}
    \Neu
    \begin{cases}
        & \HAaV  u  = 0 \text{ in } \R_+^{n+1} \\
        & {N}_* (\gradIV u) \in L^2 (\R^n) \\
        & \Nlim_{t \to 0^+} \partial_{\nu_A} u (t, \cdot) = h
    \end{cases}
  \qquad\qquad
    \Reg
    \begin{cases}
        & \HAaV  u  = 0 \text{ in } \R_+^{n+1} \\
        & {N}_* (\gradIV u) \in L^2 (\R^n) \\
        & \Nlim_{t \to 0^+} \gradVp u (t, \cdot) = \gradV g
    \end{cases}
\end{equation*}
The adapted conormal derivative is $\partial_{\nu_A} u \coloneqq (A \nabla u) \cdot e_0 = A_{0,0}\p u + \sum_{j=1}^n A_{1,j}\partial_{x_j} u$ and the limit notation means that the Whitney averages of $\partial_{\nu_A} u$ and $\gradVp u$ converge to $h$ and $\nabla_\mu g$, respectively, almost everywhere on $\R^n$, as in \eqref{eqn:convergence on Whitney averages}. The boundary value problems are called \textit{solvable} if for each boundary datum there exists at least one solution $u$ with the corresponding properties above. In the case $V\not\equiv0$, they are called \textit{uniquely posed} if for each boundary datum there exists at most one such solution $u$ in $\mathcal{V}^{1,2}_{\Loc}(\R^{n+1}_+)$. In the case $V\equiv0$, it is instead required that there exists at most one such solution $u$ in $W^{1,2}_{\Loc}(\R^{n+1}_+)$ modulo constants. The boundary value problems are called \textit{well-posed} if they are solvable and uniquely posed.

We now state our main result for the well-posedness of these boundary value problems and the equivalence between non-tangential maximal function bounds and square function bounds.

\begin{thm} \label{thm:Second-Order WP}
Suppose that $V \in B^{q}(\R^n)$ with $q\geq\max\{\tfrac{n}{2},2\}$, $n\geq 3$ and $\llbracket V\rrbracket_{q}\leq \upsilon<\infty$ whilst $(A,a)$ are $t$-independent, bounded and elliptic coefficients satisfying \eqref{eq:bddcoeff} and \eqref{eq:elliptest} with constants $0<\lambda\leq\Lambda<\infty$. If either $A$ is block, or $a$ is real-valued and $A$ is Hermitian, then the following properties hold:
\begin{enumerate}
    \item The Neumann problem $\Neu$ is well-posed with
    \[
    \int_0^{\infty} \| t \p (\gradIV u) \|_2^2 \frac{\d t}{t} \eqsim \|{N}_* (\gradIV u)\|_2^2 \eqsim \|h\|_2^2
    \]
    and $\lim_{t \to 0^+} \|\partial_{\nu_{A}} u(t,\cdot) - h\|_2 =0$, where $u$ is the unique solution for data $h$ in $L^2(\R^n)$.\vspace{3pt}
\item The Regularity problem $\Reg$ is well-posed with
\[
\int_0^{\infty} \| t \p (\gradIV u) \|_2^2 \frac{\d t}{t}
\eqsim \|{N}_* (\gradIV u)\|_2^2
\eqsim \| \gradV g \|_2^2
\]
and $\lim_{t \to 0^+} \|\gradVp u (t,\cdot)- \gradV g\|_2 = 0$, where $u$ is the unique solution for data $g$ in $\Wvo^{1,2}(\R^n)$.
\end{enumerate}
The implicit constants in each estimate above depend only on $n$, $\lambda$, $\Lambda$ and $\upsilon$.
\end{thm}

We also prove that the functional calculus bounds implied by the quadratic estimates for $DB$ depend analytically on perturbations of the coefficients $B$ with respect to the $L^{\infty}$-norm. This allows us to obtain the perturbation results for well-posedness below.

\begin{thm} \label{thm:Second-Order WP pert}
Suppose that $V \in B^{q}(\R^n)$ with $q\geq\max\{\tfrac{n}{2},2\}$ and $n\geq 3$ whilst $(A_0,a_0)$ are $t$-independent, bounded and elliptic coefficients satisfying \eqref{eq:bddcoeff} and \eqref{eq:elliptest} with $0<\lambda_0\leq\Lambda_0<\infty$. If $A \in L^{\infty} (\R^n; \mathcal{L} (\C^{n+1}))$ and $a\in L^{\infty}(\R^n)$, then the following properties hold:
\begin{enumerate}
    \item If $\Neu$ with coefficients $(A_0,a_0)$ is well-posed, then there exists $\varepsilon\in(0,\lambda_0)$ such that $\Neu$ with coefficients $(A,a)$ is well-posed whenever $\max\{\|A-A_0\|_{\infty},\|a-a_0\|_{\infty}\} < \varepsilon$.
    \item If $\Reg$ with coefficients $(A_0,a_0)$ is well-posed, then there exists $\varepsilon\in(0,\lambda_0)$ such that $\Reg$ with coefficients $(A,a)$  is well-posed whenever $\max\{\|A-A_0\|_{\infty},\|a-a_0\|_{\infty}\}< \varepsilon$.
\end{enumerate} 
\end{thm}

The following Fatou-type result shows that the non-tangential maximal function control of the adapted gradient $\gradV u$ that is required for solvability of $\Neu$ and $\Reg$ is sufficient to guarantee that a weak solution has an $L^2$-boundary trace for each  problem. It also highlights the equivalence between square functions bounds and non-tangential maximal function bounds for such solutions.

\begin{thm} \label{thm:Second-Order Fatou}
Suppose that $V \in B^{q}(\R^n)$ with $q\geq\max\{\tfrac{n}{2},2\}$, $n\geq 3$ and $\llbracket V\rrbracket_{q}\leq \upsilon<\infty$ whilst $(A,a)$ are $t$-independent, bounded and elliptic coefficients satisfying \eqref{eq:bddcoeff} and \eqref{eq:elliptest} with constants $0<\lambda\leq\Lambda<\infty$. There exist subspaces $\Hil^\pm$ of $L^2(\R^{n};\C^{n+2})$ such that \[\Hil^+ \oplus \Hil^{-}=\{(h,\gradV g) : h\in L^2(\R^n),\, g\in\Wvo^{1,2}(\R^n)\}\] and the following property holds: If $\HAaV  u  = 0$ in $\R^{n+1}_+$ and ${N}_* (\gradIV u) \in L^2 (\R^n)$, then there exists $h\in L^2 (\R^n)$ and $g\in \Wvo^{1,2}(\R^n)$ such that $(h,\gradV g)\in \mathcal{H}^+$ and
\[
\int_0^{\infty} \| t \p (\gradIV u) \|_2^2 \frac{\d t}{t} \eqsim \|{N}_* (\gradIV u)\|_2^2 \eqsim \|h\|_2^2 + \|\gradV g\|_2^2,
\]
where the implicit constants depend only on $n$, $\lambda$, $\Lambda$ and $\upsilon$, whilst $\Nlim_{t \to 0^+} \partial_{\nu_{A}} u (\cdot,t) = h$, $\Nlim_{t \to 0^+} \gradV u(t,\cdot) = \gradV g$, $\lim_{t \to 0^+} \|\partial_{\nu_{A}} u (\cdot,t)-h\|_2 = 0$ and $\lim_{t \to 0^+} \|\gradV u(t,\cdot)-\gradV g\|_2=0$.
\end{thm}

We also obtain square function versions of the above results for boundary value problems when $V\in L^{n/2}(\R^n)$ and $n\geq 5$. The following result complements those in the series of papers initiated by Bortz--Hofmann--Luna--Mayboroda--Poggi in~\cite{BHLMP22} for equations with additional first-order coefficients. The non-tangential maximal function bounds needed to solve $\Neu$ and $\Reg$ for such potentials are the subject of the second and forthcoming paper in that series.

\begin{thm} \label{thm:SFBVP}
Suppose that $V\in L^{n/2}(\R^n)$ with $n\geq 5$ and $\|V\|_{n/2}\leq \upsilon<\infty$ whilst $(A,a)$ are $t$-independent, bounded and elliptic coefficients satisfying \eqref{eq:bddcoeff} and \eqref{eq:elliptest} with constants $0<\lambda\leq\Lambda<\infty$. If either $A$ is block, or the product $aV$ is real-valued and $A$ is Hermitian, then given $h \in L^2 (\R^n)$ and $g \in \Wvo^{1,2} (\R^n)$, there exist unique functions $u_h$ and $u_g$ in $\mathcal{V}^{1,2}_{\Loc}(\R^{n+1}_+)$ (when $V\equiv0$ uniqueness is instead in $W^{1,2}_{\Loc}(\R^{n+1}_+)$ modulo constants) with the following properties:
\begin{equation*}
    \SNeu
    \begin{cases}
        & \HAaV  u_h  = 0 \text{ in } \R_+^{n+1} \\
        & \int_0^{\infty} \| t \p (\gradIV u_h) \|_2^2 \frac{\d t}{t} < \infty \\
        & \lim_{t \to 0^+} \|\partial_{\nu_{A}} u_h (\cdot,t)-h\|_2 = 0
    \end{cases}
  \qquad\qquad
    \SReg
    \begin{cases}
        & \HAaV  u_g  = 0 \text{ in } \R_+^{n+1} \\
        & \int_0^{\infty} \| t \p (\gradIV u_g) \|_2^2 \frac{\d t}{t} < \infty \\
        & \lim_{t \to 0^+} \|\gradVp u_g(t,\cdot)-\gradV g\|_2=0
    \end{cases}
\end{equation*}
Moreover, it holds that
\begin{equation}\label{eq:SFEsol}
\int_0^{\infty} \| t \p (\gradIV u_h) \|_2^2 \frac{\d t}{t} \eqsim \|h\|_2^2
\quad\text{and}\quad
\int_0^{\infty} \| t \p (\gradIV u_g) \|_2^2 \frac{\d t}{t}
\eqsim \| \gradV g \|_2^2,
\end{equation}
where the implicit constants depend only on $n$, $\lambda$, $\Lambda$ and $\upsilon$.
\end{thm}

The paper is organised as follows. In Section~\ref{sec:pre}, we consider Sobolev spaces adapted to singular potentials satisfying \eqref{eq:VLpBqdef} and elliptic coefficients satisfying \eqref{eq:bddcoeff} and \eqref{eq:elliptest}. The main quadratic estimates in Theorem~\ref{thm:QE} are proved in Sections~\ref{ssec:CC}--\ref{ssec:CME}, with the application to Kato square root-type estimates in Section~\ref{ssec:KBE} and analytic perturbation results for the functional calculus in Section~\ref{ssec:andep}. The equivalence between weak solutions for the Schr\"{o}dinger equation and weak solutions for the associated Cauchy--Riemann-type system is given by Proposition~\ref{prop:Reduction to first-order}. The semigroup characterisation for solutions to the related first-order initial value problems follows in Theorem~\ref{Thm:First order solutions are semigroups}. The boundary isomorphisms for block and Hermitian coefficients are then established via Rellich-type estimates in Section~\ref{sec:ADD}. The nontangential maximal function estimates are obtained in Section~\ref{sec:NT}, assuming that $V\in B^{q}(\R^n)$ with $q\geq\max\{\tfrac{n}{2},2\}$, and the main equivalence with square function bounds is in
Theorem~\ref{thm:First-Order NT Control}.

The purpose of Section~\ref{sec:con} is to make precise the equivalence that has been established between well-posedness for second-order boundary value problems and isomorphisms for boundary trace mappings associated with first-order initial value problems. This brings together ideas developed throughout the paper and may provide a useful entry point for readers unfamiliar with the first-order approach. It is also used to conclude the proofs of Theorems~\ref{thm:Second-Order WP}--\ref{thm:SFBVP}. Meanwhile, Appendix~\ref{sec:app} contains an overview of the holomorphic functional calculus for bisectorial operators used throughout the paper, including the $\mathcal{F}$-functional calculus, analytic semigroups and quadratic estimates.

The following notation is used throughout the paper. The positive real line is $\R_+ \coloneqq (0,\infty)$ and the negative real line is $\R_- \coloneqq (-\infty,0)$. For $a,\, b \in \R$, writing $a \lesssim b$ means that there exists an implicit constant $C \in (0,\infty)$ such that $a \leq C b$, whilst $a \eqsim b$ means that $a \lesssim b$ and $b \lesssim a$. For $n,\, N\in\N$, the symbols $\|\cdot\|_2$ and $\langle \cdot , \cdot \rangle$ denote the norm and inner-product on $L^2(\R^n;\C^N)$, respectively, whilst $|\cdot|$ and $(\cdot , \cdot)$ denote the Euclidean norm and inner-product on $\C^N$, and $|\cdot|_{\mathcal{L}(\C^N)}$ denotes the associated norm on the space $\mathcal{L}(\C^N)$ of bounded linear operators on $\C^N$.

\subsection*{Acknowledgments}

Morris would especially like to thank the late Alan McIntosh for the inspiration that made this project possible. We would like to sincerely thank Julian Bailey and Pierre Portal for sharing preliminary versions of their recent work with us, which allowed us to remove an unnatural Riesz transform assumption in the proof of the main quadratic estimate. We are also deeply grateful to Moritz Egert for generously sharing his expertise and advice concerning the first-order functional calculus approach to boundary value problems. We would also like to thank Pascal Auscher, Simon Bortz, Arnaud Dumont, Steve Hofmann, Alessio Martini, Andreas Ros\'{e}n and Adam Sikora for helpful discussions and insights that have contributed significantly to this paper. We are also especially thankful for the comprehensive comments provided by the anonymous referee. 

This research was conducted primarily at the University of Birmingham. We also thank the Instituto de Ciencias Matem\'{a}ticas in Madrid for hosting us during the Research Term on Real Harmonic Analysis in 2018. This work was supported by the Engineering and Physical Sciences Research Council [EP/N509590/1] and the Royal Society [IES$\backslash$R3$\backslash$193232]. No data were created or analysed in this study.


\section{Preliminaries}\label{sec:pre}


We collect here the main tools needed to develop the first-order framework to solve boundary value problems for the Schr\"{o}dinger equation \eqref{eqn:Schrodinger Equation}. In the first subsection, we adapt the usual class of Sobolev spaces to account for the presence of singular potentials. This is motivated by Sobolev's inequality, which we show suggests that the square root of the potential $|V|^{1/2}$ should be subject to homogeneity considerations similar to those of a first-order differential operator. For potentials in a reverse H\"{o}lder class, a natural length scale adapted to the potential will also become apparent and for which we will recall a compatible Fefferman--Phong inequality. This will be essential for obtaining square function and non-tangential maximal function estimates. We then consider how the bound~\eqref{eq:bddcoeff} and ellipticity~\eqref{eq:elliptest} of the coefficients compares with related notions of ellipticity.


\subsection{Sobolev Spaces adapted to Singular Potentials}\label{section:SobSpa}


We begin by adapting the usual scale of Sobolev spaces $W^{1,p}(\Omega)$, for open subsets $\Omega$ of $\R^d$, $d\in\N$ and $p\in[1,\infty)$, to account for potentials $V$ satisfying \eqref{eq:VLpBqdef}. Here we use $d$ to denote the dimension for convenience so that the results can be applied with $d=n+1$ to provide Sobolev spaces on subsets of the upper half-space $\R^{n+1}_+$, and also with $d=n$ to provide Sobolev spaces on the boundary $\partial\R^{n+1}_+\cong\R^n$. The following notation will be convenient for this purpose. If $f \in L^1_{\Loc}(\Omega)$ has a weak derivative $\nabla f =  (\partial_1 f,\ldots,\partial_d f) \in L^1_{\Loc}(\Omega;\C^d)$, then $\gradV f:\Omega\to\C^{d+1}$ is the measurable function
\begin{equation}\label{eq:def_nabla_nu}
    \gradV f \coloneqq \begin{bmatrix}
    \nabla f\\|V|^{1/2} f 
    \end{bmatrix},
\end{equation}
where $(|V|^{1/2} f)(x)\coloneqq |V(x)|^{1/2} f(x)$ for all $x\in \Omega$.

A principal motivation for these spaces is a minor variant of a standard Sobolev inequality on $\R^d$, whereby if $f\in L^q(\R^d)$ for some $q\in[1,\infty]$, and $\nabla f \in L^p(\R^d)$ for some $p\in[1,d)$, then
\begin{equation}\label{eq:Sobolev}
    \|f\|_{p^*} \lesssim \|\nabla f\|_p,
\end{equation}
where $p^*\coloneqq pd/(d-p)$ is the Sobolev exponent for $\R^d$ and the implicit constant depends only on $d$. This can be proved directly on $\test(\R^d)$, whilst regularisation provides a sequence $(f_n)_{n\in\N}$ in $\test(\R^d)$ that converges to $f$ in $L^q(\R^d)$ such that $\nabla f_n$ converges to $\nabla f$ in $L^p(\R^d)$, whence $f_n$ also converges to $f$ in $L^{p^*}(\R^d)$ and \eqref{eq:Sobolev} holds (see, for instance, \cite[Chapter~V, Sections 2.1--2.5]{SteinSmall}). If $d\geq 3$, then the case $q=2$ gives the standard embedding of the Sobolev space $W^{1,p}(\R^d)$ in $L^{p^*}(\R^d)$, whilst the case $q=2^*$ and $p=2$ provides a realisation of the homogeneous Sobolev space $\dot{W}^{1,2}(\R^d)$ (see the discussion around ~\eqref{eq:dotVinj}) for use in Lemma~\ref{lem:RanD}.

The potentials we consider can be controlled by the Sobolev inequality~\eqref{eq:Sobolev} as follows: 

\noindent If $V\in L^{d/2}(\R^d)$, then H\"{o}lder's inequality implies that
\begin{equation}\label{eq:VLn}
    \||V|^{1/2} f\|_p
    \leq \|V\|_{d/2}^{1/2} \|f\|_{p^*};
    \end{equation}

\noindent If $V\in B^{d/2}(\R^{d})$, then H\"{o}lder's inequality implies the local variant
\begin{equation}\label{eq:VBn}
    \||V|^{1/2} f\|_{L^p(Q)}
    \leq \|V\|_{L^{d/2}(Q)}^{1/2} \|f\|_{L^{p^*}(Q)}
    \leq \llbracket V\rrbracket_{d/2}^{1/2} \left( \ell(Q)^2\dashint_Q |V| \right)^{1/2} \|f\|_{p^*}
\end{equation}
for all cubes $Q\subset \R^d$.

The following technical lemma provides the basis for the definition of the adapted Sobolev spaces. The proof uses ideas from the work of Badr in ~\cite[Proposition~2.14]{badr09}.

\begin{lem}\label{lem:WVbase}
Suppose that $d\in\N$, $p\in [1,\infty]$, $q\in [1,\infty]$ and $|V|^{1/2} \in L^1_{\Loc}(\Omega)$. If $(f_m)_{m\in\N}$ is a sequence of weakly differentiable functions in $L^q(\Omega)$ that converges to some $f$ in $L^q(\Omega)$, and $(\gradV f_m)_{m\in\N}$ is a Cauchy sequence in $L^p(\Omega;\C^{d+1})$, then $f$ is weakly differentiable and $\gradV f_m$ converges to $\gradV f$ in $L^p(\Omega;\C^{d+1})$.
\end{lem}

\begin{proof}
Suppose that $f_m$ and $f$ satisfy the hypotheses of the lemma, in which case $\nabla f_m$ converges to some $(F_1,\ldots, F_d)$ in $L^p(\Omega;\C^{d+1})$, and $ |V|^{1/2} f_m$ converges to some $F_{d+1}$ in $L^p(\Omega)$. It suffices to prove that $F_j = \partial_j f$ when $j\in\{1,\ldots,d\}$ whilst $F_{d+1}=|V|^{1/2}f$. If $j\in\{1,\ldots,d\}$ and $\phi \in \mathcal{C}_c^\infty(\Omega)$, then the definition of the weak derivative and H\"{o}lder's inequality give
\[
    \left| \left(\int F_j\phi \right)- \left( -\int f\partial_j\phi \right) \right|
    \leq \|F_j-\partial_jf_m\|_p\|\phi\|_{p'} + \|f_m-f\|_q\|\partial_j\phi\|_{q'},
\]
hence $\int F_j \phi = -\int f \partial_j \phi$ and $F_j = \partial_j f$. Next, since $f_m$ converges to $f$ in $L^q(\Omega)$, there exists a subsequence $f'_m$ that converges to $f$ almost everywhere on $\Omega$. It follows that $|V|^{1/2}f'_m$ converges to $|V|^{1/2}f$ almost everywhere on $\Omega$, as $|V(x)|^{1/2}<\infty$ for all $x\in\Omega$. Meanwhile, since $|V|^{1/2}f'_m$ converges to $F_{d+1}$ in $L^p(\Omega)$, there exists a subsequence $(f''_m)_{m\in\N}$ of $(f'_m)_{m\in\N}$ such that $|V|^{1/2}f''_m$ converges to $F_{d+1}$ almost everywhere on $\Omega$, hence $F_{d+1}=|V|^{1/2}f$.
\end{proof}

For each $p\in[1,\infty)$, we define the adapted Sobolev space $\Wv^{1,p}(\Omega)$ to be the set
\[
    \Wv^{1,p}(\Omega) \coloneqq  \{f\in L^p(\Omega) : \gradV f \in L^p(\Omega;\C^{d+1})\}
\]
with the norm
\[
    \|f\|_{\Wv^{1,p}(\Omega)}
    \coloneqq  \left(\|f\|_{L^p(\Omega)}^p + \|\gradV f\|_{L^p(\Omega;\C^{d+1})}^p\right)^{1/p}
\]
for all $f\in \Wv^{1,p}(\Omega)$. Lemma~\ref{lem:WVbase} shows that $\Wv^{1,p}(\Omega)$ is a Banach space whilst $\Wv^{1,2}(\Omega)$ is a  Hilbert space with the inner-product
\[
\langle f,g \rangle_{\Wv^{1,2}(\Omega)}
\coloneqq  \int_\Omega (f\overline{g} + \gradV f \cdot \overline{\gradV g})
\]
for all $f,g\in \Wv^{1,2}(\Omega)$. We also define $\Wv^{1,p}_{\Loc}(\Omega)$ to be the set of all $f\in L^p_{\Loc}(\Omega)$ such that $f\in \Wv^{1,p}(\Omega')$ for all open sets $\Omega'$ with compact closure $\overline{\Omega'}\subset\Omega$ (henceforth denoted by $\Omega'\Subset\Omega$). We then define $\Wvc^{1,p}(\Omega)$ to be the closure of $\mathcal{C}_c^{\infty}(\Omega)$ in $\Wv^{1,p}(\Omega)$.

If $V\in L^{d/2}(\R^d)$ and $p\in[1,d)$, then  $\|f\|_{\Wv^{1,p}}^p \eqsim \|f\|_p^p + \|\nabla f\|_p^p =: \|f\|_{W^{1,p}}^p$ by \eqref{eq:Sobolev} and \eqref{eq:VLn}, so $\Wv^{1,p}(\R^d)$ is equivalent to the usual Sobolev space $W^{1,p}(\R^d)$, for which density of $\mathcal{C}_c^{\infty}(\R^d)$ is well-known. The following lemma shows that $\Wv^{1,p}(\R^d)=\Wvc^{1,p}(\R^d)$ whenever $|V|^{p/2} \in L^1_{\Loc}(\R^d)$. The proof follows the strategy outlined by Davies in \cite[Theorem 1.8.1]{Davies89} but we include the details below for later reference. An alternative proof is given by Auscher--Ben Ali in \cite[Lemma~10.1]{AB07}.

\begin{lem}\label{lem:testdense}
If $d\in\N$, $p\in[1,\infty)$ and $|V|^{p/2}\in L^1_{\Loc}(\R^d)$, then $\test(\mathbb{R}^d)$ is dense in $\mathcal{V}^{1,p}(\R^d)$. Moreover, if also $p_0,\, p_1\in [1,p]$ and $f\in\mathcal{V}^{1,p_0}(\R^d)\cap\mathcal{V}^{1,p_1}(\R^d)$, then there exists a sequence in  $\test(\mathbb{R}^d)$ that converges to $f$ in both $\mathcal{V}^{1,p_0}(\R^d)$ and $\mathcal{V}^{1,p_1}(\R^d)$.
\end{lem}

\begin{proof}
We use $B(x,r)$ to denote the open  ball in $\R^d$ with centre $x\in\R^d$ and radius $r>0$. Suppose that $f\in \mathcal{V}^{1,p}(\R^d)$. For each $N\in\N$, consider the truncated approximant
\[
f_N(x)\coloneqq 
\begin{cases}
N, & \textrm{if } f(x) > N;\\
f(x), &\textrm{if } |f(x)| \leq N;\\
-N, &\textrm{if }  f(x) < -N,
\end{cases}
\]
for all $x\in \R^d$. Observe that $f_N\in \mathcal{V}^{1,p}(\R^d) \cap L^\infty(\R^d)$ with $|f_N(x)| + |\gradV f_N(x)| \leq |f(x)| + |\gradV f(x)|$ for almost every $x\in\R^d$ and all $N\in\N$ (see, for instance, \cite[Lemma 7.6]{GT77}), so $\lim_{N\to\infty}\|f_N-f\|_{\mathcal{V}^{1,p}(\R^d)}=0$ by Lebesgue's dominated convergence theorem. 

Next, fix a function $\eta\in\test(\R^d)$ with support in $B(0,1)$ such that $\eta (x) = 1$ for all $|x|\leq\tfrac{1}{2}$. For each $R\geq 1$, consider the compactly supported approximant 
\[
f_{N,R}(x)\coloneqq  \eta\left(\frac{x}{R}\right)f_N(x)
\]
for all $x\in \R^d$. Observe that $f_{N,R}\in \mathcal{V}^{1,p}(\R^d) \cap L^\infty(\R^d)$ with support in $B(0,R)$. We also have $|f_{N,R}(x)| + |\gradV f_{N,R}(x)| \leq (\|\eta\|_\infty+\|\nabla\eta\|_\infty)|f_N(x)| + \|\eta\|_\infty|\gradV f_N(x)|$ for almost every $x\in\R^d$ and all $R\geq1$, so $\lim_{R\to\infty}\|f_{N,R}-f_N\|_{\mathcal{V}^{1,p}(\R^d)}=0$ for each $N\in\N$ by Lebesgue's dominated convergence theorem.

Finally, fix a non-negative function $\rho\in\test(\R^d)$ with $\int \rho(x) \d x = 1$ and support in $B(0,1)$. For each $\varepsilon\in(0,1)$, consider the mollified approximant
\[
f_{N,R,\varepsilon}(x)\coloneqq  \varepsilon^{-d} \int_{\R^d} \rho\left(\frac{x-y}{\varepsilon}\right)f_{N,R}(y) \d y
=\varepsilon^{-d} \int_{B(x,\varepsilon)\cap B(0,R)} \rho\left(\frac{x-y}{\varepsilon}\right)f_{N,R}(y) \d y
\]
for all $x\in\R^d$. Observe that $f_{N,R,\varepsilon}\in \test(\R^d)$ with  $\lim_{\varepsilon\to0}\|f_{N,R,\varepsilon}-f_{N,R}\|_{W^{1,p}(\R^d)}=0$ and support in $\supp{(f_{N,R})}+B(0,\varepsilon)\subseteq B(0,2R)$ because $f_{N,R}\in W^{1,p}(\R^d)$ with support in $B(0,R)$ (see, for instance, \cite[Chapter~V, Proposition~1]{SteinSmall}). Meanwhile, the truncation afforded by $N$ and $R$ are the key to establishing the required convergence of $|V|^{1/2}f_{N,R,\varepsilon}$, since we have
\[
|f_{N,R,\varepsilon}(x)|
\lesssim_d \|\rho\|_\infty \|f_{N,R}\|_\infty \mathds{1}_{B(0,2R)}(x)
\leq N\|\rho\|_\infty\|\eta\|_\infty \mathds{1}_{B(0,2R)}(x)
\]
for all $x\in\R^d$ and $\varepsilon\in(0,1)$. In particular, we know that $\mathds{1}_{B(0,2R)}|V|^{1/2}\in L^p(\R^d)$ because $V^{p/2}\in L^1_{\Loc}(\R^d)$, and we have
$\lim_{\varepsilon\to0}f_{N,R,\varepsilon}(x)=f_{N,R}(x)$ for almost every $x\in\R^d$ (see, for instance, \cite[Chapter~III, Theorem~2]{SteinSmall}), so $\lim_{\varepsilon\to0}\||V|^{1/2}(f_{N,R,\varepsilon}-f_{N,R})\|_{L^p(\R^d)}=0$  for each $N\in\N$ and $R>1$ by Lebesgue's dominated convergence theorem. These results combined prove that $\lim_{\varepsilon\to0}\|f_{N,R,\varepsilon}-f_{N,R}\|_{\mathcal{V}^{1,p}(\R^d)}=0$ for each $N\in\N$ and $R>1$.

To conclude, we use a limit argument to obtain a sequence $(f_m)_{m\in\N}$ in $\test(\R^d)$ such that $\lim_{m\to
\infty}\|f_m-f\|_{\mathcal{V}^{1,p}(\R^d)}=0$ (for each $m\in \N$, set $f_m=f_{N,R,\varepsilon}$ with  $\varepsilon$, $R$ and $N$ chosen appropriately in turn), as required. The final part of the lemma follows at once, since if also $p_0,\, p_1\in [1,p]$, then both $|V|^{p_0/2}$ and $|V|^{p_1/2}$ are in $L^1_{\Loc}(\R^d)$, so the construction above provides a single approximating sequence that converges in both $\mathcal{V}^{1,p_0}(\R^d)$ and $\mathcal{V}^{1,p_1}(\R^d)$.
\end{proof}

We now define the homogeneous space $\Wvo^{1,2}(\R^d)$ to be the completion of the normed space consisting of the set $\mathcal{C}_c^{\infty}(\R^d)$ with the norm
\[
    \|f\|_{\Wvo^{1,2}(\R^d)}
    \coloneqq  \|\gradV f\|_{L^2(\R^d;\C^{d+1})}.
\]
The precompleted space is a genuine normed space, since if $f
\in L^1_{\Loc}(\R^d)$ and $\|\nabla f\|_{L^2(\R^d)}=0$, then $f$ is a constant function, so when $f$ is also in $\mathcal{C}_c^{\infty}(\R^d)$, it must hold that $f=0$. If $d \geq 3$, then the Sobolev inequality~\eqref{eq:Sobolev}, and Lemma \ref{lem:WVbase} in the case $q=2^*$ and $p=2$, show that there is an injective embedding from the completion into $L^{2^*}(\R^d)$, allowing us to henceforth identify it as the set
\begin{equation}\label{eq:dotVinj}
\Wvo^{1,2}(\R^d) = \{f \in L^{2^*}(\R^d) : \gradV f \in L^2(\R^d)\},
\end{equation}
with the norm equivalence
\[
    \|f\|_{\Wvo^{1,2}(\R^d)}
    = \|\gradV f\|_{L^2(\R^d;\C^{d+1})}
    \eqsim \left(\|f\|_{L^{2^*}(\R^d)}^2 + \|\gradV f\|_{L^2(\R^d;\C^{d+1})}^2\right)^{1/2}.
\]
The set inclusion $\Wvo^{1,2}(\R^d) \supseteq \{f \in L^{2^*}(\R^d) : \gradV f \in L^2(\R^d)\}$ requires the density of $\mathcal{C}_c^{\infty}(\R^d)$, with respect to the norm $\|\gradV f\|_{L^2(\R^d;\C^{d+1})}$, in the latter set, which follows from the arguments used to prove Lemma~\ref{lem:testdense}. If $V\in L^{d/2}(\R^d)$ and $d\geq 3$, then $\|f\|_{\Wvo^{1,2}}^2 \eqsim \|f\|_{2^*}^2 + \|\nabla f\|_2^2 \eqsim \|\nabla f\|_{2}^2$ by \eqref{eq:Sobolev} and \eqref{eq:VLn}, so in that case $\Wvo^{1,2}(\R^d)$ in \eqref{eq:dotVinj} is the realisation of the usual homogeneous Sobolev space $\dot{W}^{1,2}(\R^d)$ in which each equivalence class of locally integrable functions modulo constants $[g]\in L^1_{\Loc}(\R^d) / \C$ is identified with the unique function $f \in L^{2^*}(\R^d)$ such that $f\in[g]$.

We include below a version of the local Fefferman--Phong inequalities obtained by Shen in \cite{Shen95} and Auscher--Ben Ali in \cite{AB07}. Note that we do not use the improvement involving the exponent $\beta$ which appears in the latter reference. We can already see the utility of this estimate in the case $p=2$, since it bounds the integral of the potential that appears in ~\eqref{eq:VBn}. The case $p=2$, obtained by Shen in \cite[Lemma~1.9]{Shen95}, is used to obtain the quadratic estimates in Section~\ref{sec:QE}. The case $p=1$, obtained by Auscher--Ben Ali in \cite[Lemma~2.1]{AB07}, is used to obtain the non-tangential maximal function estimates in Section~\ref{sec:NT}.

\begin{prop}[Fefferman--Phong] \label{prop:Fefferman-Phong}
If $d\in\mathbb{N}$, $p\in[1,\infty)$, $q\in(1,\infty)$ and $V^{p/2}\in B^q(\R^{d})$ with $\llbracket V^{p/2}\rrbracket_{q}\leq \upsilon < \infty$, then
\begin{equation}\label{eq:FePh}
    \min\left\{1,\ell(Q)^p \dashint_Q V^{p/2}\right\} \int_Q |f|^p 
    \lesssim \ell(Q)^p \int_Q |\nabla_\mu f|^p
\end{equation}
for all cubes $Q\Subset\Omega$ and $f\in \mathcal{V}^{1,p}_{\Loc}(\Omega)$, where the implicit constant depends only on $d$ and $\upsilon$.
\end{prop}

\begin{proof}
This follows from \cite[Lemma~2.1]{AB07} when $f \in \mathcal{C}^1(\R^d)$, since $V^{p/2}$ is a  Muckenhoupt weight. More generally, if $f\in\mathcal{V}^{1,p}_{\Loc}(\Omega)$ and $Q\Subset\Omega$, then we can modify the proof of Lemma~\ref{lem:testdense} to obtain a sequence $(f_m)_{m\in\N}$ in $\mathcal{C}^\infty(Q)$ that converges to $f$ in $\mathcal{V}^{1,p}(Q)$ (the local approximation results can be found in \cite[Lemmas~7.2 and 7.3]{GT77}). The estimate \eqref{eq:FePh} holds for each $f_m$, since functions in $\mathcal{C}^\infty(Q)$ have extensions at least in $\mathcal{C}^1(\R^d)$, so the result follows.
\end{proof}


\subsection{The Elliptic Coefficients} 


We consider $a$ in $L^{\infty} (\R^n)$ and $A$ in $L^{\infty} (\R^n ; \mathcal{L} (\C^{n+1}))$ given by $t$-independent, complex-valued and bounded elliptic coefficients of the Schr\"{o}dinger equation as in \eqref{eqn:Schrodinger Equation}--\eqref{eq:elliptest}, so that $a(t,x)=a(x)\in\C$ and $A(t,x)=A(x) \in \mathcal{L} (\C^{n+1})$ for all $(t,x)\in\R^{n+1}_+$. It is convenient to express these coefficients via the matrix $\cA_{A,a,V}$ in $\mathcal{L} (\C^{n+2})$ given by
\begin{equation}\label{eq:AcA}
    A=: \begin{bmatrix}
    \App & \Apv \\
    \Avp & \Avv
    \end{bmatrix}
\quad\text{and}\quad
    \cA_{A,a,V}
    \coloneqq \begin{bmatrix}
    \App & \Apv & 0 \\
    \Avp & \Avv & 0 \\
    0 & 0 & a e^{i\arg{V}}
    \end{bmatrix},
\end{equation}
where $\App (x) \in \mathcal{L} (\C)$, $\Apv (x) \in \mathcal{L} (\C^n ; \C)$, $\Avp (x) \in \mathcal{L} (\C ; \C^n)$ and $\Avv (x) \in \mathcal{L} (\C^n)$ for all $x\in\R^n$, although typically we write $\cA$ instead of $\cA_{A,a,V}$ when the meaning of $(A,a,V)$ is clear from the context. We represent vectors $v \in \C^{n+2}=\C\oplus\C^n\oplus\C$ accordingly by writing 
\[
    v = \begin{bmatrix}
        v_{\perp} \\ 
        v_{\parallel} \\ v_{\oo}
    \end{bmatrix}
\]
where $v_{\perp} \in \C$, $v_{\parallel} \in \C^n$ and $v_{\oo} \in \C$.

The functions in $\mathcal{C}_c^{\infty}(\R^n)$ are dense in $\Wvo^{1,2}(\R^n)$ (see the discussion after Lemma~\ref{lem:testdense}), so the bounds and ellipticity on the coefficients in \eqref{eq:bddcoeff} and \eqref{eq:elliptest}, with $0<\lambda\leq \Lambda<\infty$, imply that
\begin{equation}\label{eq:bddellA}
\|\cA_{A,a,V}\|_\infty \leq \Lambda
\quad\text{ and }\quad
\Re \left\langle \cA_{A,a,V} \begin{bmatrix}  f\\ \gradV g \end{bmatrix} , \begin{bmatrix}  f\\ \gradV g \end{bmatrix}\right\rangle \geq \lambda (\|f\|_2^2+\|\gradV g\|_2^2)
\end{equation}
for all $f\in L^2(\R^n)$ and $g \in \Wvo^{1,2}(\R^n)$. Note that (as in \cite[p.~259]{AAMc10-2}) these conditions imply that
\begin{equation}\label{eq:ellimp}
\Re A_{\perp \perp}(x) \geq \lambda
\ \ \text{and}\ \
\Re \int_{\R^n} \left((A_{\parallel \parallel}(y) \nabla g(y), \nabla g (y)) + a(y) V(y) |g(y)|^2\right) \d y
\geq \lambda \|\gradV g \|_2^2
\end{equation}
for almost every $x \in \R^n$ and all $g \in \Wvo^{1,2} (\R^n)$. To see this, first set $f = \xi \mathds{1}_E$ and $g=0$ in \eqref{eq:bddellA} to obtain $\int_{E} (\Re A_{\perp \perp} (x) - \lambda) |\xi|^2\d x \geq 0$ for all compact sets $E\subset\R^n$, hence $\Re (A_{\perp \perp} (x) \xi, \xi) \geq \lambda |\xi|^2$ for all $\xi \in \C$ and almost every $x \in \R^n$. The second inequality follows by setting $f = 0$ in \eqref{eq:bddellA}. 

The ellipticity in \eqref{eq:bddellA} holds when the pointwise ellipticity $\Re (\cA(x)\xi,\xi) \geq \lambda |\xi|^2$ holds for all $\xi \in \C^{n+2}$ and almost every $x \in \R^n$. It is stronger, however, than the G{\aa}rding-type ellipticity
\[
    \Re \iint_{\R^{n+1}_+} \left(\cA (x) \begin{bmatrix}  \p f(t,x)\\ \gradV f(t,x) \end{bmatrix} , \begin{bmatrix} \p f(t,x)\\ \gradV f(t,x) \end{bmatrix}\right) \d x \d t
    \geq \lambda \iint_{\R^{n+1}_+} |\p f (t,x)|^2+|\gradV f (t,x)|^2\ \d x \d t 
\]
for all $f \in \test(\R^{n+1}_+)$.


\section{Quadratic Estimates for First-Order Operators with Singular Potentials}\label{sec:QE}


The focus of this section is to prove quadratic estimates for first-order operators of the $DB$ and $BD$ type introduced by  Auscher--McIntosh--Nahmod  in~\cite{AMcN97}, and then first developed by Auscher--Axelsson--McIntosh in \cite{AAMc10-1}, which we adapt to incorporate singular potentials. We will see in Section~\ref{sec:BV} that these operators arise naturally in the first-order approach to boundary value problems for the Schr\"{o}dinger equation~\eqref{eqn:Schrodinger Equation}. Throughout this section we adopt the framework from Section~\ref{sec:pre} with $d=n\geq3$, $\Omega=\R^n$ and $V\in L^1_{\Loc}(\R^n)$. This will allow us to solve boundary value problems on the upper half-space $\R^{n+1}_+$ by applying the quadratic estimates obtained here on the domain boundary $\partial\R^{n+1}_+\cong\R^n$.

To begin, the self-adjoint operator $D \colon \D (D) \subset L^2 (\R^n ; \C^{n+2}) \to L^2 (\R^n ; \C^{n+2})$ is defined by the first-order operator
\begin{equation}\label{eq:Ddef}
    D u \coloneqq -\begin{bmatrix}
    0 & \gradV^* \\
    \gradV & 0 
    \end{bmatrix} 
    \begin{bmatrix}
    u_\perp \\ (u_\|, u_\oo)
    \end{bmatrix}
    = -\begin{bmatrix}
    \gradV^*(u_\|, u_\oo) \\ \gradV u_\perp 
    \end{bmatrix}
\end{equation}
for all $u = (u_\perp, u_\|, u_\mu) \in L^2(\R^n) \oplus L^2(\R^n;\C^n) \oplus L^2(\R^n)$ in the domain
\[
    \D(D) \coloneqq \{u\in L^2(\R^n;\C^{n+2}) : u_\perp \in \mathcal{V}^{1,2}(\R^n) \textrm{ and } (u_\|,u_\oo) \in \D(\gradV^*)\},
\]
where $\gradV^*:\D(\gradV^*)\subseteq L^2(\R^n;\C^{n+1})\to L^2(\R^n)$ denotes the (Hermitian) adjoint of the operator $\gradV: \mathcal{V}^{1,2}(\R^n) \subseteq L^2(\R^n)\to L^2(\R^n;\C^{n+1})$ given by \eqref{eq:def_nabla_nu} on the domain $\D(\gradV)\coloneqq \mathcal{V}^{1,2}(\R^n)$. The latter operator is closed and densely defined by Lemmas~\ref{lem:WVbase} and~\ref{lem:testdense}. In particular, if $u_\perp\in \Wv^{1,2}(\R^n)$, $u_\| \in\D(\nabla^*)$ and $\V u_\oo \in L^2(\R^n)$, then setting $\div\coloneqq -\nabla^*$ we have
\begin{equation}\label{eq:Drep}
    Du
    =
    - \begin{bmatrix}
    0 & -\div & \V \\
    \nabla & 0 & 0 \\
    \V & 0 & 0
    \end{bmatrix}
    \begin{bmatrix}
    u_\perp \\
    u_\| \\
    u_\oo 
    \end{bmatrix}
    =-\begin{bmatrix}
    -\div u_\| + \V u_\oo \\
    \nabla u_\perp \\
    \V u_\perp
    \end{bmatrix}.
\end{equation}
For example, this holds when $u\in \test(\R^n;\C^{n+2})$. More generally, however, the representation $\gradV^*(u_\|,u_\oo)=-\div u_\| + \V u_\oo$ may not hold, as it is not clear if $\mathcal{C}^\infty_c(\R^n;\C^{n+1})$ is dense in $\D(\gradV^*)$. The proof of Proposition~\ref{prop:Off diagonal estimates} will show, however, that the domain of $D$ is invariant under multiplication by functions $\eta$ in $\test(\R^n)$ and the commutator $[D,\eta I]u\coloneqq  D (\eta u) - \eta D u$ satisfies the bound $|[D,\eta I]u| \leq |\nabla \eta | |u|$ for all $u\in\D(D)$, which was crucial in the case $V\equiv0$.

We have the orthogonal direct sum decomposition $L^2(\R^n;\C^{n+2}) = \Nul(D)\overset{\perp}{\oplus}\overline{\Ran(D)}$ because $D$ is self-adjoint. The null space of $D$ is given by
\[
    \Nul(D) = \{u\in L^2(\R^n;\C^{n+2}) : u_\perp=0 \textrm{ and } (u_\|,u_\oo) \in \Nul(\gradV^*)\},
\]
whilst the closure of its range is characterised by the following lemma.

\begin{lem}\label{lem:RanD}
If $n\geq3$ and $V\in L^1_{\Loc}(\R^n)$, then the closure of the range of $D$ in $L^2(\R^n;\C^{n+2})$ is
\[
    \overline{\Ran(D)} = \{u\in L^2(\R^n;\C^{n+2}) : u_\perp \in L^2(\R^n) \textrm{ and } (u_\|,u_\oo) = \gradV g \textrm{ for some } g \in \Wvo^{1,2}(\R^n)\}.
\]
\end{lem}

\begin{proof}
First, suppose that $u\in\overline{\Ran(D)}$, so then $u_\perp \in L^2(\R^n)$ and there exists a sequence $\{ g_m \}$ in $\Wv^{1,2}(\R^n)$ such that $\{ \gradV g_m \}$ converges to $(u_\|,u_\oo)$ in $L^2(\R^n;\C^{n+1})$. The Sobolev inequality~\eqref{eq:Sobolev} then implies that $\{ g_m \}$ is Cauchy and hence convergent to some function $g$ in $L^{2^*}(\R^n)$. Therefore, by Lemma \ref{lem:WVbase} in the case $q=2^*$, the sequence $\{ \gradV g_m \}$ must converge to $\gradV g$ in $L^2(\R^n;\C^{n+1})$, hence $(u_\|,u_\oo) = \gradV g$ and $g\in \Wvo^{1,2}(\R^n)$, as required.

For the converse, suppose that $u_\perp \in L^2(\R^n)$ and that $(u_\|,u_\oo) = \gradV g$ for some $g \in \Wvo^{1,2}(\R^n)$. If $v\in \Nul(D)$, then $v_\perp=0$ and $(v_\|,v_\oo)\in\Nul(\gradV^*)$, so $\langle u,v \rangle = \langle \gradV g,(v_\|,v_\oo)\rangle=0$. This allows us to conclude that $u\in [\Nul(D)]^\perp = \overline{\Ran(D)}$.
\end{proof}

Now suppose that $B \colon L^2 (\R^n ; \C^{n+2}) \to L^2 (\R^n ; \C^{n+2})$ is
a multiplication operator, in the sense that $Bu(x)\coloneqq B(x)u(x)$, where $B(x)\in\mathcal{L} (\C^{n+2})$, for all $x\in\R^n$ and $u\in L^2 (\R^n ; \C^{n+2})$. We assume that there exist constants $0<\kappa\leq K<\infty$ such that the multiplication coefficients have the bound and component structure
\begin{equation}\label{eq:bound on R(D)}
\|B\|_{L^{\infty} (\R^n ; \mathcal{L} (\C^{n+2}))} \leq K
\quad\text{ and }\quad
B = \begin{bmatrix}
B_{\perp \perp} & B_{\perp \parallel} & 0 \\
B_{\parallel \perp} & B_{\parallel \parallel} & 0 \\
0 & 0 & b
\end{bmatrix}
\end{equation}
whilst the multiplication operator is strictly accretive on $R(D)$ in the sense that
\begin{equation} \label{eq:elliptic on R(D)}
    \Re \langle B v , v \rangle \geq \kappa \|v\|_2^2
\end{equation}
for all $v\in\Ran(D)$. We denote $\kappa(B) \coloneqq  \inf_{v \in \Ran(D)\setminus\{0\}} \Re \langle B v , v \rangle / \|v\|_2^{2} \in (0,\infty)$ and the angle of accretivity of $B$ is then $\omega(B) \coloneqq \sup_{v \in \Ran (D) \setminus \{ 0 \} } |\arg\langle B v,v\rangle| \leq \cos^{-1}(\kappa(B)/\|B\|_\infty)$.

The results obtained by Auscher--Axelsson--McIntosh in~\cite[Propositions~3.3]{AAMc10-2} and \cite[Propositions~3.1]{AAMc10-1} now apply, since $D$ is self-adjoint whilst $B$ is bounded and strictly accretive, to show that $DB$ and $BD$
are closed and densely defined bisectorial operators of type $S_{\omega(B)}$ on $L^2(\R^n;\C^{n+2})$ with the resolvent bounds
\begin{equation}\label{eq:DBresbdd}
\norm{(\lambda I - DB)^{-1}}{}{} \lesssim \dist(\lambda , S_{\omega})^{-1}
\quad\text{ and }\quad
\norm{(\lambda I - BD)^{-1}}{}{} \lesssim \dist(\lambda , S_{\omega})^{-1} 
\end{equation}
for all $\lambda \in \C\setminus S_{\omega}$, where the implicit constants depend only on $\kappa$ and $K$. Moreover, we have the topological direct sum decompositions
\begin{equation}\label{eq:Hodge}
L^2 (\R^n ; \C^{n+2} ) = \overline{\Ran(DB)} \oplus \Nul(DB)
= \overline{\Ran(BD)} \oplus \Nul(BD)
\end{equation}
with $\overline{\Ran(DB)}= \overline{\Ran(D)}$, $B\Nul(DB)=\Nul(D)$, $\overline{\Ran(BD)} = B\overline{\Ran(D)}$ and $\Nul(BD)=\Nul(D)$.

The following quadratic estimates are the main result of this section. 

\begin{thm} \label{thm:QE}
If $n\geq 3$, $q\geq\max\{\tfrac{n}{2},2\}$ and $V \in B^{q}(\R^n)$ with $\llbracket V\rrbracket_{q}\leq \upsilon<\infty$, then
\[
\int_0^{\infty} \|t DB (I + (tDB)^2)^{-1} u \|_2^2 \frac{\d t}{t}
\lesssim \|u\|_2^2
\quad\text{ and }\quad
\int_0^{\infty} \|t BD (I + (tBD)^2)^{-1} u \|_2^2 \frac{\d t}{t}
\lesssim \|u\|_2^2
\]
for all $u \in L^2 (\R^n ; \C^{n+2})$, where the implicit constants depend only on $n$, $\kappa$, $K$ and $\upsilon$.

If $n\geq 5$ and $\|V\|_{n/2}<\varepsilon_n$, where $\varepsilon_n\in(0,1)$ is from Lemma~\ref{lem:Riesz Transform Estimates with small norm}, then these estimates hold with implicit constants that depend only on $n$, $\kappa$ and $K$.
\end{thm}

\begin{proof}
If $n\geq 3$, $q\geq\max\{\tfrac{n}{2},2\}$ and $V \in B^{q}(\R^n)$, then the estimate for $DB$ is obtained from Propositions~\ref{prop:Principle Part Approximation} and~\ref{prop:CMEprop}, which are the main results in Sections~\ref{ssec:PPR} and \ref{ssec:CME} below. The estimate for $BD$ follows, since $BD=B(DB)B^{-1}$ on $B\overline{\Ran(D)} = \overline{\Ran(BD)}$, hence  $(I + itBD)^{-1} u = B (I + itDB)^{-1} B^{-1} u$ for all $u \in \overline{\Ran(BD)}$, whilst $(I + itBD)^{-1} u=u$ for all $u \in N(BD)$, and $t BD (I + (tBD)^2)^{-1}$ is a linear combination of such resolvents as in~\eqref{eq:RPQdef}.

If $n\geq 5$ and $\|V\|_{n/2}<\varepsilon_n$, then the proof follows similarly by using Propositions~\ref{prop:VLpPPR} and~\ref{prop:VLpCMEpropaux}, since the implicit constants therein depend only on $n$, $\kappa$ and $K$.
\end{proof}


\subsection{Coercivity and Cancellation}\label{ssec:CC}


To complete the proof of Theorem~\ref{thm:QE}, it remains to prove the reduction to a Carleson measure estimate in Proposition~\ref{prop:Principle Part Approximation} and the Carleson measure estimate in Proposition~\ref{prop:CMEprop}. These are proved in Sections~\ref{ssec:PPR} and \ref{ssec:CME} below. The main obstruction to obtaining these results in the presence of a singular potential is the fact that the coercivity estimate ($\|\nabla \otimes u\|_2 \lesssim \|Du\|_2$ for all $u\in\Ran(D)\cap\D(D)$) and cancellation  ($\int Du = 0$ for all $u\in\D(D)$ with compact support) which are instrumental to the proof in the case when $V\equiv 0$ (see, for instance, \cite[Theorem~3.4]{AAMc10-1} and \cite[(H7)--(H8) and Theorem~3.1(iii)]{AKMc06}) can fail to hold. In this section, we use Riesz transform bounds obtained by Shen~\cite{Shen95} and Auscher--Ben Ali~\cite{AB07}, as well as the Fefferman--Phong inequality~\eqref{eq:FePh}, to establish suitable substitutes for these properties.

We start by defining projections $\mathbb{P}_{\perp}$, $\mathbb{P}_{\parallel}$, $\mathbb{P}_\oo$ and $\mathbb{P}_{\perp\|}$ on $L^2 (\R^n; \C^{n+2})$ according to
\begin{equation}\label{eq:projdef}
    \mathbb{P}_{\perp} u \coloneqq \begin{bmatrix} u_{\perp} \\ 0 \\ 0 \end{bmatrix}, \quad
    \mathbb{P}_{\parallel} u \coloneqq \begin{bmatrix} 0 \\ u_{\parallel} \\ 0 \end{bmatrix}, \quad
    \mathbb{P}_\oo u \coloneqq \begin{bmatrix} 0 \\ 0 \\ u_\oo \end{bmatrix}
    \quad \text{ and } \quad
    \mathbb{P}_{\perp\|}u \coloneqq \begin{bmatrix} u_{\perp} \\ u_{\parallel} \\ 0 \end{bmatrix}
\end{equation}
for all $u = (u_{\perp}, u_{\parallel}, u_\oo) \in L^2 (\R^n) \oplus L^2 (\R^n ; \C^n) \oplus L^2 (\R^n)$. It is also convenient to introduce the notation $\nabla_\mu \otimes (u_1,\ldots,u_N) \coloneqq (\nabla_\mu u_1 , \ldots, \nabla_\mu u_N)$ for functions $u_1, \ldots, u_N$ in $\D(\nabla_\mu)$. We now use the Riesz transform bounds to obtain the following modified coercivity estimates for potentials in a reverse H\"{o}lder class.

\begin{lem}[Coercivity] \label{lem:Riesz Transform Estimates}
If $n\in\N$, $q\geq\max\{\tfrac{n}{2},2\}$ and $V \in B^{q}(\R^n)$ with $\llbracket V\rrbracket_{q}\leq \upsilon<\infty$, then
\[
\| DB (\mathbb{P}_\oo u)\|_2 \leq \|B\|_\infty \|D u\|_2
\quad\text{ and }\quad
\| \nabla_\mu \otimes (\mathbb{P}_{\perp\|}u)\|_2 \lesssim \| D u \|_2
\]
for all $u \in \Ran(D) \cap \D (D)$, where the implicit constants depend only on $n$ and $\upsilon$.
\end{lem}

\begin{proof}
Suppose that $u \in \Ran(D) \cap \D(D)$, so $u = (u_{\perp}, -\gradV v_{\perp})$ for some $v \in \D (D)$, in which case $u_{\perp} \in \Wv^{1,2}(\R^n)$ and $v_{\perp} \in \D (\gradV^* \gradV)$. We have $\D(\gradV^* \gradV) = \D(\Delta) \cap \D(V)$ by \cite[Corollary 1.3]{AB07}, where $\Delta\coloneqq -\nabla^*\nabla$, hence $\nabla v_{\perp} \in \D(\nabla^*)$ and $\V (\V v_{\perp})\in L^2(\R^n)$, so~\eqref{eq:Drep} shows that
\[
    D u = \begin{bmatrix} \gradV^* \gradV v_\perp \\ -\gradV u_{\perp}  \end{bmatrix}
    = \begin{bmatrix} (-\Delta + V) v_{\perp} \\ -\gradV u_{\perp} \end{bmatrix}.
\]
We also have $\mathbb{P}_\oo u = (0,\V v_\perp) \in \D(DB)$ with $\|DB \mathbb{P}_\oo u\|_2= \|(b V v_{\perp}, 0)\|_2 \leq \|B\|_\infty \|V v_{\perp}\|_2$ whilst (setting $\nabla^2 v_{\perp}:=\nabla\otimes(\nabla v_{\perp})$)
\[
    \|\nabla_\mu \otimes (\mathbb{P}_{\perp\|} u)\|_2
    =\|\nabla_\mu \otimes (u_\perp,-\nabla v_\perp)\|_2
    \leq \|\nabla_\mu u_{\perp}\|_2 + \|\nabla^2 v_{\perp}\|_2 + \|\V \nabla v_{\perp}\|_2.
\]
The Riesz transform bounds obtained in~\cite{AB07} and~\cite{Shen95} show that
\begin{equation}\label{eq:RTbdd}
    \|V v_{\perp}\|_2 + \|\nabla^2 v_{\perp}\|_2 + \|\V \nabla v_{\perp}\|_2 \lesssim \|(-\Delta + V) v_{\perp}\|_2,
\end{equation}
where the implicit constant depends only on $n$ and $\upsilon$. In all dimensions, the first two bounds follow from \cite[Theorem 1.1]{AB07} (and only require $V\in B^2(\R^n)$) whilst the final bound follows from the comments beneath \cite[Corollary~1.5]{AB07} (or \cite[Theorem~1.6]{BAII}). For $n\geq 3$, these were obtained earlier by Shen for $V\in B^{q}(\R^n)$ with $q\geq\max\{\tfrac{n}{2},2\}$ in \cite[Theorems 3.1, 0.3 and 4.13]{Shen95}.
\end{proof}

\begin{rem}\label{rem:Riesz}
Observe that $V\in B^{n/2}(\R^n)$ is only required in the above result to bound the mixed Riesz transform $\V \nabla(-\Delta + V)^{-1}$. The rest of the proof holds when $V\in B^2(\R^n)$.
\end{rem}

We can also obtain these modified coercivity estimates when $n\geq 5$ and $V \in L^{n/2} (\R^n)$ with sufficiently small norm, as then the Riesz transform bounds are automatic.

\begin{lem}[Coercivity] \label{lem:Riesz Transform Estimates with small norm}
If $n\geq 5$, then there exists $\varepsilon_n\in(0,1)$, depending only on $n$, such that whenever $V \in L^{n/2} (\R^n)$ with $\|V\|_{n/2}<\varepsilon_n$ it holds that
\[
\| DB (\mathbb{P}_\oo u)\|_2 \leq \|B\|_\infty \|D u\|_2
\quad\text{ and }\quad
\| \nabla_\mu \otimes (\mathbb{P}_{\perp\|}u)\|_2 \lesssim \| D u \|_2
\]
for all $u \in \Ran(D) \cap \D (D)$, where the implicit constant depends only on $n$.
\end{lem}

\begin{proof}
The proof follows as in Lemma~\ref{lem:Riesz Transform Estimates} once the Riesz transform bounds in~\eqref{eq:RTbdd} have been established in this case. If $f\in W^{2,2}(\R^n)$, then the Sobolev inequality~\eqref{eq:Sobolev} shows that
\[
    \| V f \|_2 + \| \V \nabla f \|_2 \leq \|V \|_{n/2} (\| f \|_{2^{**}}+\|\nabla f\|_{2^*}) \leq c_n \| V \|_{n/2} \| \nabla^2 f \|_2,
\]
where $2^{**}=2n/(n-4)$ and $c_n\in(0,\infty)$ depends only on $n$, so the classical Riesz transform bounds imply that
\[
    \| \nabla^2 f \|_2 \lesssim \| \Delta f \|_2 \leq \| (-\Delta + V) f \|_2 + \| V f \|_2 \leq \| (-\Delta + V) f \|_2 + c_n \| V \|_{n/2} \| \nabla^2 f \|_2.
\]
This shows that if $\| V \|_{n/2}$ is small enough, depending only on $n$, then $\| \nabla^2 f \|_2 \lesssim \| (-\Delta + V) f \|_2$, where the implicit constant depends only on $n$, so the required Riesz transform bounds hold.
\end{proof}

The collection $\Delta \coloneqq \bigcup_{k\in\Z} \Delta_{2^k}$, where $\Delta_t \coloneqq \{ 2^k (m + (0,1]^n) : m \in \Z^n \}$ for all $t\in(2^{k-1},2^k]$ and $k\in\Z$, is the usual dyadic decomposition of $\R^n$. For potentials $V$ in a reverse H\"{o}lder class, the natural length scale suggested by Sobolev's inequality~\eqref{eq:Sobolev}, estimate~\eqref{eq:VBn} and the Fefferman--Phong inequality~\eqref{eq:FePh}, is encapsulated by the collections
\begin{equation}\label{eq:homogcubes}
\Delta^V \coloneqq  \left\{Q \in \Delta : \ell(Q)^2 \dashint_Q V \leq 1\right\}
\quad\textrm{ and }\quad
\Delta_t^V \coloneqq \Delta^V \cap \Delta_t
\end{equation}
for all $t>0$. The elements of $\Delta^V$ are called \textit{homogeneous cubes} for the potential $V$, whilst those in $\Delta \setminus \Delta^V$ are called \textit{inhomogeneous cubes}. This terminology is chosen to reflect the fact that when $V\equiv 0$, the Schr\"{o}dinger equation~ \eqref{eqn:Schrodinger Equation} is homogeneous and $\Delta^V=\Delta$ consists of all dyadic cubes. Meanwhile, when $V\equiv 1$, the equation is inhomogeneous and $\Delta^V=\{Q \in \Delta : \ell(Q) \leq 1\}$ becomes restricted to the cubes at small length scales on which we expect the properties of solutions to be dominated by the behaviour of solutions to the homogeneous equation.

The first example of this dichotomy amongst dyadic cubes is given by the following analogue of the cancellation property for operators $D$ in the presence of a singular potential. The result is adapted from the work of Axelsson--Keith--McIntosh in
\cite[Lemma~6]{AKMc06-2}, which treats $V\equiv 1$ and is itself based on the case $V\equiv 0$ treated by the same authors in~\cite[Lemma 5.6]{AKMc06}, and by Auscher--Hofmann--Lacey--McIntosh--Tchamitchian in~\cite[Lemma~5.15]{AHLMcT02}. The key idea to obtaining the more general result below is to combine the Sobolev inequality~\eqref{eq:Sobolev} with estimate~\eqref{eq:VLn} to isolate the local $L^{n/2}$-norm on the potential, which can then be controlled as required.

\begin{lem}[Cancellation] \label{lem:Interpolation Lemma}
If $n\geq 3$ and $V\in L^{n/2}_{\Loc}(\R^n)$, then
\begin{equation} \label{eqn:ILEL}
    \left| \dashint_Q D u \right|^2 \lesssim \frac{1+\|\mathds{1}_QV\|_{n/2}}{\ell(Q)}\left( \dashint_Q |u|^2 \right)^{1/2} \left( \dashint_Q |D u|^2 \right)^{1/2}
\end{equation}
for all $Q \in \Delta$ and $u \in \D (D)$, where the implicit constant depends only on $n$. Moreover, if $n\geq 3$ and $V\in B^{n/2} (\R^n)$ with $\llbracket V\rrbracket_{n/2}\leq \upsilon<\infty$, then 
\begin{equation} \label{eqn:ILEB}
    \left| \dashint_Q D u \right|^2 \lesssim \frac{1}{\ell(Q)} \left( \dashint_Q |u|^2 \right)^{1/2} \left( \dashint_Q |D u|^2 \right)^{1/2}
\end{equation}
for all $Q \in \Delta^V$ and $u \in \D (D)$, where the implicit constant depends only on $n$ and $\upsilon$.
\end{lem}

\begin{proof}
Suppose that $u \in \D (D)$ with $t \coloneqq  \|\mathds{1}_Qu\|_2 / \|\mathds{1}_Q D u\|_2\in(0,\infty)$, since otherwise there is nothing to prove. If $t \geq \ell(Q)/2$, then by the Cauchy--Schwarz inequality
\[
\left| \dashint_Q D u \right|^2 \leq \left( \dashint_Q |D u|^2\right)
    \leq \frac{2}{\ell(Q)} \left( \dashint_Q |u|^2 \right)^{1/2}\left( \dashint_Q |D u|^2 \right)^{1/2}
\]
and \eqref{eqn:ILEL} follows.

Now suppose that $t < \ell(Q)/2$ and consider a test function $\eta \in \mathcal{C}_c^{\infty}(\R^n;\C^{n+2})$ with compact support contained in $Q$ such that $\eta(x) = 1$ whenever $\dist(x,\R^n \setminus Q) >t$ and $\|\nabla \otimes \eta \|_{\infty} \leq 2 t^{-1}$. Using that $D$ is self-adjoint, we obtain
\begin{align}\begin{split}\label{eq:coe1}
\left| \int_Q Du \right| 
&= \left|\int_{\R^n} \eta Du + \int_Q (1-\eta) Du \right|\\
&= \left|\langle D\eta, u\rangle + \int_Q (1-\eta) Du \right| \\ 
&\leq \int_Q (|\gradV^*(\eta_\|,\eta_\oo)| + |\gradV\eta_\perp|) |u| + \int_Q |1-\eta||D u|.
\end{split}\end{align}

To estimate these terms, we use $\|\nabla \otimes \eta \|_{\infty} \leq 2t^{-1}$ and $|\supp (\nabla \otimes \eta)| \lesssim \ell(Q)^{n-1} t$ to obtain 
\begin{align*}
    \int_Q |\nabla \otimes \eta | |u|
    \leq \left( \int_Q |\nabla \otimes\eta|^2 \right)^{1/2} \left( \int_Q |u|^2 \right)^{1/2} 
     \lesssim (\ell(Q)^{(n-1)}/t)^{1/2} \left( \int_Q |u|^2 \right)^{1/2},
\end{align*}
where the implicit constant depends only on $n$. Meanwhile, given that $(n/2)' = 2^*/2$, H\"older's inequality and the Sobolev inequality~\eqref{eq:Sobolev}, combined with the preceding estimate, show that
\begin{align*}
    \int_Q |V^{1/2}\eta| |u| & \leq \left( \int_Q |V| |\eta|^2 \right)^{1/2} \left( \int_Q |u|^2 \right)^{1/2} \\
    & \leq \left( \int_Q |V|^{n/2}\right)^{1/n} \left( \int_Q |\eta|^{2^*}\right)^{1/2^*} \left( \int_Q |u|^2 \right)^{1/2} \\
    & \lesssim \|\mathds{1}_QV\|_{n/2}^{1/2} \left( \int_Q | \nabla \otimes \eta |^2 \right)^{1/2} \left( \int_Q |u|^2 \right)^{1/2} \\
    & \lesssim (\ell(Q)^{(n-1)}/t)^{1/2} \|\mathds{1}_QV\|_{n/2}^{1/2} \left( \int_Q |u|^2 \right)^{1/2},
\end{align*}
where the implicit constants depend only on $n$. Next, since $|Q \cap \supp (1-\eta)| \lesssim \ell(Q)^{n-1} t$ and $\|1-\eta\|_\infty \leq 1$, we have
\[
    \int_Q |1-\eta||D u | \leq \left( \int_Q |1-\eta|^2 \right)^{1/2} \left( \int_Q |Du|^2 \right)^{1/2} \leq (\ell(Q)^{(n-1)}t)^{1/2} \left( \int_Q |Du|^2 \right)^{1/2}.
\]

We substitute the three preceding estimates into \eqref{eq:coe1} to obtain
\begin{align*}
    \left| \int_Q D u \right| 
	& \lesssim (\ell(Q)^{(n-1)}/t)^{1/2} (1 + \|\mathds{1}_QV\|_{n/2}^{1/2}) \left( \int_Q |u|^2 \right)^{1/2} + (\ell(Q)^{n-1}t)^{1/2} \left( \int_Q |D u|^2 \right)^{1/2} \\ 
	& \lesssim \left( \frac{|Q|}{{\ell(Q)}}\right)^{1/2} (1 + \|\mathds{1}_QV\|_{n/2}^{1/2}) \left( \int_Q |u|^2 \right)^{1/4} \left( \int_Q |D u|^2 \right)^{1/4},
\end{align*}
as $t = \|\mathds{1}_Qu\|_2 / \|\mathds{1}_Q D u\|_2$, so dividing by $|Q|$ and squaring gives~\eqref{eqn:ILEL}. Finally, if $V\in B^{n/2} (\R^n)$ with $\llbracket V\rrbracket_{n/2}\leq \upsilon<\infty$, and $Q \in \Delta^V$, then
\[
    \|\mathds{1}_QV\|_{n/2} = |Q|^{2/n} \left( \dashint_Q V^{n/2} \right)^{2/n} \lesssim \ell(Q)^2 \dashint_Q V \leq 1,
\]
where the implicit constant depends only on $n$ and $\upsilon$, so~\eqref{eqn:ILEB} follows from~\eqref{eqn:ILEL}.
\end{proof}


\subsection{Reduction to a Carleson Measure Estimate}\label{ssec:PPR}


The purpose of this section is to reduce the quadratic estimates in Theorem~\ref{thm:QE} to a Carleson measure estimate on the homogeneous cubes for the potential $V$ defined in \eqref{eq:homogcubes}. This reduction is the content of Proposition~\ref{prop:Principle Part Approximation} for $V$ in $B^{n/2} (\R^n)$. The version for $V$ in $L^{n/2} (\R^n)$ with sufficiently small norm is in Proposition~\ref{prop:VLpPPR}.

We follow the approach developed by Axelsson--Keith--McIntosh in~\cite{AKMc06}, which in turn is based on the original proof by Auscher--Hofmann--Lacey--McIntosh--Tchamitchian in~\cite{AHLMcT02}, but separate treatments will often be required for estimates on inhomogeneous cubes and homogeneous cubes when the potential is in the reverse H\"{o}lder class. This is because the cancellation in Lemma~\ref{lem:Interpolation Lemma} is only valid on homogeneous cubes. The cancellation is needed here to prove Lemma~\ref{lem:Schur Estimate} and is also an essential part of the proof of the Carleson measure estimate itself in the next section (see Lemma~\ref{lem:Properties of Test Functions} and Proposition~\ref{prop:Final reduction}).

We begin by introducing an averaging operator and obtaining off-diagonal bounds for certain resolvent-type operators. For each $t\in \R_+$, the dyadic averaging operator $A_t$ in $\mathcal{L}(L^2 (\R^n ; \C^{n+2}))$ is given at $x \in \R^n$ by 
\[
    A_t u(x) \coloneqq \dashint_{Q_{t,x}} u (y) \d y
\]
for all $u\in L^2 (\R^n ; \C^{n+2})$, where $Q_{t,x}$ denotes the unique dyadic cube in $\Delta_t$ that contains $x$, and the integral is understood component-wise. For each $t\in \R$, the resolvent-type operators $R_t^B$, $P_t^B$ and $Q_t^B$ in $\mathcal{L}(L^2 (\R^n ; \C^{n+2}))$ are given by
\begin{align}\begin{split}\label{eq:RPQdef}
    R_t^B & \coloneqq (I + i t DB)^{-1}, \\
    P_t^B & \coloneqq (I + (t DB)^2 )^{-1} = \tfrac{1}{2} (R_t^B + R_{-t}^B), \\
    Q_t^B & \coloneqq t DB (I + (t DB)^2 )^{-1} = t DB P_t^B = \tfrac{1}{2i} ( - R_t^B + R_{-t}^B).
\end{split}\end{align}
We abbreviate these operators as $R_t$, $P_t$ and $Q_t$ when $B=I$ is the identity operator. The resolvent bounds in \eqref{eq:DBresbdd} imply that $R_t^B$, $P_t^B$ and $Q_t^B$ are all uniformly bounded in $\mathcal{L}(L^2 (\R^n ; \C^{n+2}))$ with respect to $t\in\R$. They also satisfy the following off-diagonal bounds, where $\langle x \rangle \coloneqq 1 + |x|$ when $x\in\R$, since ultimately the commutators $[D,\eta I]u\coloneqq D(\eta u)-\eta Du$ of $D$ with scalar-valued cut-off functions $\eta$ are not influenced by inhomogeneous contributions from the potential $V$.

\begin{prop}[Off-Diagonal Estimates] \label{prop:Off diagonal estimates}
If $n\in\N$, $V\in L^1_{\Loc}(\R^n)$ and $M \in \N$, then there exists $C_M\in(0,\infty)$, depending only on $\kappa$, $K$ and $M$, such that whenever $U_t \in \{R_t^B,P_t^B,Q_t^B\}$ there is the off-diagonal estimate
\[
    \norm{\mathds{1}_E U_t\mathds{1}_F u}{{2}}{} \leq C_M \left\langle \frac{\dist(E,F)}{t} \right\rangle^{-M} \norm{\mathds{1}_Fu}{2}{}
\]
for all $u \in L^2 (\R^n ; \C^{n+2})$, Borel sets $E,F \subseteq \R^n$ and $t\in\R$.
\end{prop}

\begin{proof}
It suffices to prove that the domain of $D$ is invariant under multiplication by functions $\eta$ in $\test(\R^n)$ and that the commutator $[D,\eta I]u\coloneqq  D (\eta u) - \eta D u$ satisfies $|[D,\eta I]u| \leq |\nabla \eta | |u|$ for all $u\in\D(D)$. The required off-diagonal bounds then follow as in~\cite[Proposition 5.1]{AAMc10-1}.

To this end, suppose that $\eta \in \mathcal{C}_c^{\infty} (\R^n)$ and $u\in\D(D)$. We then have $(u_\|,u_\oo) \in \D (\gradV^*)$, so $\eta \gradV^* (u_\|,u_\oo) - \nabla \eta \cdot u_\|$ is in $L^2 (\R^n)$ with
\begin{align*}
    \langle \eta \gradV^* (u_\|,u_\oo) - \nabla \eta \cdot u_{\parallel} , \varphi \rangle
    & = \langle (u_\|,u_\oo) , \gradV (\conj{\eta} \varphi) \rangle - \langle u_{\parallel} , \varphi \nabla \conj{\eta} \rangle \\
    & = \langle u_{\parallel} , \conj{\eta} \nabla \varphi + \varphi \nabla \conj{\eta}  \rangle + \langle u_{\mu} , \V \conj{\eta} \varphi \rangle - \langle u_{\parallel} , \varphi \nabla \conj{\eta} \rangle \\
    & = \langle \eta (u_\|,u_\oo) , \gradV \varphi \rangle
\end{align*}
for all $\varphi \in \mathcal{C}_c^{\infty} (\R^n)$, and thus $\eta (u_\|,u_\oo) \in \D (\gradV^*)$ with $\gradV^* (\eta (u_\|,u_\oo)) = \eta \gradV^* (u_\|,u_\oo) - \nabla \eta \cdot u_{\parallel}$. Meanwhile, the product rule for weak derivatives ensures that $\eta  u_{\perp} \in \D(\gradV)$, so $\eta u \in \D(D)$ with
\[
[D,\eta I]u =  D (\eta u) - \eta D u
    = -\begin{bmatrix} \gradV^* (\eta (u_\|,u_\oo))-\eta \gradV^* (u_\|,u_\oo) \\ \nabla (\eta u_{\perp}) - \eta \nabla u_{\perp} \\ 0 \end{bmatrix}
    = \begin{bmatrix} \nabla \eta \cdot u_{\parallel} \\ -(\nabla \eta)u_{\perp} \\ 0 \end{bmatrix}
\]
and $|([D,\eta I]u)(x)| \leq |\nabla \eta(x)| |u(x)|$ for almost every $x\in\R^n$, as required.
\end{proof}

The off-diagonal bounds imply that $Q_t^B$ in $\mathcal{L}(L^2 (\R^n ; \C^{n+2}))$ has an extension, which we continue to denote as $Q_t^B$, that maps $L^\infty(\R^n ; \C^{n+2})$ into  $L^2_{\Loc}(\R^n ; \C^{n+2})$. This follows the reasoning of Axelsson--Keith--McIntosh in~\cite[p.~478]{AKMc06}. For each $(t,x)\in \R^{n+1}_+$, the principal part $\gamma_t(x) \in \mathcal{L}(\C^{n+2})$ of the operator $Q_t^B$ is then defined by
\[
\gamma_t (x) w \coloneqq ( Q_t^B (w\mathds{1}_{\mathbb{R}^n})) (x)
\]
for all $x\in\R^n$ and $w \in \C^{n+2}$, where $w\mathds{1}_{\mathbb{R}^n}$ in $L^\infty(\R^n ; \C^{n+2})$ denotes the constant function equal to $w$ on $\R^n$. For each $t\in\R_+$, the multiplication operator $\gamma_t: L^2(\R^n ; \C^{n+2}) \to L^1_{\Loc}(\R^n ; \C^{n+2})$ is given by $(\gamma_tu)(x)\coloneqq \gamma_t(x)(u(x))$ for all $x\in\R^n$ and $u\in L^2(\R^n ; \C^{n+2})$. We also adapt the ideas introduced by Bailey in~\cite{Bailey18} by considering the component operators
\[
\widetilde{\gamma}_t(x)w\coloneqq \gamma_t(x)(\mathbb{P}_{\perp\|}w)
\quad\text{and}\quad \widetilde{\gamma}_tu\coloneqq \gamma_t\mathbb{P}_{\perp\|}u
\]
for all $(t,x)\in \R^{n+1}_+$, $w \in \C^{n+2}$ and $u\in L^2(\R^n ; \C^{n+2})$, where $\mathbb{P}_{\perp\|}$ is the projection from~\eqref{eq:projdef}. The following properties are a corollary of the off-diagonal bounds established above.

\begin{lem} \label{lem:properties of gamma-tilde}
If $n\in\N$ and $V\in L^1_{\Loc}(\R^n)$, then $\dashint_Q |\gamma_t(x)|_{\mathcal{L} (\C^{n+2})}^2 \d x \lesssim 1$ for all $Q \in \Delta_t$ and $t\in\R_+$, and also $\sup_{t > 0} \| \gamma_t A_t \| \lesssim 1$, where the implicit constants depend only on $\kappa$ and $K$.
\end{lem}

\begin{proof}
This follows~\cite[Corollary~5.3]{AKMc06} using the off-diagonal estimates in Proposition~\ref{prop:Off diagonal estimates}.
\end{proof}

We can now state the main result of this section.

\begin{prop} \label{prop:Principle Part Approximation}
If $n\geq 3$, $q\geq\max\{\tfrac{n}{2},2\}$ and $V \in B^{q}(\R^n)$ with $\llbracket V\rrbracket_{q}\leq \upsilon<\infty$, then
\begin{equation} \label{eqn:Carleson Measure Term}
\int_0^{\infty} \|Q_t^B u \|_2^2 \frac{\d t}{t}
\lesssim \|u\|_2^2
+ \int_0^{\infty} \int_{\Omega_t^V} |A_t u(x)|^2 |\widetilde{\gamma}_t(x)|_{\mathcal{L}(\C^{n+2})}^2 \d x\frac{\d t}{t} \end{equation}
for all $u \in L^2(\R^n;\C^{n+2})$, where $\Omega_t^V \coloneqq \bigcup \{Q: Q \in \Delta_t^V\} = \bigcup\{Q : Q\in\Delta_t \text{ and }  \ell(Q)^2\dashint_Q V \leq 1\}$ and the implicit constant depends only on $n$, $\kappa$, $K$ and $\upsilon$.
\end{prop}

\begin{proof}
We will rely on the topological decomposition $L^2 (\R^n ; \C^{n+2} ) = \overline{\Ran(D)} \oplus \Nul(DB)$ from \eqref{eq:Hodge}. If $u \in \Nul(DB)$, then $Q_t^B u=t P_t^B DB u=0$ so there is nothing to prove.

Now suppose that $u \in \Ran(D)$, so $u = D v$ for some $v \in \D (D)$. The quadratic estimates for the self-adjoint operator $D$ then imply that
\begin{align}\begin{split}\label{eq:P_new Estimate}
    \int_0^{\infty} \| Q_t^B (I - P_t) u\|_2^2 \frac{\d t}{t}
    & = \int_0^{\infty} \| t Q_t^B D (I - P_t) v \|_2^2 \frac{\d t}{t^3} \\ 
    & \lesssim \int_0^{\infty} \| (I - P_t) v \|_2^2 \frac{\d t}{t^3} \\
    & = \int_0^{\infty} \|Q_t D v \|_2^2 \frac{\d t}{t} \\
    & \lesssim \|u\|_2^2,
\end{split}\end{align}
where we used the identities $t Q_t^B D = (t Q_t^B DB)B^{-1}\mathbb{P}_{\overline{\Ran(BD)}}$ and $t Q_t^B DB=I-P_t^B$, the uniform bounds for $P_t^B$, the accretivity of $B$ and the bounded projection $\mathbb{P}_{\overline{\Ran(BD)}}$ onto $\overline{\Ran(BD)}$ from~\eqref{eq:Hodge} to obtain the first estimate. Next, since $Q_t^B=DBP_t^B$, the coercivity in Lemma~\ref{lem:Riesz Transform Estimates} shows that
\begin{equation}\label{eq:P_3 Estimate}
    \int_0^{\infty} \| Q_t^B \mathbb{P}_\oo P_t u \|_2^2 \frac{\d t}{t}  
     \lesssim \int_0^{\infty} \|t DB \mathbb{P}_\oo D P_t v \|_2^2 \frac{\d t}{t} 
     \lesssim \int_0^{\infty} \| Q_t D v \|_2^2 \frac{\d t}{t} 
     \lesssim \|u\|_2^2.
\end{equation}
We will see in Lemma~\ref{lem:Large Cubes Estimate} below that the $\mathbb{P}_{\perp\|}$-component on inhomogeneous cubes satisfies
\begin{equation}\label{eq:3.9}
\int_0^{\infty} \|\mathds{1}_{\R^n\setminus \Omega_t^V} Q_t^B \mathbb{P}_{\perp\|} P_t u\|_2^2 \frac{\d t}{t} \lesssim \|u\|_2^2,
\end{equation}
whilst Lemmas~\ref{lem:Poincare Estimate} and~\ref{lem:Schur Estimate} will show that the $\mathbb{P}_{\perp\|}$-component on homogeneous cubes satisfies
\begin{equation}\label{eq:3.1011}
   \int_0^{\infty} \|\mathds{1}_{\Omega_t^V}Q_t^B\mathbb{P}_{\perp\|} P_t u \|_2^2 \frac{\d t}{t}
   \lesssim \|u\|_2^2 + \int_0^{\infty} \|\mathds{1}_{\Omega_t^V} \widetilde{\gamma}_t A_t u\|_2^2 \frac{\d t}{t}.
\end{equation}
These estimates together prove~\eqref{eqn:Carleson Measure Term} when $u \in \Ran(D)$.

Finally, suppose that $u\in\overline{\Ran (D)}$. For each $\delta \in (0,1)$, choose $v_\delta \in \Ran (D)$ so that $\|u - v_\delta\|_2^2 < \delta^3$. The uniform bounds for $Q_t^B$ and $\widetilde{\gamma}_t A_t$ (see Lemma~\ref{lem:properties of gamma-tilde}) then imply that
\[
    \int_{\delta}^{1/\delta} (\| Q_t^B (u - v_\delta) \|_2^2 + \|\widetilde{\gamma}_t A_t (u -  v_\delta)\|_2^2) \frac{\d t}{t} \lesssim \int_{\delta}^{1/\delta} \| u - v_\delta \|_2^2 \frac{\d t}{t} < \delta,
\]
so \eqref{eqn:Carleson Measure Term} applied to $v_\delta$ shows that 
\[
    \int_{\delta}^{1/\delta} \| Q_t^B u \|_2^2 \frac{\d t}{t} \lesssim \|u\|_2^2 + \int_0^{\infty} \|\mathds{1}_{\Omega_t^V} \widetilde{\gamma}_t A_t u\|_2^2 \frac{\d t}{t} + \delta
\]
and the result follows from the monotone convergence theorem.
\end{proof}

The next three lemmas contain the estimates used to prove~\eqref{eq:3.9} and~\eqref{eq:3.1011} above. The first result is analogous to the small scale reduction obtained by Axelsson--Keith--McIntosh in \cite[(4.1)--(4.2)]{AKMc06-2} and Morris in \cite[Proposition~5.1]{Morris12} for the case $V\equiv 1$, since then $\R^n\setminus \Omega_t^V$ is empty when $t\in(0,1]$, whilst $\R^n\setminus \Omega_t^V=\R^n$ when $t>1$. This reduction to homogeneous cubes is what ultimately allows us to apply the cancellation from~\eqref{eqn:ILEB} to prove Lemma~\ref{lem:Schur Estimate} below. In contrast to the case when $V\equiv 1$, the quadratic estimate on inhomogeneous cubes is not immediate, and instead requires combining off-diagonal bounds with the Fefferman--Phong inequality~\eqref{eq:FePh} and the adapted coercivity for $\mathbb{P}_{\perp\|}$-components from Lemma~\ref{lem:Riesz Transform Estimates}.

\begin{lem} \label{lem:Large Cubes Estimate}
If $n\geq 3$, $q\geq\max\{\tfrac{n}{2},2\}$ and $V \in B^{q}(\R^n)$ with $\llbracket V\rrbracket_{q}\leq \upsilon<\infty$, then
\[
\int_0^{\infty} \|\mathds{1}_{\R^n\setminus \Omega_t^V} Q_t^B \mathbb{P}_{\perp\|} P_t u\|_2^2 \frac{\d t}{t} \lesssim \|u\|_2^2,
\]
for all $u \in \Ran(D)$, where $\Omega_t^V \coloneqq \bigcup \{Q: Q \in \Delta_t^V\} = \bigcup\{Q : Q\in\Delta_t \text{ and }  \ell(Q)^2\dashint_Q V \leq 1\}$ and the implicit constant depends only on $n$, $\kappa$, $K$ and $\upsilon$.
\end{lem}

\begin{proof}
Let $f = \mathbb{P}_{\perp\|} P_t u$, consider $M \in \N$ to be chosen later, and suppose that $Q \in \Delta_t \setminus \Delta_t^V$. The off-diagonal estimates in Proposition~\ref{prop:Off diagonal estimates} and the Cauchy--Schwarz inequality show that
\begin{align*}
\| \mathds{1}_Q Q_t^B \mathbb{P}_{\perp\|} P_t u \|_2^2
    &  \lesssim   \left(\|\mathds{1}_{2 Q} f \|_2 + \sum_{k=1}^{\infty} \| \mathds{1}_Q Q_t^B \mathds{1}_{2^{k+1}Q\setminus 2^k Q} \| \|\mathds{1}_{2^{k+1} Q} f \|_2 \right)^2 \lesssim \sum_{k=1}^{\infty} 2^{-k M} \|\mathds{1}_{2^kQ} f \|_2^2.
\end{align*}

We next claim that $\|\mathds{1}_{2^k Q} f \|_2^2 \lesssim t^2 2^{nk} \| \mathds{1}_{2^k Q} \gradV f\|_2^2$. This follows from the Fefferman--Phong inequality in Proposition~\ref{prop:Fefferman-Phong} in the case $p=2$, since if $ \ell(2^k Q)^2 \dashint_{2^k Q} V \geq 1$, then 
\[
\|\mathds{1}_{2^k Q} f \|_2^2 \lesssim 2^{2k}\ell(Q)^2 \| \mathds{1}_{2^k Q} \gradV f\|_2^2,
\]
whilst if $ \ell(2^k Q)^2\dashint_{2^k Q} V < 1$, then because $ \ell(Q)^2 \dashint_{Q} V > 1$, we have
\begin{align*}
\|\mathds{1}_{2^k Q} f \|_2^2 
&\leq \left( \ell(Q)^2 \dashint_Q V \right) \| \mathds{1}_{2^k Q} f \|_2^2 \\
&\lesssim 2^{(n-2)k} \left( \ell(2^k Q)^2 \dashint_{2^k Q} V \right) \|\mathds{1}_{2^k Q} f\|_2^2 \\
&\lesssim 2^{nk} \ell(Q)^2 \| \mathds{1}_{2^k Q} \gradV f\|_2^2.
\end{align*}

We now choose $M > 2n$ and combine the above claim with the preceding estimate and note that $\{2^kQ:Q\in\Delta_t\}$ is a locally finite covering of $\R^n$ to obtain
\begin{align*}
\|\mathds{1}_{\R^n\setminus \Omega_t^V} Q_t^B \mathbb{P}_{\perp\|} P_t u\|_2^2 
&= \sum_{Q \in \Delta_t \setminus \Delta_t^V} \|\mathds{1}_Q Q_t^B \mathbb{P}_{\perp\|} P_t u\|_2^2 \\
&\lesssim \sum_{Q \in \Delta_t \setminus \Delta_t^V} \sum_{k=1}^\infty 2^{-k M} \|\mathds{1}_{2^k Q} f \|_2^2 \\
& \lesssim t^2 \sum_{k=1}^\infty 2^{-k(M-n)} \sum_{Q \in \Delta_t} \|\mathds{1}_{2^k Q} \gradV f\|_2^2  \\
& \lesssim t^2 \sum_{k=1}^{\infty} 2^{-k(M-2n)} \| \gradV f \|_2^2  \\
& \lesssim t^2  \| \gradV f \|_2^2.
\end{align*}
The coercivity in Lemma~\ref{lem:Riesz Transform Estimates} and quadratic estimates for $D$ then imply that
\begin{align*}
    \int_0^{\infty} \|\mathds{1}_{\R^n\setminus \Omega_t^V} Q_t^B \mathbb{P}_{\perp\|} P_t u\|_2^2 \frac{\d t}{t} 
     \lesssim \int_0^{\infty} t^2 \| \gradV \mathbb{P}_{\perp\|} P_t u \|_2^2 \frac{\d t}{t} 
     \lesssim \int_0^{\infty} \| t D P_t u \|_2^2 \frac{\d t}{t} 
     \lesssim \| u \|_2^2,
\end{align*}
as required.
\end{proof}

The next lemma is based on the result by Axelsson--Keith--McIntosh in \cite[Proposition~5.5]{AKMc06}, but after applying the Poincar\'{e} inequality we find that the adapted coercivity in Lemma~\ref{lem:Riesz Transform Estimates} is only compatible with the $\mathbb{P}_{\perp\|}$-component. The $\mathbb{P}_{\mu}$-component was therefore treated separately at \eqref{eq:P_3 Estimate} in the proof of Proposition~\ref{prop:Principle Part Approximation}, avoiding the Poincar\'{e} inequality and using only the coercivity relevant to that component.

\begin{lem} \label{lem:Poincare Estimate}
If $n\geq 3$, $q\geq\max\{\tfrac{n}{2},2\}$ and $V \in B^{q}(\R^n)$ with $\llbracket V\rrbracket_{q}\leq \upsilon<\infty$, then
\[
    \int_0^{\infty} \|(Q_t^B - \gamma_t A_t) \mathbb{P}_{\perp\|} P_t u \|_2^2 \frac{\d t}{t} \lesssim \| u \|_2^2,
\]
for all $u \in \Ran(D)$, where the implicit constant depends only on $n$, $\kappa$, $K$ and $\upsilon$.
\end{lem}

\begin{proof}
Let $f = \mathbb{P}_{\perp\|} P_t u$ and $f_Q \coloneqq \dashint_{Q} f$ for $Q\in\Delta$. The off-diagonal estimates in Proposition~\ref{prop:Off diagonal estimates} with $M>2n+2$, the Poincar\'e inequality and the coercivity in Lemma~\ref{lem:Riesz Transform Estimates} show that
\begin{align*}
    \sum_{Q \in \Delta_t} \| \mathds{1}_Q (Q_t^B - \gamma_t A_t) f \|_2^2 
    &\leq \sum_{Q \in \Delta_t}\! \left(\|\mathds{1}_Q Q_t^B \mathds{1}_{2Q}(f - f_Q)\|_2 + \sum_{k=1}^{\infty} \| \mathds{1}_Q Q_t^B \mathds{1}_{2^{k+1}Q\setminus 2^kQ}(f - f_Q)\|_2\right)^2 \\
    & \lesssim  \sum_{Q \in \Delta_t} \sum_{k=1}^{\infty} 2^{-kM} \|\mathds{1}_{2^k Q} (f - f_{2^kQ} + f_{2^kQ} - f_Q)\|_2^2 \\
    & \lesssim \sum_{Q \in \Delta_t} \sum_{k=1}^{\infty} 2^{-k(M-n)} \ell(2^k Q)^2 \|\mathds{1}_{2^k Q} \nabla f \|_2^2 \\
    & \lesssim t^2 \sum_{k=1}^{\infty} 2^{-k(M-2n-2)} \|\nabla f\|_2^2 \\
    & \lesssim t^2 \|\nabla \mathbb{P}_{\perp\|} P_t u \|_2^2 \\
    & \lesssim t^2 \| D P_t u \|_2^2,
\end{align*}
where we used that $\{2^kQ:Q\in\Delta_t\}$ is a locally finite covering of $\R^n$ in the fourth line. Hence,
\[
    \int_0^{\infty} \|(Q_t^B - \gamma_t A_t) \mathbb{P}_{\perp\|} P_t u \|_2^2 \frac{\d t}{t} \lesssim \int_0^{\infty} \| Q_t u \|_2^2 \frac{\d t}{t} \lesssim \| u \|_2^2,
\]
as required.
\end{proof}

The final lemma is adapted from the result by Axelsson--Keith--McIntosh in \cite[Proposition~ 5]{AKMc06-2} in the case $V\equiv 1$, and since this is where the cancellation from Lemma~\ref{lem:Interpolation Lemma} is required, the result only concerns homogeneous cubes.

\begin{lem} \label{lem:Schur Estimate}
If $n\geq 3$, $q\geq\max\{\tfrac{n}{2},2\}$ and $V \in B^{q}(\R^n)$ with $\llbracket V\rrbracket_{q}\leq \upsilon<\infty$, then
\[
    \int_0^{\infty} \| \mathds{1}_{\Omega_t^V} \gamma_t A_t \mathbb{P}_{\perp\|} (P_t - I) u \|_2^2 \frac{\d t}{t} \lesssim \| u \|_2^2,
\]
for all $u \in \Ran(D)$, where $\Omega_t^V \coloneqq \bigcup \{Q: Q \in \Delta_t^V\} = \bigcup\{Q : Q\in\Delta_t \text{ and }  \ell(Q)^2\dashint_Q V \leq 1\}$ and the implicit constant depends only on $n$, $\kappa$, $K$ and $\upsilon$.
\end{lem}

\begin{proof}
We first prove the bound
\begin{equation} \label{Bound for Schur Estimate}
   \| \mathds{1}_{\Omega_t^V} A_t \mathbb{P}_{\perp\|} (P_t - I) Q_s v \|_2^2 = \sum_{Q \in \Delta_t^V} \|\mathds{1}_Q A_t \mathbb{P}_{\perp\|} (P_t - I) v \|_2^2 \lesssim \min \left\{ \frac{s}{t} , \frac{t}{s} \right\}  \|v\|_2^2
\end{equation}
for all $v\in L^2 (\R^n ; \C^{n+2})$. If $0<t \leq s<\infty$, then since $(P_t - I) Q_s = \frac{t}{s} Q_t (P_s-I)$, we have
\[
   \|A_t \mathbb{P}_{\perp\|} (P_t - I) Q_s v \|_2^2 \lesssim  \| \tfrac{t}{s} Q_t (P_s-I) v \|_2^2 \lesssim \left(\frac{t}{s}\right)^2 \|v\|_2^2 \leq \frac{t}{s} \|v\|_2^2.
\]
If $0<s \leq t<\infty$, then since $P_t Q_s = \frac{s}{t} Q_t P_s$, we have
\begin{align*}
    \sum_{Q \in \Delta_t^V}\| \mathds{1}_Q A_t \mathbb{P}_{\perp\|} (P_t - I) Q_s v \|_2^2 
    &\lesssim \| P_t Q_s v \|_2^2 + \sum_{Q \in \Delta_t^V} \| \mathds{1}_Q A_t \mathbb{P}_{\perp\|} Q_s v \|_2^2 \\
    &\lesssim \frac{s}{t} \|v\|_2^2 + \sum_{Q \in \Delta_t^V} \| \mathds{1}_Q A_t \mathbb{P}_{\perp\|} Q_s v \|_2^2.
\end{align*}
The cancellation from~\eqref{eqn:ILEB} in Lemma~\ref{lem:Interpolation Lemma}, for cubes $Q$ in $\Delta^V$, now shows that
\begin{align*}
    \sum_{Q \in \Delta_t^V} \| \mathds{1}_Q A_t \mathbb{P}_{\perp\|} Q_s v \|_2^2
    & = \sum_{Q \in \Delta_t^V} \int_Q \left| \dashint_Q \mathbb{P}_{\perp\|} Q_s v \right|^2 \\
    & \leq \sum_{Q \in \Delta_t^V} s^2 |Q| \left| \dashint_Q D P_s v \right|^2 \\
    & \lesssim \sum_{Q \in \Delta_t} s^2 \frac{|Q|}{\ell(Q)} \left( \dashint_Q |D P_s v|^2 \right)^{1/2} \left( \dashint_Q |P_s v|^2 \right)^{1/2} \\
    & \leq \frac{s}{t} \| Q_s u \|_2 \|P_t v \|_2 \\
    & \lesssim \frac{s}{t} \|v\|_2^2.
\end{align*}
These estimates together prove~\eqref{Bound for Schur Estimate}.

We now use a Schur-type estimate to complete the proof. Observe that, since $A_t$ is defined component-wise on $L^2 (\R^n ; \C^{n+2})$, whilst both $\gamma_t$ and $\widetilde{\gamma}_t$ are defined pointwise as multiplication operators on $\mathbb{R}^n$, and $A_t^2=A_t$ maps (locally) to constants in $\mathbb{C}^{n+2}$ on each cube in $\Delta_t$, we have
\[
\mathds{1}_Q \gamma_t (A_t \mathbb{P}_{\perp\|}v)
=\mathds{1}_Q \gamma_t (\mathbb{P}_{\perp\|} A_t \mathbb{P}_{\perp\|}v)
=\mathds{1}_Q \widetilde{\gamma}_t (A_t \mathbb{P}_{\perp\|}v)
=\mathds{1}_Q \widetilde{\gamma}_t (\mathds{1}_Q A_t \mathbb{P}_{\perp\|}v)
=\mathds{1}_Q \widetilde{\gamma}_t (A_t \mathds{1}_Q A_t \mathbb{P}_{\perp\|}v)
\]
whenever $Q\in \Delta_t$, $t\in\R_+$ and $v\in L^2 (\R^n ; \C^{n+2})$. Using the notation $m(s,t) \coloneqq \min \left\{ \frac{s}{t} , \frac{t}{s} \right\}^{1/2}$, we now combine this observation with the uniform bounds for $\widetilde{\gamma}_tA_t$ from Lemma~\ref{lem:properties of gamma-tilde}, the Calder\'{o}n--McIntosh reproducing formula $\int_0^{\infty} Q_s^2 u \frac{\d s}{s}\eqsim u$ for $u\in\Ran(D)$ (cf. \cite[Section~7]{Mc88}, \cite[Theorem~5.2.6]{Haase05} or \cite[Theorem~3.3.9]{EgertPhD}), Minkowski's inequality and \eqref{Bound for Schur Estimate} to obtain
\begin{align*}
    \int_0^{\infty} \| \mathds{1}_{\Omega_t^V} \gamma_t A_t \mathbb{P}_{\perp\|} (P_t - I) u \|_2^2 \frac{\d t}{t}
    & = \int_0^{\infty}  \sum_{Q \in \Delta_t^V} \| \mathds{1}_Q \widetilde{\gamma}_t A_t\mathds{1}_Q A_t \mathbb{P}_{\perp\|} (P_t - I) u \|_2^2 \frac{\d t}{t} \\
    &\lesssim \int_0^{\infty} \sum_{Q \in \Delta_t^V} \|\mathds{1}_Q A_t \mathbb{P}_{\perp\|} (P_t - I) u \|_2^2 \frac{\d t}{t} \\
    &= \int_0^{\infty} \|\mathds{1}_{\Omega_t^V} A_t \mathbb{P}_{\perp\|} (P_t - I) u \|_2^2 \frac{\d t}{t} \\
    & \lesssim \int_0^{\infty} \left\|\mathds{1}_{\Omega_t^V} A_t \mathbb{P}_{\perp\|} (P_t - I) \left( \int_0^{\infty} Q_s^2 u \frac{\d s}{s} \right) \right\|_2^2 \frac{\d t}{t} \\
    & \lesssim \int_0^{\infty} \left( \int_0^{\infty} \| \mathds{1}_{\Omega_t^V} A_t \mathbb{P}_{\perp\|} (P_t - I) Q_s (Q_s u)\|_2 \frac{\d s}{s} \right)^2 \frac{\d t}{t} \\
    & \lesssim \int_0^{\infty} \left( \int_0^{\infty}  m(s,t) \| Q_s u \|_2 \frac{\d s}{s} \right)^2 \frac{\d t}{t} \\
    & \lesssim \int_0^{\infty} \sup_{t>0}  \left(\int_0^{\infty} m(s,t) \frac{\d s}{s} \right) \left( \int_0^{\infty} m(s,t) \| Q_s u \|_2^2 \frac{\d s}{s} \right) \frac{\d t}{t} \\
    & \lesssim \sup_{s>0} \left( \int_0^{\infty} m(s,t) \frac{\d t}{t} \right) \int_0^{\infty} \| Q_s u \|_2^2 \frac{\d s}{s} \\
    & \lesssim \| u \|_2^2.
\end{align*}
This completes the proof.
\end{proof}

This concludes the proof of the reduction to a Carleson measure estimate in Proposition~\ref{prop:Principle Part Approximation}. If $V \in L^{n/2} (\R^n)$ with sufficiently small norm, then the coercivity in Lemma~\ref{lem:Riesz Transform Estimates with small norm} and cancellation in~\eqref{eqn:ILEL} can be applied instead on all dyadic cubes to obtain the following analogue.

\begin{prop}\label{prop:VLpPPR}
If $n\geq 5$ and $\|V\|_{n/2}<\varepsilon_n$, where $\varepsilon_n\in(0,1)$ is from Lemma~\ref{lem:Riesz Transform Estimates with small norm}, then
\[
\int_0^{\infty} \|Q_t^B u \|_2^2 \frac{\d t}{t}
\lesssim \|u\|_2^2
+ \int_0^{\infty} \int_{\R^n} |A_t u(x)|^2 |\widetilde{\gamma}_t(x)|_{\mathcal{L}(\C^{n+2})}^2 \d x\frac{\d t}{t}
\]
for all $u \in L^2(\R^n;\C^{n+2})$, where the implicit constant depends only on $n$, $\kappa$ and $K$.
\end{prop}

\begin{proof}
We modify the proof of Proposition~\ref{prop:Principle Part Approximation} by first using the coercivity from Lemma~\ref{lem:Riesz Transform Estimates with small norm}, instead of that from Lemma~\ref{lem:Riesz Transform Estimates}, to obtain the estimates in~\eqref{eq:P_3 Estimate} and Lemma~\ref{lem:Poincare Estimate}. Next, the cancellation from~\eqref{eqn:ILEL} in Lemma~\ref{lem:Interpolation Lemma}, which holds for all dyadic cubes $Q$ in $\Delta$ and not just those in $\Delta^V$, is used to replace \eqref{eqn:ILEB} in the proof of Lemma~\ref{lem:Schur Estimate} to obtain
\begin{equation}\label{eq:SchRep}
\int_0^{\infty} \|\gamma_t A_t \mathbb{P}_{\perp\|} (P_t - I) u \|_2^2 \frac{\d t}{t} \lesssim \| u \|_2^2,
\end{equation}
where the implicit constants in all these estimates depend only $n$, $\kappa$ and $K$, as $\varepsilon_n$ in Lemma~\ref{lem:Riesz Transform Estimates with small norm} depends only on $n$. There is no need for the estimate in Lemma~\ref{lem:Large Cubes Estimate} and the result follows.
\end{proof}


\subsection{Carleson Measure Estimate}\label{ssec:CME}


To complete the proof of the main quadratic estimates in Theorem~\ref{thm:QE}, it remains to bound the final term that appears in the reduction to a Carleson measure estimate in~ \eqref{eqn:Carleson Measure Term}. This will follow by establishing a Carleson measure estimate on the homogeneous cubes for the potential $V$, as defined in \eqref{eq:homogcubes}, adapting the approach of Axelsson--Keith--McIntosh in \cite[Section 5.3]{AKMc06}.

A measure $\mu$ on $\R_{+}^{n+1}$ will be called a \textit{homogeneous Carleson measure} for the potential $V$ if 
\begin{equation}\label{eq:defcmV}
\norm{\mu}{{\mathcal{C}_V}}{} \coloneqq \sup_{Q \in \Delta^V} \frac{1}{|Q|} \mu (\mathcal{C}(Q)) < \infty,
\end{equation}
where $\mathcal{C} (Q) \coloneqq Q \times (0,\ell(Q)]$ is the \textit{Carleson box} over the cube $Q$. The proof of the following proposition, which has been adapted from the work of Morris in~\cite[Theorem~4.2]{Morris12}, also holds with $|\widetilde{\gamma}_t (x)|_{\mathcal{L}(\C^{n+2})}^2 \d x\frac{\d t}{t}$ replaced by $d\mu(x,t)$ for any such homogeneous Carleson measure.

\begin{prop}\label{prop:CMEprop}
If $n\geq 3$, $q\geq\max\{\tfrac{n}{2},2\}$ and $V \in B^{q}(\R^n)$ with $\llbracket V\rrbracket_{q}\leq \upsilon<\infty$, then
\[
\int_0^{\infty} \int_{\Omega_t^V} |A_t u(x)|^2 |\widetilde{\gamma}_t(x)|_{\mathcal{L}(\C^{n+2})}^2 \d x\frac{\d t}{t} 
\leq \|u\|_2^2
\]
for all $u \in L^2(\R^n;\C^{n+2})$, where $\Omega_t^V \coloneqq \bigcup \{Q: Q \in \Delta_t^V\} = \bigcup\{Q : Q\in\Delta_t \text{ and }  \ell(Q)^2\dashint_Q V \leq 1\}$ and the implicit constant depends only on $n$, $\kappa$, $K$ and $\upsilon$.
\end{prop}

\begin{proof}
We consider the measure $\mu$ with $\d\mu(x,t)\coloneqq  |\widetilde{\gamma}_t (x)|_{\mathcal{L}(\C^{n+2})}^2 \d x\frac{\d t}{t}$ for all $(t,x)\in\R^{n+1}_+$. It is ultimately proved in Proposition~\ref{prop:CMEpropaux} below that $\mu(\mathcal{C}(Q)) \lesssim |Q|$ for all $Q\in\Delta^V$, where the implicit constant depends only on $n$, $\kappa$, $K$ and $\upsilon$, so $\mu$ is a homogeneous Carleson measure for the potential $V$, as defined by \eqref{eq:defcmV}.

Suppose that $u \in L^2(\R^n;\C^{n+2})$. For each $k\in\Z$, let $\{ Q_{\alpha}^k : \alpha \in \mathbb{I}_k \}$ denote an enumeration of the cubes in $\Delta_{2^k}^V$ for some index set $\mathbb{I}_k \subseteq \N$. We will use the notation $u_{\alpha,k} \coloneqq  \dashint_{Q_{\alpha}^k} |u(y)| \d y$ and $\mu_{\alpha,k} \coloneqq  \mu (Q_{\alpha}^k \times (2^{k-1} , 2^k])$ for all $\alpha\in\mathbb{I}_k$. Using that $\Delta_t^V = \Delta_{2^k}^V$ when $2^{k-1} < t \leq 2^k$, we obtain
\begin{align}\begin{split}\label{eq:CMP}
    \int_0^{\infty} \int_{\Omega_t^V} |A_t u(x) |^2 \d\mu(x,t) 
    & = \sum_{k\in\Z} \sum_{Q \in \Delta_{2^{k}}^V} \int_{2^{k-1}}^{2^{k}} \int_Q \left|\dashint_Q u(y) \d y \right|^2 \d\mu(x,t) \\
    & \leq \sum_{k\in\Z} \sum_{\alpha \in \mathbb{I}_k}  u_{\alpha,k}^2 \mu_{\alpha,k} \\
    & = \sum_{k\in\Z} \sum_{\alpha \in \mathbb{I}_k}  \mu_{\alpha,k} \int_0^{u_{\alpha,k}} 2r \, \d r \\
    & = \int_0^{\infty} \sum_{k\in\Z} \sum_{\alpha \in \mathbb{I}_k}  \mu_{\alpha,k} \mathds{1}_{\{ |u_{\alpha,k}| >r \}} (r) 2r \, \d r.
\end{split}\end{align}

For each $r>0$, observe that $u_{\alpha,k} \leq r$ whenever $(\|u\|_2/r)^{2/n}\lesssim \ell(Q_\alpha^k)$, where the implicit constant depends only on $n$, since $u_{\alpha,k}\leq |Q_\alpha^k|^{-1/2}\|u\|_2$. This allows us to use a stopping-time argument to construct the collection of maximal dyadic cubes $Q_{\alpha}^k$ in $\Delta^V$ such that $u_{\alpha,k} > r$. Let $\{R_j^{\, r}:j \in \mathbb{J}_r\}$ denote an enumeration of the cubes in this collection for some index set $\mathbb{J}_r \subseteq \mathbb{N}$.

Next, we define $M_*^V u (x) \coloneqq \sup \{ \dashint_Q u : x\in Q \in \Delta^V\}$ and claim that
\begin{equation}\label{eq:setclaim}
\bigcup\nolimits_{j\in\mathbb{J}_r} R_j^{\, r} = \{ x \in \R^n : M_*^V |u| (x) > r \}  
\end{equation}
for all $r>0$. Indeed, if $x \in  R_j^{\, r}$ for some $j \in \mathbb{J}_r$, then 
$r < \dashint_{R_j^{\, r}} |u| \leq M_*^V |u| (x)$. Conversely, if $x \in \R^n$ and $M_*^V |u| (x) > r$, then there exists $Q \in \Delta^V$ such that $x \in Q$ and $\dashint_{Q} |u|>r$, so $Q \subseteq R_j^{\, r}$ for some $j \in \mathbb{J}_r$ by maximality. This proves the equality claimed in \eqref{eq:setclaim}. 
 
Now observe that if $Q_{\alpha}^k \in \Delta^V$ and $u_{\alpha,k}>r$ for some $k\in\mathbb{Z}$, $\alpha\in\mathbb{I}_k$ and $r>0$, then  $Q_{\alpha}^k \subseteq R_j^{\, r}$ for some $j \in \mathbb{J}_r$ by maximality. This shows that \eqref{eq:CMP} is bound by
\begin{align*}
\int_0^{\infty} \sum_{j\in\mathbb{J}_r} \sum_{\substack{R \in \Delta^V;\\ R \subseteq R_j^{\, r}}}  \mu (R \times (\tfrac{1}{2}\ell(R), \ell(R)]) 2r \d r 
    & \leq \int_0^{\infty} 2r \sum_{j\in\mathbb{J}_r} \mu(\mathcal{C} (R_j^{\, r})) \d r \\
    & \lesssim \int_0^{\infty} 2r \sum_{j\in\mathbb{J}_r} |R_j^{\, r}|\d r \\
    & = \int_0^{\infty} 2r | \{ x \in \R^n : M_*^V |u| (x) > r \} | \d r \\
    & = \norm{M_*^V |u|}{2}{2}
    \leq \norm{M_* |u|}{2}{2} 
    \lesssim \norm{u}{2}{2},
\end{align*}
where we used that $\mu$ is a homogeneous Carleson measure for $V$ in the second line, identity \eqref{eq:setclaim} in the third line, and the  comparison $M_*^V |u|(x)\leq M_*|u|(x)$ for all $x\in\R^n$ with the Hardy--Littlewood maximal operator $M_*$ in the final line. This completes the proof.
\end{proof}

A stopping-time construction will ultimately be used in Proposition~\ref{prop:Final reduction} to extract suitable collections of dyadic subcubes from a given homogeneous cube for the potential $V$. Unlike the cases when $V\equiv0$ or $V\equiv1$ treated by Axelsson--Keith--McIntosh in \cite{AKMc06} and \cite{AKMc06-2}, however, a dyadic subcube of a homogeneous cube may not itself be a homogeneous cube for the potential $V$. Therefore, we need to modify the stopping-time construction to ensure that only homogeneous subcubes are selected. This is achieved by using the following lemma, which provides a uniform bound on the number of times a homogeneous cube needs to be subdivided to ensure that all of the subsequent dyadic subdivisions are homogeneous cubes. The proof is motivated by the work of Shen in~\cite[Lemma 1.2]{Shen95}.

\begin{lem} \label{lem:Small Cube Depth Bound}
If $n\geq 3$ and $V\in B^{n/2} (\R^n)$ with $\llbracket V\rrbracket_{n/2}\leq \upsilon<\infty$, then there exists ${c_0} \in (0,1]$, depending only on $n$ and $\upsilon$, such that ${c_0} \ell(Q) < \ell(\widetilde{Q})$ when $Q \in \Delta^V$, $\widetilde{Q} \in \Delta \setminus \Delta^V$ and $\widetilde{Q} \subseteq Q$.
\end{lem}

\begin{proof}
Suppose that $Q \in \Delta^V$, $\widetilde{Q} \in \Delta \setminus \Delta^V$ and $\widetilde{Q} \subset Q$. A well-known self-improvement property of reverse H\"{o}lder weights (see, for instance, \cite[Chapter V, Section 3]{SteinBig}) ensures that there exists $p > n/2$ such that $V \in B^p (\R^n)$, where $p$ and $\llbracket V \rrbracket_p$ depend only on $n$ and $\llbracket V\rrbracket_{n/2}$. Therefore, since $\widetilde{Q} \subseteq Q$, we have 
\[
    \dashint_{\widetilde{Q}} V 
    \leq \bigg( \frac{|Q|}{|\widetilde{Q}|} \dashint_Q V^p \bigg)^{1/p}
    \leq  \llbracket V \rrbracket_p \bigg( \frac{\ell(Q)}{\ell(\widetilde{Q})}\bigg)^{n/p} \dashint_Q V.
\]
Moreover, since $Q \in \Delta^V$ and $\widetilde{Q} \in \Delta \setminus \Delta^V$, we have
\[
    1 < \ell(\widetilde{Q})^2 \dashint_{\widetilde{Q}}V 
    \leq \llbracket V \rrbracket_p \bigg( \frac{\ell(Q)}{\ell(\widetilde{Q})} \bigg)^{(n/p) - 2} \ell(Q)^2 \dashint_Q V 
    \leq \llbracket V \rrbracket_p \bigg( \frac{\ell(\widetilde{Q})}{\ell(Q)} \bigg)^{2-(n/p)}.
\]
Now, since $p > n/2$, we have $2-(n/p) > 0$, hence $\ell(Q) < \llbracket V \rrbracket_p^{p/(2p-n)} \ell(\widetilde{Q})$. This completes the proof with ${c_0}=\llbracket V \rrbracket_p^{-p/(2p-n)}$, since $\llbracket V \rrbracket_p\geq 1$.
\end{proof}

Next, for each $\xi\in\C^{n+2}$ and $Q\in\Delta$, we define the function $\xi_Q \coloneqq \mathds{1}_{2Q} \xi$ in $L^2(\R^n;\C^{n+2})$. Moreover, for each $\varepsilon > 0$, we define the test function
\begin{equation}\label{eq:test}
    f_{Q, \varepsilon}^{\xi} \coloneqq P_{\varepsilon  \ell(Q)}^B \xi_Q = (I + (\varepsilon  \ell(Q))^2 (DB)^2 )^{-1} (\xi_Q)
\end{equation}
by analogy with the test functions introduced by Auscher--Ros\'{e}n--Rule in~\cite[Section 3.6]{ARR15}. The following  properties of these test functions are essential to the proof of the Carleson measure estimate. The proof of property (3), which is based on \cite[Lemma 3.16]{ARR15}, requires the cancellation in \eqref{eqn:ILEB} and is therefore restricted to homogeneous cubes.

\begin{lem} \label{lem:Properties of Test Functions}
If $n\geq 3$, $q\geq\max\{\tfrac{n}{2},2\}$ and $V \in B^{q}(\R^n)$ with $\llbracket V\rrbracket_{q}\leq \upsilon<\infty$, then whenever $\xi\in\C^{n+2}$ and $\varepsilon>0$ the test functions $f_{Q,\varepsilon}^{\xi}$ in~\eqref{eq:test} satisfy the following bounds:
\begin{enumerate}
    \item $ \| f_{Q,\varepsilon}^{\xi }\|_2 \lesssim |Q|^{1/2} $ for all $Q \in \Delta$;
    \item $\|\varepsilon  \ell(Q) DB f_{Q,\varepsilon}^{\xi} \|_2 \lesssim |Q|^{1/2}$ for all $Q \in \Delta$;
    \item $\left| \displaystyle{\dashint_Q} f_{Q,\varepsilon}^{\xi} - \xi \right| \lesssim \varepsilon^{1/2}$ for all $Q \in \Delta^V$,
\end{enumerate}
where the implicit constants depend only on $n$, $\kappa$, $K$ and $\upsilon$.
\end{lem}

\begin{proof}
The uniform resolvent bounds for $DB$ show that $\| f_{Q,\varepsilon}^{\xi }\|_2 = \| P_{\varepsilon  \ell(Q)}^B \xi_Q \|_2 \lesssim \|\xi_Q\|_2 \leq |Q|^{1/2}$ and $\|\varepsilon  \ell(Q) DB f_{Q,\varepsilon}^{\xi} \|_2 = \| Q_{\varepsilon  \ell(Q)}^B \xi_Q \|_2  \lesssim \| \xi_Q \|_2 \leq |Q|^{1/2}$ for all $Q \in \Delta$. If $Q\in\Delta^V$, then the cancellation from~\eqref{eqn:ILEB} in Lemma~\ref{lem:Interpolation Lemma} shows that
\begin{align*}
    \left| \dashint_Q f_{Q,\varepsilon}^{\xi} - \xi \right|^2 
    & = \left| \dashint_Q (f_{Q,\varepsilon}^{\xi} - \xi_Q) \right|^2 \\
    & = \left| \dashint_Q (\varepsilon  \ell(Q))^2 (DB)^2 P_{\varepsilon  \ell(Q)}^B \xi_Q \right|^2 \\
    & \lesssim \frac{(\varepsilon  \ell(Q))^4}{ \ell(Q)} \left( \dashint_Q |(DB)^2 P_{\varepsilon  \ell(Q)}^B \xi_Q |^2 \right)^{1/2} \left( \dashint_Q | B DB P_{\varepsilon  \ell(Q)}^B \xi_Q |^2 \right)^{1/2} \\
    & \leq \frac{\varepsilon \|B\|_{\infty}}{|Q|} \left( \int_Q | (P_{\varepsilon  \ell(Q)}^B - I) \xi_Q |^2 \right)^{1/2} \left( \int_Q | Q_{\varepsilon  \ell(Q)}^B \xi_Q|^2 \right)^{1/2} \\ 
    & \lesssim \frac{\varepsilon}{|Q|} \|\xi_Q\|_2^2 
    \leq \varepsilon.
\end{align*}
This completes the proof.
\end{proof}

We now use Lemma \ref{lem:Properties of Test Functions} to choose $\varepsilon_0 > 0$, depending only on $n$, $\kappa$, $K$ and $\upsilon$, such that $\left| \dashint_Q f_{Q,\varepsilon_0}^{\xi} - \xi\right|^2 \leq \tfrac{1}{2}$ for all $Q\in\Delta^V$, and then define the associated test functions
\begin{equation}\label{eq:test2}
    f_Q^{\xi} \coloneqq f_{Q , \varepsilon_0}^{\xi} = P_{\varepsilon_0 \ell(Q)}^B \xi_Q
\end{equation}
for all $Q\in\Delta$. The polarisation identity now shows that 
\[\textstyle
    4 \Re \left\langle \xi , \dashint_Q f_Q^{\xi} \right\rangle 
    = \left| \xi + \dashint_Q f_Q^{\xi} \right|^2 - \left|\xi - \dashint_Q f_Q^{\xi} \right|^2
    \geq |\xi|^2 + 2\Re \left\langle \xi , \dashint_Q f_Q^{\xi} \right\rangle + \left| \dashint_Q f_Q^{\xi} \right|^2 - \tfrac{1}{2},
\]
so when $|\xi|=1$, it follows that
\begin{equation} \label{eqn:Lower Bound for Stopping Time}\textstyle
    \Re \left\langle \xi , \dashint_Q f_Q^{\xi} \right\rangle \geq \tfrac{1}{4}
\end{equation}
for all $Q\in\Delta^V$. We now show how to arrange the required stopping-time construction for these test functions so that only homogeneous cubes are selected.

\begin{lem} \label{lem:Stopping Time}
If $n\geq 3$, $q\geq\max\{\tfrac{n}{2},2\}$ and $V \in B^{q}(\R^n)$ with $\llbracket V\rrbracket_{q}\leq \upsilon<\infty$, then there exist $\alpha\in(0,1)$ and $\tau \in (0,1)$, both depending only on $n$, $\kappa$, $K$ and $\upsilon$, such that whenever $\xi\in\C^{n+2}$, $|\xi|=1$ and $Q \in \Delta^V$, the test function $f_{Q}^{\xi}$ in \eqref{eq:test2} has the following property: There exists a countable collection $\mathcal{B}=\{Q_k\}$ of pairwise disjoint cubes $Q_k$ in $\Delta^V$ such that $Q_k \subseteq Q$, $\ell(Q_k) \leq {c_0} \ell(Q)$ and $| Q \setminus \bigcup_{k} Q_k| > \tau |Q|$, and for which
\begin{equation} \label{eqn:conditions for good cube in stopping time}
    \dashint_{S} |f_{Q}^{\xi} | \leq \alpha \quad \text{and} \quad \Re \left\langle \xi , \dashint_{S} f_{Q}^{\xi} \right\rangle \geq \frac{1}{\alpha}
\end{equation}
whenever $S\in\Delta$, $S\subseteq Q$, $\ell(S) \leq {c_0} \ell(Q)$ and either $Q_k \subset S$ for some $Q_k \in \mathcal{B}$ or $(\bigcup_k Q_k) \cap S = \emptyset$, where ${c_0}$ is from Lemma~\ref{lem:Small Cube Depth Bound}.
\end{lem}

\begin{proof}
Suppose that $Q \in \Delta^V$, consider $\alpha \in (0,1)$ to be chosen later, and let $\{Q_k^\prime\}$ denote an enumeration of the set $\mathcal{B}_1$ of maximal dyadic subcubes of $Q$ such that
\[
\ell(Q_k^\prime) \leq {c_0} \ell(Q) \quad\text{ and }\quad \dashint_{Q_k^\prime} |f_{Q}^{\xi} | > \frac{1}{\alpha}.
\]
This means that if $S\in\Delta$, $S\subseteq Q$ and $\ell(S) \leq {c_0} \ell(Q)$ with either $Q_k^\prime \subset S$ for some $Q_k^\prime \in \mathcal{B}_1$ or $(\bigcup_k Q_k^\prime) \cap S = \emptyset$, then $\dashint_{S} |f_{Q}^{\xi} | \leq \frac{1}{\alpha}$. Lemma~\ref{lem:Small Cube Depth Bound} ensures that $\mathcal{B}_1 \subseteq \Delta^V$, since $Q \in \Delta^V$ whilst $Q_k^\prime \subseteq Q$ and $\ell(Q_k^\prime) \leq {c_0} \ell(Q)$ for all $Q_k^\prime \in \mathcal{B}_1$. Using (1) in Lemma~\ref{lem:Properties of Test Functions}, we also have
\begin{equation}\label{eq:B1}
        \left| \bigcup \mathcal{B}_1 \right|
    = \sum_{Q_k^\prime \in \mathcal{B}_1} |Q_k^\prime|
    < \alpha \sum_{Q_k^\prime \in \mathcal{B}_1} \int_{Q_k^\prime} |f_{Q}^{\xi} |
    \leq \alpha \int_Q |f_{Q}^{\xi} |
    \leq \alpha C |Q|,
\end{equation}
where $C\in[1,\infty)$ depends only on $n$, $\kappa$, $K$ and $\upsilon$. Next, let $\{Q_k^{\prime\prime}\}$ denote an enumeration of the set $\mathcal{B}_2$ of maximal dyadic subcubes of $Q$ such that
\[
\ell(Q_k^{\prime\prime}) \leq {c_0} \ell(Q) \quad\text{ and }\quad \Re \left\langle \xi , \dashint_{Q_k^{\prime\prime}} f_{Q}^{\xi} \right\rangle < \alpha.
\]
This means that if $S\in\Delta$, $S\subseteq Q$ and $\ell(S) \leq {c_0} \ell(Q)$ with either $Q_k^{\prime\prime} \subset S$ for some $Q_k^{\prime\prime} \in \mathcal{B}_2$ or $(\bigcup_k Q_k^{\prime\prime}) \cap S = \emptyset$, then $\Re \langle \xi , \dashint_{S} f_{Q}^{\xi}\rangle \geq \alpha$. Lemma~\ref{lem:Small Cube Depth Bound} again ensures that $\mathcal{B}_2 \subseteq \Delta^V$ and we combine~\eqref{eqn:Lower Bound for Stopping Time} with (1) in Lemma~\ref{lem:Properties of Test Functions} to obtain  
\begin{align*}
    \tfrac{1}{4} 
    & \leq \Re \left\langle \xi , \dashint_Q f_Q^{\xi} \right\rangle \\
    & = \sum_{Q_k^{\prime\prime} \in \mathcal{B}_2} \frac{|Q_k^{\prime\prime}|}{|Q|} \Re \left\langle \xi , \dashint_{Q_k^{\prime\prime}} f_Q^{\xi} \right\rangle + \Re \left\langle \xi , \frac{1}{|Q|}     \int_{Q \setminus \bigcup \mathcal{B}_2} f_Q^{\xi} \right\rangle \\
    & < \alpha + \frac{1}{|Q|} \int_{Q \setminus \bigcup \mathcal{B}_2} |f_Q^{\xi}| \\
    & \leq \alpha + C \left( \frac{|Q \setminus \bigcup \mathcal{B}_2|}{|Q|}\right)^{1/2},
\end{align*}
where $C\in[1,\infty)$ is from~\eqref{eq:B1}.

We now suppose that $\alpha \in (0,\tfrac{1}{4} )$, so that $(\tfrac{1}{4}  - \alpha)^2 |Q| \leq C^2|Q \setminus \bigcup \mathcal{B}_2|$, and let $\{ Q_k \}$ denote an enumeration of the set $\mathcal{B}$ of maximal cubes in $\mathcal{B}_1 \cup \mathcal{B}_2$, so by~\eqref{eq:B1} we have
\[ 
    | Q \setminus \bigcup\nolimits_{k} Q_k| \geq |Q \setminus \bigcup \mathcal{B}_2 | - | \bigcup \mathcal{B}_1| \geq \left[\left(\frac{1-4\alpha}{4C}\right)^2 - \alpha C\right] |Q|.
\]
This allows us to choose $\alpha \in (0,\tfrac{1}{4})$ sufficiently small so that $| Q \setminus \bigcup\nolimits_{k} Q_k| \geq \tau |Q|$ for some $\tau\in(0,1)$, where both $\alpha$ and $\tau$ depend only on $n$, $\kappa$, $K$ and $\upsilon$. The properties of $\mathcal{B}_1$ and $\mathcal{B}_2$ ensure that $\mathcal{B} \subseteq \Delta^V$ and~\eqref{eqn:conditions for good cube in stopping time} holds. This completes the proof.
\end{proof}

The next step is to prove the Carleson measure estimate on a sawtooth $\mathcal{C} (Q) \setminus \bigcup_k \mathcal{C} (Q_k)$ given by the preceding lemma when $\widetilde{\gamma}_t(x)$ is  restricted to conical regions in $\mathcal{L} (\C^{n+2})$. This requires  estimates used to prove the reduction to a Carleson measure estimate in Section~\ref{ssec:PPR} for which it was necessary to give separate treatment to the $\mathbb{P}_{\perp\|}$- and $\mathbb{P}_{\mu}$-components from \eqref{eq:projdef}. Motivated by the work of Bailey in \cite[Section 4]{Bailey18}, this is achieved here by working in the space 
\[
    \mathcal{L}_{\perp\parallel} \coloneqq \{ \nu \in \mathcal{L} (\C^{n+2}) \setminus \{ 0 \} : \nu \mathbb{P}_{\perp\|} = \nu \}.
\]
We now fix $\sigma \coloneqq  {1}/{2\alpha^2}$, where $\alpha\in(0,1)$ depending only on $n$, $\kappa$, $K$ and $\upsilon$ is from Lemma~\ref{lem:Stopping Time}. For each $\nu\in \mathcal{L}_{\perp\parallel}$ with $|\nu| = 1$, we consider the cone $K_{\nu, \sigma}$ in $\mathcal{L}_{\perp\parallel}$ of aperture $\sigma$ given by
\[
    K_{\nu, \sigma} \coloneqq \{ \mu \in \mathcal{L}_{\perp\parallel} : | \mu / |\mu|_{\mathcal{L}(\C^{n+2})} - \nu |_{\mathcal{L}(\C^{n+2})} \leq \sigma \}.
\]

\begin{prop} \label{prop:Final reduction}
If $n\geq 3$, $q\geq\max\{\tfrac{n}{2},2\}$ and $V \in B^{q}(\R^n)$ with $\llbracket V\rrbracket_{q}\leq \upsilon<\infty$, then there exists $\tau\in(0,1)$ and $T\in(0,\infty)$, depending only on $n$, $\kappa$, $K$ and $\upsilon$, such that whenever $Q \in \Delta^V$ and $\nu \in \mathcal{L}_{\perp\parallel}$ with $|\nu| = 1$, there is a countable collection $\{ Q_k \}$ in $\Delta^V$ of pairwise disjoint subcubes of $Q$ such that
\begin{equation}\label{eq:coneE*est}
    |E_{Q,\nu}| > \tau |Q|
    \quad\text{ and }\quad
    \iint_{\substack{(x,t) \in E_{Q,\nu}^* \\ \widetilde{\gamma}_t (x) \in K_{\nu,\sigma}}} |\widetilde{\gamma}_t (x)|_{\mathcal{L}(\C^{n+2})}^2 \frac{\d x \d t}{t} \leq T |Q|,
\end{equation}
where $E_{Q,\nu} \coloneqq Q \setminus \bigcup_k Q_k$ and $E_{Q,\nu}^* \coloneqq \mathcal{C} (Q) \setminus \bigcup_k \mathcal{C} (Q_k)$.
\end{prop}

\begin{proof}
Suppose that $Q \in \Delta^V$ and $\nu\in \mathcal{L}_{\perp\parallel}$ with $|\nu|_{\mathcal{L}(\C^{n+2})} = 1$. We choose $\xi$ and $\zeta$ in $\C^{n+2}$ such that $|\xi| = |\zeta| = 1$, $\xi = \nu^* (\zeta)$ and $\mathbb{P}_{\perp\|} \xi = \xi$. This is possible as $|\nu|_{\mathcal{L}(\C^{n+2})} = 1$, so there exists $w \in \C^{n+2}$ such that $|w| = 1$ and $|\nu (w)| = 1$, thus we set $\xi \coloneqq \nu^* ( \nu (w))$ and $\zeta \coloneqq \nu (w)$. We have $|\xi| \leq |\nu^*|_{\mathcal{L}(\C^{n+2})} |\nu (w)| = 1 = \langle \nu (w) , \nu (w) \rangle = \langle w , \xi \rangle \leq |\xi|$, so $|\xi|=1$. Also, if $z \in \C^{n+2}$, then $\langle \xi , z \rangle = \langle \nu (w) , \nu (z) \rangle = \langle \nu (w) , \nu (\mathbb{P}_{\perp\|} z) \rangle = \langle \xi , \mathbb{P}_{\perp\|} z \rangle = \langle \mathbb{P}_{\perp\|} \xi , z \rangle$, as $\nu \in \mathcal{L}_{\perp\parallel}$, so $\xi =  \mathbb{P}_{\perp\|} \xi$.

Now consider the corresponding test function $f_{Q}^{\xi}$ as defined in~\eqref{eq:test2}. We use Lemma~\ref{lem:Stopping Time} to obtain a countable collection $\mathcal{B}=\{Q_k\}$ of pairwise disjoint cubes $Q_k$ in $\Delta^V$ such that $Q_k \subseteq Q$, $\ell(Q_k) \leq {c_0} \ell(Q)$, $| Q \setminus \bigcup_{k} Q_k| > \tau |Q|$ and \eqref{eqn:conditions for good cube in stopping time} holds with constants $\alpha\in(0,1)$ and $\tau\in(0,1)$ that depend only on $n$, $\kappa$, $K$ and $\upsilon$.

To prove~\eqref{eq:coneE*est}, first let $c_0$ denote the constant from Lemma~\ref{lem:Small Cube Depth Bound} and use the local uniform bounds for $\widetilde{\gamma}_t$ in Lemma~\ref{lem:properties of gamma-tilde} to obtain 
\begin{align}\label{eq:clcCM}
\int_{{c_0} \ell(Q)}^{\ell(Q)} \int_Q |\widetilde{\gamma}_t (x)|^2 \frac{\d x \d t}{t} 
\lesssim \frac{|Q|}{{c_0} \ell(Q)} \int_{{c_0} \ell(Q)}^{\ell(Q)} \d t
\lesssim |Q|,
\end{align}
where the implicit constants depend only on $n$, $\kappa$, $K$ and $\upsilon$, as $c_0$ depends only on $n$ and $\upsilon$. Next, observe that if $t\in(0,c_0 \ell(Q)]$ and $(t,x) \in E_{Q,\nu}^*\coloneqq \mathcal{C} (Q) \setminus \bigcup_k \mathcal{C} (Q_k)$, then there exists $S \in \Delta_t^V$ such that $x\in S \subseteq Q$, $\ell(S) \leq {c_0} \ell(Q)$ and either $Q_k \subset S$ for some $Q_k \in \mathcal{B}$ or $(\bigcup_k Q_k) \cap S = \emptyset$, so by Lemma \ref{lem:Stopping Time} we have
\[
    |A_t f_Q^{\xi}(x)| \leq \dashint_S |f_Q^{\xi}| \leq \frac{1}{\alpha}
    \quad\text{and}\quad
    | \nu (A_t f_Q^{\xi} (x)) | 
    \geq \Re \langle \zeta , \nu (A_t f_Q^{\xi} (x)) \rangle 
    = \Re \left\langle \xi , \dashint_S f_Q^{\xi} \right\rangle \geq \alpha,
\]
hence when also $\gamma_t (x) \in K_{\nu,\sigma}$ we have
\[
    \left| \frac{\widetilde{\gamma}_t (x)(A_t f_Q^{\xi} (x))}{|\widetilde{\gamma}_t (x)|_{\mathcal{L}(\C^{n+2})}} \right| 
    \geq | \nu (A_t f_Q^{\xi} (x)) | -  \left| \frac{\widetilde{\gamma}_t (x)}{|\widetilde{\gamma}_t (x)|_{\mathcal{L}(\C^{n+2})}} - \nu \right|_{\mathcal{L}(\C^{n+2})} |A_t f_Q^{\xi}(x)|
    \geq \alpha - \frac{\sigma}{\alpha}
    = \tfrac{1}{2}\alpha,
\]
since $\sigma \coloneqq {1}/{2\alpha^2}$. These considerations show that
\begin{align}\begin{split}\label{eq:c0}
    \iint_{\substack{(x,t) \in E_{Q,\nu}^* \\ \widetilde{\gamma}_t (x) \in K_{\nu,\sigma}}} |\widetilde{\gamma}_t (x)|^2 \frac{\d x \d t}{t} 
	& \lesssim \int_0^{c_0\ell(Q)} \int_Q | \widetilde{\gamma}_t (x) (A_t f_Q^{\xi} (x))|^2 \frac{\d x \d t}{t} + |Q|,
\end{split}\end{align}
where the implicit constants depend only on $n$, $\kappa$, $K$ and $\upsilon$.

It remains to control the contribution from the test functions in \eqref{eq:c0}. To this end, observe that the uniform bounds for the test functions from (2) in Lemma~\ref{lem:Properties of Test Functions} show that
\begin{equation}\label{eq:c1}
 \int_0^{c_0\ell(Q)} \int_Q |Q_t^B f_Q^{\xi}|^2 \frac{\d x\d t}{t}
    \lesssim \int_0^{\ell(Q)} \|t DB f_Q^{\xi} \|_2^2 \frac{\d t}{t}
    \lesssim \frac{|Q|}{\varepsilon_0^2\ell(Q)^2} \int_0^{\ell(Q)} t \d t
    \lesssim |Q|,
\end{equation}
where $\varepsilon_0$, and thus the implicit constants, depend only on $n$, $\kappa$, $K$ and $\upsilon$. Next, write
\begin{align}
\begin{split} \label{eq:Introducing Pt}
\int_0^{c_0\ell(Q)} \int_Q |(Q_t^B &- \widetilde{\gamma}_t A_t) (f_Q^{\xi} - \xi_Q)|^2 \frac{\d x \d t}{t} \\
    &\lesssim \int_0^{c_0\ell(Q)} \int_Q | (Q_t^B \mathbb{P}_{\perp\|} P_t - \widetilde{\gamma}_t A_t) (f_Q^{\xi} - \xi_Q)|^2 \frac{\d x \d t}{t} \\
    &\quad +\int_0^\infty \| Q_t^B \mathbb{P}_{\oo} P_t (f_Q^{\xi} - \xi_Q)\|_2^2 \frac{\d t}{t} 
    +\int_0^\infty \| Q_t^B (I - P_t) (f_Q^{\xi} - \xi_Q)\|_2^2 \frac{\d t}{t}.
\end{split}
\end{align}
It follows from~\eqref{eq:P_new Estimate} and~\eqref {eq:P_3 Estimate}, since $f_Q^{\xi} - \xi_Q = (P_{\varepsilon_0 \ell(Q)}^B -I) \xi_Q = (\varepsilon_0 \ell(Q) DB)^2 P_{\varepsilon_0 \ell(Q)} \xi_Q$ is in $\Ran(D)$, that the final two terms in \eqref{eq:Introducing Pt} are each bounded by $\| f_Q^{\xi} - \xi_Q\|_2^2$. Meanwhile, Lemmas~\ref{lem:Poincare Estimate} and~\ref{lem:Schur Estimate} show that the remaining term in \eqref{eq:Introducing Pt} is bounded by
\begin{align*}
& \int_0^{c_0\ell(Q)} \int_Q \left(|(Q_t^B \mathbb{P}_{\perp\|} - \gamma_t \mathbb{P}_{\perp\|} A_t) P_t (f_Q^{\xi} - \xi_Q)|^2
+ | \gamma_t \mathbb{P}_{\perp\|} A_t (P_t-I) (f_Q^{\xi} - \xi_Q)|^2 \right) \frac{\d x \d t}{t} \\
&\quad \lesssim \int_0^{\infty} \|\mathds{1}_{\Omega_t^V} (Q_t^B - \gamma_t A_t) \mathbb{P}_{\perp\|} P_t (f_Q^{\xi} - \xi_Q) \|_2^2 \frac{\d t}{t} + \int_0^{\infty} \|\mathds{1}_{\Omega_t^V} \gamma_t A_t \mathbb{P}_{\perp\|} (P_t-I) (f_Q^{\xi} - \xi_Q) \|_2^2 \frac{\d t}{t} \\
&\quad \lesssim  \| f_Q^{\xi} - \xi_Q\|_2^2,
\end{align*}
where the second line relies on the fact that $Q\in \Delta^V$, thus Lemma~\ref{lem:Small Cube Depth Bound} implies that $\widetilde{Q}\in\Delta^V$ whenever $\widetilde{Q}\subseteq Q$ with $\widetilde{Q}\in\Delta_t$ and $t\in(0,c_0\ell(Q)]$. Together, these estimates show that
\begin{equation}\label{eq:c2}
\int_0^{c_0\ell(Q)} \int_Q |(Q_t^B - \widetilde{\gamma}_t A_t) (f_Q^{\xi} - \xi_Q)|^2 \frac{\d x \d t}{t}
\lesssim \|f_Q^{\xi} - \xi_Q\|_2^2 = \| (P_{\varepsilon_0 \ell(Q)}^B - I) \xi_Q \|_2^2 \lesssim |Q|,
\end{equation}
where the implicit constants depend only on $n$, $\kappa$, $K$ and $\upsilon$.

Finally, observe that if $(x,t) \in \mathcal{C}(Q)$, then $(A_t\xi_Q)(x) = \xi$ as $\xi_Q=\mathds{1}_{2Q}\xi\equiv\xi$ on $2Q$, hence
\[
Q_t^B\xi_Q(x) - \widetilde{\gamma}_t A_t\xi_Q(x)
= Q_t^B\xi_Q(x) - \widetilde{\gamma}_t(x) \xi
= Q_t^B(\mathds{1}_{2Q}\xi)(x) - Q_t^B( \mathbb{P}_{\perp\|}\xi)(x)
=  Q_t^B(\mathds{1}_{\R^n\setminus{2Q}}\xi)(x),
\]
where we used $\mathbb{P}_{\perp\|} \xi = \xi$ and $\mathds{1}_{2Q} - \mathds{1}_{\R^n}=\mathds{1}_{\R^n\setminus{2Q}}$ to obtain the final equality. Now, since $\supp (\mathds{1}_{\R^n\setminus{2Q}} \xi) \cap 2Q = \emptyset$, the off-diagonal estimates in Proposition~\ref{prop:Off diagonal estimates} with $M > n$ imply that
\begin{align*}
    \| \mathds{1}_Q Q_t^B (\mathds{1}_{\R^n\setminus{2Q}} \xi )\|_2^2 
    & \leq \left( \sum_{k=1}^{\infty} \| \mathds{1}_Q Q_t^B \mathds{1}_{2^{k+1}Q\setminus 2^kQ} \xi \|_2 \right)^2 \\
    & \lesssim \left( \frac{t}{\ell(Q)} \right)^M \sum_{k=1}^{\infty} 2^{-k M} \| \mathds{1}_{2^{k+1}Q} \xi \|_2^2 \\
    & \lesssim \left( \frac{t}{\ell(Q)} \right)^M \sum_{k=1}^{\infty} 2^{-k (M - n)} |Q| \\
    & \lesssim \left( \frac{t}{\ell(Q)} \right)^M |Q|.
\end{align*}
Therefore, these considerations show that
\begin{equation}\label{eq:c3}
    \int_0^{c_0\ell(Q)} \int_Q | (Q_t^B - \widetilde{\gamma}_t A_t) \xi_Q |^2 \frac{\d x \d t}{t} \lesssim \frac{|Q|}{\ell(Q)^M} \int_0^{\ell(Q)} t^{M-1} \d t \lesssim |Q|,
\end{equation}
where the implicit constants depend only on $n$, $\kappa$, $K$ and $\upsilon$.

Altogether, estimates \eqref{eq:c1}, \eqref{eq:c2} and \eqref{eq:c3} show that 
\[
\int_0^{c_0\ell(Q)} \int_Q | \widetilde{\gamma}_t (x) (A_t f_Q^{\xi} (x))|^2 \frac{\d x \d t}{t}
=\int_0^{c_0\ell(Q)} \int_Q | \widetilde{\gamma}_t A_t f_Q^{\xi}|^2 \frac{\d x \d t}{t}
\lesssim |Q|,
\]
which combined with~\eqref{eq:c0} proves~\eqref{eq:coneE*est} and completes the proof.
\end{proof} 

We can now prove the main homogeneous Carleson measure estimate of this section.

\begin{prop}\label{prop:CMEpropaux}
If $n\geq 3$, $q\geq\max\{\tfrac{n}{2},2\}$ and $V \in B^{q}(\R^n)$ with $\llbracket V\rrbracket_{q}\leq \upsilon<\infty$, then
\[
\sup_{Q \in \Delta^V} \int_0^{ \ell(Q)} \dashint_Q |\widetilde{\gamma}_t(x)|_{\mathcal{L}(\C^{n+2})}^2 \d x\frac{\d t}{t} 
\lesssim 1,
\]
where the implicit constant depends only on $n$, $\kappa$, $K$ and $\upsilon$.
\end{prop}

\begin{proof}
Suppose that $Q \in \Delta^V$ and $\nu \in \mathcal{L}_{\perp\parallel}$. A compactness argument in $\mathcal{L}(\C^{n+2})$ reduces matters to proving that
\begin{equation} \label{eqn:Carleson measure on cones}
\int_0^{ \ell(Q)} \int_Q \mathds{1}_{\{(t,x)\in\R^{n+1}_+:\widetilde{\gamma}_t (x) \in K_{\nu, \sigma}\}} (x,t) |\widetilde{\gamma}_t (x)|_{\mathcal{L}(\C^{n+2})}^2 \d x \frac{\d t}{t} \lesssim |Q|,
\end{equation}
where the implicit constant depends only on $n$, $\kappa$, $K$ and $\upsilon$. For each $\delta\in(0,1)$, Lemma \ref{lem:properties of gamma-tilde} shows that the measure $\mu_{\delta,\nu}$ given by
\[
d\mu_{\delta,\nu}(x,t) \coloneqq \mathds{1}_{(\delta,\delta^{-1})}(t) \mathds{1}_{\{(t,x)\in\R^{n+1}_+:\widetilde{\gamma}_t (x) \in K_{\nu, \sigma}\}} (x,t)  |\widetilde{\gamma}_t (x)|_{\mathcal{L}(\C^{n+2})}^2 \d x \frac{\d t}{t}
\]
for all $(t,x)\in\R^{n+1}_+$, satisfies
\begin{align*}
\mu_{\delta,\nu}(\mathcal{C} (Q))
\leq \frac{1}{\delta} \int_0^{\delta^{-1}} \int_Q |\widetilde{\gamma}_t (x)|_{\mathcal{L}(\C^{n+2})}^2 \d x \d t
\lesssim \frac{|Q|}{\delta^2},
\end{align*}
where the implicit constant depends only on $\kappa$ and $K$. This shows that $\mu_{\delta,\nu}$ is a Carleson measure for the potential $V$ with $\|\mu_{\delta,\nu}\|_{\mathcal{C}_V} < \infty$ in the sense of \eqref{eq:defcmV}.

To prove that $\|\mu_{\delta,\nu}\|_{\mathcal{C}_V}$ does not depend on $\delta$, and only on permitted constants, consider a countable collection $\{Q_k\}$ in $\Delta^V$ associated with $Q$ and $\nu$ as in Proposition~\ref{prop:Final reduction}. We have 
\begin{align*}
\mu_{\delta,\nu}(\mathcal{C}(Q))
    & \leq \iint_{\substack{(x,t) \in E_{Q,\nu}^* \\ \widetilde{\gamma}_t (x) \in K_{\nu,\sigma}}} |\widetilde{\gamma}_t (x)|_{\mathcal{L}(\C^{n+2})}^2 \frac{\d x \d t}{t} + \sum\nolimits_{k} \mu_{\delta,\nu} (\mathcal{C}_V (Q_k) ) \\
    & \leq T |Q| + \| \mu_{\delta,\nu} \|_{\mathcal{C}_V} \sum\nolimits_{k} |Q_k| \\
    & \leq T |Q| + \|\mu_{\delta,\nu}\|_{\mathcal{C}_V} (1-\tau)|Q|,
\end{align*}
where $\tau\in(0,1)$ and $T\in(0,\infty)$ are the constants from Proposition~\ref{prop:Final reduction} that depend only on $n$, $\kappa$, $K$ and $\upsilon$. It follows that $\| \mu_{\delta,\nu} \|_{\mathcal{C}_V} \leq T/\tau$, so by the monotone convergence theorem
\begin{align*}
\int_0^{ \ell(Q)} \int_Q \mathds{1}_{\{(t,x)\in\R^{n+1}_+:\widetilde{\gamma}_t (x) \in K_{\nu, \sigma}\}} (x,t) |\widetilde{\gamma}_t (x)|_{\mathcal{L}(\C^{n+2})}^2 \d x \frac{\d t}{t} 
= \lim_{\delta \to 0} \mu_{\delta,\nu}(\mathcal{C} (Q))
\leq \frac{T}{\tau} |Q|,
\end{align*}
which proves \eqref{eqn:Carleson measure on cones} and thus completes the proof.
\end{proof}

This concludes the proof of Theorem~\ref{thm:QE} when $V\in B^{q}(\R^n)$ with $q\geq\max\{\tfrac{n}{2},2\}$. The following result completes the proof in the remaining case when $V \in L^{n/2} (\R^n)$ with sufficiently small norm.

\begin{prop}\label{prop:VLpCMEpropaux}
If $n\geq 5$ and $\|V\|_{n/2}<\varepsilon_n$, where $\varepsilon_n\in(0,1)$ is from Lemma~\ref{lem:Riesz Transform Estimates with small norm}, then
\[
\int_0^{\infty} \int_{\R^n} |A_t u(x)|^2 |\widetilde{\gamma}_t(x)|_{\mathcal{L}(\C^{n+2})}^2 \d x\frac{\d t}{t} 
\leq \|u\|_2^2
\]
for all $u \in L^2(\R^n;\C^{n+2})$, where the implicit constant depends only on $n$, $\kappa$ and $K$.
\end{prop}

\begin{proof}
This will follow, as in the proof of Proposition~\ref{prop:CMEprop}, from the Carleson measure estimate 
\begin{equation}\label{eq:LpC}
\sup_{Q \in \Delta} \int_0^{ \ell(Q)} \dashint_Q |\widetilde{\gamma}_t(x)|_{\mathcal{L}(\C^{n+2})}^2 \d x\frac{\d t}{t}
\lesssim 1.
\end{equation}
We prove this by modifying the proof of Proposition~\ref{prop:CMEpropaux} as follows. First, Lemma \ref{lem:properties of gamma-tilde} shows that 
$\norm{\mu_{\delta,\nu}}{{\mathcal{C}}}{} \coloneqq \sup_{Q \in \Delta} \mu_{\delta,\nu} (\mathcal{C}(Q)) / |Q| < \infty$. The cancellation from~\eqref{eqn:ILEL} in Lemma~\ref{lem:Interpolation Lemma}, which holds for all dyadic cubes and not just those in $\Delta^V$, then shows that the properties of the test functions in Lemma~\ref{lem:Properties of Test Functions} now hold whenever $Q\in\Delta$.

The stopping-time constructions in Lemma~\ref{lem:Stopping Time} can then made for each $Q$ in $\Delta$ with the scale restrictions $\ell(Q_k) \leq {c_0} \ell(Q)$ and $\ell(S) \leq {c_0} \ell(Q)$ removed. This, in turn, allows for the construction of a countable collection $\{Q_k\}$ in $\Delta$ satisfying \eqref{eq:coneE*est}. In modifying the proof of Proposition~\ref{prop:Final reduction}, estimate~\eqref{eq:clcCM} is avoided. Moreover, the proof of Proposition~\ref{prop:VLpPPR} explains how in this context the estimates in~\eqref{eq:P_new Estimate}, \eqref{eq:P_3 Estimate} and Lemma~\ref{lem:Poincare Estimate} hold, whilst the estimate in Lemma~\ref{lem:Schur Estimate} improves to~\eqref{eq:SchRep}. We then obtain \eqref{eq:LpC} as in Proposition~\ref{prop:CMEpropaux}, noting that the implicit constants in each of the estimates now depend only on $n$, $\kappa$ and $K$, as required.
\end{proof}


\subsection{Kato Boundary Estimates}\label{ssec:KBE}


As a corollary of Theorem~\ref{thm:QE}, we obtain new Kato-type square root results for Schr\"{o}dinger operators $b\mathscr{H}_{A,a,V} = b(-\div A \nabla + aV)$ acting on the domain boundary $\partial\R^{n+1}_+\cong\R^n$. These results were obtained by Gesztesy--Hofmann--Nichols in \cite[Theorem~4.2]{GHN16} for the case when $b\equiv1$, $V\in L^{n/2}(\R^n)$ with sufficiently small norm and $n\geq 3$. Those results were extended by Bailey in \cite[Theorem 5.2]{Bailey18} to allow for $V\in B^2 (\R^n)$ when $n\in\N$ (cf. \cite[Theorem~5.1]{Bailey18}) and $V\in L^{n/2} (\R^n)$ when $n \geq 5$ (cf. \cite[Proposition 5.3]{Bailey18}). In the latter result, where the small norm requirement is not imposed, the ellipticity assumption is modified to incorporate the potential, as in \eqref{eq:ellcoeffKato} below.

We extend Bailey's results by allowing for bounded elliptic coefficients $b$ satisfying~\eqref{eq:b} when $V \in B^{q}(\R^n)$, $q\geq\max\{\tfrac{n}{2},2\}$ and $n\geq 3$, although we do not recover a result for $V\in B^2 (\R^n)$ and $n\in\N$. This is because our method relies on the mixed Riesz transform bound described in Remark~\ref{rem:Riesz} and the comparability in Lemma~\ref{lem:Small Cube Depth Bound}. However, the setup in \eqref{eqn:definition of D and B in Kato proof} below shows that the Kato estimate when $b\equiv1$ follows from the case $B_{\perp\perp}=1$, $B_{\perp\parallel}=0$ and $B_{\parallel\perp}=0$ in Theorem~\ref{thm:QE}, for which the reduction to a Carleson measure estimate in Proposition~\ref{prop:Principle Part Approximation} could possibly be rearranged to avoid the mixed Riesz transform bound. We shall not pursue this here, as our treatment of boundary value problems in Section~\ref{sec:BV} requires $B$ to have the full component structure in~\eqref{eq:bound on R(D)}.
These considerations are not relevant when $V\in B^2 (\R^n)$ and $n\in\{1,2\}$, however, and that case will be treated in a forthcoming paper by Dumont--Morris.

We also recover the analogous results when $V\in L^{n/2} (\R^n)$ and $n \geq 5$ by using a scaling argument to reduce matters to the quadratic estimates obtained for potentials with sufficiently small norm in Theorem~\ref{thm:QE}. The implicit constants in the resulting Kato estimates then depend naturally on the size of the norm $\|V\|_{n/2}$. We do not recover the known results when $b\equiv1$ and $n\in\{3,4\}$, however, because of the second-order Riesz transform bounds used to prove the coercivity in Lemma~\ref{lem:Riesz Transform Estimates with small norm}. These will be considered in a forthcoming paper by Gesztesy--Hofmann--Morris--Nichols.

To state the results, suppose that there exist constants $0<\lambda\leq\Lambda<\infty$ such that the (boundary) coefficients $A \in L^{\infty} (\R^n ; \mathcal{L} (\C^n))$ and $a \in L^{\infty} (\R^n ; \mathcal{L} (\C))$ satisfy the bound
\begin{equation}\label{eq:bddcoeffKato}
\max\left\{\|A\|_{L^{\infty} (\R^n; \mathcal{L} (\C^n))},
\|a\|_{L^{\infty}(\R^n)}\right\} \leq \Lambda,
\end{equation}
as well as the G{\aa}rding-type ellipticity
\begin{equation}\label{eq:ellcoeffKato}
\Re \int_{\R^n} ( (A(x) \nabla f(x), \nabla f(x)) + a(x) V(x) |f(x)|^2) \d x
    \geq \lambda \int_{\R^n} (|\nabla f(x)|^2 + |V(x)| |f(x)|^2)\d x
\end{equation}
for all $f\in C_c^\infty(\R^n)$ (and hence $f\in \Wv^{1,2}(\R^n)$ by Lemma~\ref{lem:testdense}). Next, define the sesquilinear form $\mathfrak{h}_{A,a,V}:\Wv^{1,2} (\R^n)\times\Wv^{1,2} (\R^n)\to\C$ in $L^2(\R^n)$ by 
\[
    \mathfrak{h}_{A,a,V} (f,g) \coloneqq \langle A\nabla f, \nabla g \rangle + \langle a Vf,g\rangle
    = \langle \begin{bmatrix} A & 0 \\ 0 & ae^{i\arg V}|V|^{1/2} \end{bmatrix} \gradV f, \gradV g\rangle
\]
for all $f,\, g \in \Wv^{1,2} (\R^n)$.  This form is closed and densely defined by Lemma~\ref{lem:testdense}, as well as continuous and accretive (see \cite[Definition~1.4]{Ouhabaz2005} for precise definitions), since 
\[
|\mathfrak{h}_{A,a,V} (f,g)| \leq \Lambda \|\gradIV f \|_2 \|\gradIV g \|_2
\quad\text{ and }\quad
\Re \mathfrak{h}_{A,a,V} (f,f) \geq \lambda \|\gradV f\|_2^2
\]
for all $f,\, g \in \Wv^{1,2} (\R^n)$.

The (densely defined) maximal accretive operator  $\mathscr{H}_{A,a,V}:\D(\mathscr{H}_{A,a,V})\subseteq L^2(\R^n)\to L^2(\R^n)$, formally $\mathscr{H}_{A,a,V} = -\div A \nabla + aV$, is then given by 
\[
\mathfrak{h}_{A,a,V} (f,g) = \langle \mathscr{H}_{A,a,V}  f,g\rangle
\]
for all $f\in \D(\mathscr{H}_{A,a,V} )=\{f\in\Wv^{1,2} (\R^n) : A \nabla f + ae^{i\arg V}|V|^{1/2}f \in \D({\gradV}^*)\}$ (see, for instance, \cite[Sections~1.2.3 and 1.3.3]{Ouhabaz2005}). This operator has a unique maximal accretive square root $\mathscr{H}_{A,a,V}^{1/2}$ (see, for instance,  \cite[Theorem~V.3.35]{Kato95}), which can be defined equivalently using the $\mathcal{F}$-functional calculus in Theorem~\ref{thm:FfcDef}, since $\mathscr{H}_{A,a,V}$ is an injective sectorial operator on $L^2(\R^n)$. Furthermore, if $b\in L^\infty(\R^n)$ satisfies
\begin{equation}\label{eq:b}
\|b\|_{L^\infty(\R^n)} \leq \Lambda
\quad\text{ and }\quad
\Re b(x) \geq \lambda
\end{equation}
for almost every $x\in \R^n$, then $b\mathscr{H}_{A,a,V}= -b\div A \nabla + abV$ is also injective and sectorial with square root $(b\mathscr{H}_{A,a,V})^{1/2}$.

We can now state our Kato-type result for potentials $V\in B^{q}(\R^n)$ with $q\geq\max\{\tfrac{n}{2},2\}$.

\begin{cor} \label{cor:Kato for RH} 
If $n\geq 3$, $q\geq\max\{\tfrac{n}{2},2\}$ and $V \in B^{q}(\R^n)$ with $\llbracket V\rrbracket_{q}\leq \upsilon<\infty$  whilst the coefficients $(A,a,b)$ satisfy \eqref{eq:bddcoeffKato}, \eqref{eq:ellcoeffKato} and \eqref{eq:b} with constants $0<\lambda\leq\Lambda<\infty$, then
\[
\D((b\mathscr{H}_{A,a,V})^{1/2}) = \Wv^{1,2}(\R^n)
\quad\text{ and }\quad
\|(b\mathscr{H}_{A,a,V})^{1/2} f \|_2 \eqsim \|\nabla f\|_2 + \|\V f \|_2
\]
for all $f \in \Wv^{1,2} (\R^n)$, where the implicit constants depend only on $n$, $\lambda$, $\Lambda$ and $\upsilon$.
\end{cor}

\begin{proof}
The quadratic estimates for $DB^*=(BD)^*$ and $BD$ in Theorem~\ref{thm:QE} hold  when
\begin{equation}\label{eqn:definition of D and B in Kato proof}
\nabla_\mu = \begin{bmatrix} \nabla \\ \V \end{bmatrix},
\quad
D = \begin{bmatrix}
0 & -(\nabla_\mu)^* \\
-\nabla_\mu & 0 \\
\end{bmatrix}
\quad\text{ and }\quad
B =
\begin{bmatrix}
b & 0 & 0 \\
0 & A & 0 \\
0 & 0 & a
\end{bmatrix},    
\end{equation}
since $B$ satisfies \eqref{eq:bound on R(D)} and \eqref{eq:elliptic on R(D)} with $\|B\|_{\infty} \leq \Lambda$ and $\kappa(B) \geq \lambda$. Therefore, the restriction of $BD$ to $\overline{\Ran(BD)}$, which we shall denote here as $BD$, has a bounded $H^\infty(S_{\omega(B)})$-functional calculus on $\overline{\Ran(BD)}$ by Remark~\ref{rem:injpart} and Theorem~\ref{thm:Quadratic estimates and functional calculus}.

Now choose $
\mu\in(\omega(B),\frac{\pi}{2})$, depending only on $
\lambda$ and $\Lambda$, and define $\sgn(z)\coloneqq \sqrt{z^2}/z$ for all $z\in S_\mu^o$, so $\|\sgn(BD)u\|_2 \lesssim \|u\|_2$ for all $u\in\overline{\Ran(D)}$, where the implicit constant depends only on $n$, $\lambda$, $\Lambda$ and $\upsilon$. We then have $\sqrt{(BD)^2}=\sgn(BD)BD$ and $BD=\sgn(BD)\sqrt{(BD)^2}$ with $\D(\sqrt{(BD)^2})=\D(BD)$ by Theorem~\ref{thm:FfcDef}. These identities follow as in~\eqref{eq:T[T]} below, since $\sqrt{(BD)^2}=[BD]$ can be understood from an abstract composition rule for the  $\mathcal{F}$-functional calculus (see \cite[Theorem~2.4.2]{Haase06} or \cite[Theorem~ 3.2.20]{EgertPhD}). Therefore, since $f\in\D((b\mathscr{H}_{A,a,V})^{1/2})$ if and only if  $u=(f,0)\in\D(\sqrt{(BD)^2})$, we have $\D((b\mathscr{H}_{A,a,V})^{1/2})=\D(\gradV)=\Wv^{1,2}(\R^n)$. Finally, as such $u=(f,0)\in \overline{\Ran(D)}$ by Lemma~\ref{lem:RanD}, the bound and identities for $\sgn(BD)$ imply that
\[
    \|(b\mathscr{H}_{A,a,V})^{1/2} f \|
    = \|\sqrt{(BD)^2} u \|_2
    \eqsim \|BD u\|_2
    = \| \gradV f\|_2,
\]
where the implicit constants depend only on $n$, $\lambda$, $\Lambda$ and $\upsilon$, as required.
\end{proof}

We now use a scaling argument to obtain a Kato-type result for potentials $V \in L^{n/2} (\R^n)$ with arbitrary norm. The key idea is to factor out the size of the norm via a suitable coefficient matrix $B$ in conjunction with a first-order operator $D$ for which the quadratic estimates for potentials with sufficiently small norm in Theorem~\ref{thm:QE} hold.

\begin{cor} \label{cor:Kato for Ln/2}
If $n\geq 5$ and $V \in L^{n/2} (\R^n)$ with $\| V\|_{n/2}\leq \upsilon < \infty$ whilst the coefficients $(A,a,b)$ satisfy \eqref{eq:bddcoeffKato}, \eqref{eq:ellcoeffKato} and \eqref{eq:b} with constants $0<\lambda\leq\Lambda<\infty$, then \[
\D((b\mathscr{H}_{A,a,V})^{1/2}) = W^{1,2}(\R^n)
\quad\text{ and }\quad
\|(b\mathscr{H}_{A,a,V})^{1/2} u \|_2 \eqsim \|\nabla u\|_2
\]
for all $u \in W^{1,2} (\R^n)$, where the implicit constants depend only on $n$, $\lambda$, $\Lambda$ and $\upsilon$.
\end{cor}

\begin{proof}
Suppose that $\varepsilon_n\in(0,1)$ is from Lemma~\ref{lem:Riesz Transform Estimates with small norm} so that the results in Theorem~\ref{thm:QE} hold. Now consider $\widetilde{V} (x) \coloneqq \tfrac{\varepsilon_n}{2} \|V\|_{n/2}^{-1}|V(x)|$ and $\widetilde{a}(x) \coloneqq  \frac{2}{\varepsilon_n}\|V\|_{n/2}a(x)e^{i \arg(V(x))}$ for all $x\in\R^n$, so  $\|\widetilde{V}\|_{n/2}<\varepsilon_n$ and $\widetilde{a}\, \widetilde{V} = a V$. The quadratic estimates in Theorem~\ref{thm:QE} hold when $D$ and $B$ are defined as in \eqref{eqn:definition of D and B in Kato proof} but with $V$ and $a$ replaced by $\widetilde{V}$ and $\widetilde{a}$, 
since in that case $B$ satisfies \eqref{eq:bound on R(D)} and \eqref{eq:elliptic on R(D)} with $\|B\|_{\infty} \leq \max\{1,\frac{2}{\varepsilon_n}\|V\|_{n/2}\}\Lambda$ and $\kappa(B) \geq \lambda$. The strategy used to prove Corollary~\ref{cor:Kato for RH} then shows that $\D((b\mathscr{H}_{A,a,V})^{1/2})=\D((b\mathscr{H}_{A,\widetilde{a},\widetilde{V}}^{1/2}))=\D((\nabla,|\widetilde{V}|^{1/2}))=W^{1,2}(\R^n)$ with
\[
    \|(b\mathscr{H}_{A,a,V})^{1/2} u \|_2
    = \|(b\mathscr{H}_{A,\widetilde{a},\widetilde{V}})^{1/2} u \|_2
    \eqsim \|\nabla u\|_2 + \||\widetilde{V}|^{1/2} u \|_2
    \eqsim \|\nabla u\|_2,
\]
for all $u \in W^{1,2} (\R^n)$, where we used the Sobolev inequality~\eqref{eq:Sobolev} to obtain the final estimate, and the implicit constants depend only on $n$, $\lambda$, $\Lambda$ and $\upsilon$.
\end{proof}


\subsection{Analytic Dependence and Lipschitz Estimates}\label{ssec:andep}


Here we show that the functional calculus bounds implied by the quadratic estimates in Theorem~\ref{thm:QE} for operators $DB$ depend analytically on perturbations of the coefficients $B$ with respect to the $L^{\infty}$-norm. The result is used to prove the well-posedness for Hermitian coefficients in 
Theorem~\ref{thm:Second-Order WP}, which relies on the perturbation of boundary isomorphisms in Theorem~\ref{thm:absper} via a method of continuity in Proposition~\ref{prop:Self-Adjoint Isomorphisms}. It is also used to prove the perturbation of well-posedness in Theorem~\ref{thm:Second-Order WP pert}. We follow the approach of Axelsson--Keith--McIntosh in \cite[Section 6]{AKMc06}, beginning with the following lemma for self-adjoint operators $D$, in which we carefully trace the dependency of the constants to suit our purposes.

\begin{lem} \label{lem:Analytic dependence}
Suppose that $D$ is a self-adjoint operator on $L^2 (\R^n ; \C^N)$. If $U \subseteq \C$ is open and $B \colon U \to L^{\infty} (\R^{n} ; \mathcal{L} (\C^N))$ is holomorphic with $\sup_{z \in U} \| B(z)\|_{\infty} \leq K_0 < \infty$ and 
\[
0 < \kappa_0 \leq \inf_{z\in U}\kappa(B(z))
= \inf_{z\in U} \inf_{v \in \Ran(D)\setminus\{0\}} \frac{\Re \langle B(z) v , v \rangle}{\|v\|_2^2},
\]
then $z \mapsto (I + \zeta DB(z))^{-1}u$ is holomorphic on $U$ whenever $\zeta\in \C \setminus S_{\mu}^o$,  $\mu\in(\cos^{-1}(\kappa_0/K_0),\tfrac{\pi}{2}]$ and $u\in L^2(\R^n;\C^N)$. Moreover, if  $\mu\in(\cos^{-1}(\kappa_0/K_0),\tfrac{\pi}{2}]$ and
\begin{equation}\label{eq:bfcpert}
\sup_{z\in U} \sup_{f\in H^\infty(S_\mu^o)\setminus\{0\}} \sup_{v \in \Ran(D)\setminus\{0\}} \frac{\|f(DB(z))v\|_2}{\|f\|_\infty\|v\|_2} < \infty,
\end{equation}
then $z \mapsto f (DB(z))u$ is holomorphic on $U$ whenever $f\in H^\infty(S_\mu^o)$ and $u\in\overline{\Ran(D)}$.
\end{lem}

\begin{proof}
Suppose that $w,\, z \in U$, let $B_w=B(w)$ and $B_z=B(z)$, and consider $\zeta \in \C \setminus S_{\mu}^o$ with $\mu\in(\cos^{-1}(\kappa_0/K_0),\tfrac{\pi}{2}]$. If $u\in L^2(\R^n;\C^N)$, then the second resolvent identity shows that
\[
(I + \bar{\zeta} B_z^* D)^{-1}u - (I + \bar{\zeta} B_w^* D)^{-1}u = -(I + \bar{\zeta} B_z^* D)^{-1} \bar{\zeta} \left( B_z^* - B_w^*\right) D (I + \bar{\zeta} B_w^* D)^{-1}u.
\]
This identity relies on the fact that the adjoint operators preserve the domain of $D$ in the sense that $\D (B_z^* D) = \D(D) = \D (B_w^* D)$. We then combine the resolvent bounds for $B_z^* D$ and $B_w^* D$ from \eqref{eq:DBresbdd} with the accretivity of $B_w^*$ on $\Ran(D)$ to obtain
\begin{align*}
    \| (I + \bar{\zeta} B_z^* D)^{-1} u &- (I + \bar{\zeta} B_w^* D)^{-1} u \|_2 \\
    & \lesssim \|B_z^* - B_w^*\|_{\infty} \| \bar{\zeta} (B_w^*)^{-1} B_w^* D (I + \bar{\zeta} B_w^* D)^{-1} u\|_2, \\
    & \lesssim \|B_z - B_w\|_{\infty} \| \bar{\zeta} B_w^* D (I + \bar{\zeta} B_w^* D)^{-1} u\|_2, \\
    & \lesssim \|B_z - B_w\|_{\infty} \|u\|_2,
\end{align*}
where the implicit constants depend only on $\kappa_0$ and $K_0$. It follows that
\[
    \lim_{w \to z} \frac{\| (I + \bar{\zeta} B_z^* D)^{-1}u - (I + \bar{\zeta} B_w^* D)^{-1}u\|_2}{|z-w|} \lesssim \lim_{w \to z} \frac{\| B_z - B_w\|_{\infty}}{|z-w|}\|u\|_2 = 0,
\]
since $B$ is holomorphic, hence $z \mapsto (I + \bar\zeta B_z^*D)^{-1}u$ and in-turn the adjoint $z \mapsto (I + \zeta DB_z)^{-1}u$ are holomorphic on $U$.

To complete the proof, recall from beneath \eqref{eq:DBresbdd} that $\overline{\Ran(DB_z)}=\overline{\Ran(D)}$, so the injective part (see Remark~\ref{rem:injpart}) of $DB_z$ is the restriction of $DB_z$ to $\overline{\Ran(D)}$. Now let $DB_z$ denote this restriction. If $\psi \in \Psi (S_{\mu}^o)$, then we can prove that $z\mapsto \psi(DB_z)$ is holomorphic on $U$ by using Riemann sums to approximate the Cauchy integral representation $\psi (DB_z) = \frac{1}{2\pi i}\int_{+\partial S_\theta^0} \psi (\zeta) (\zeta - DB_z)^{-1} \d \zeta$ from \eqref{eq:CIF} and applying the result just proved for resolvents. To conclude, suppose that $\eqref{eq:bfcpert}$ holds and $f \in H^{\infty} (S_{\mu}^o)$. In that case,  we can choose a uniformly bounded sequence of functions $\psi_n$ in $\Psi (S_{\mu}^o)$ that converges uniformly on compact sets to $f$ and apply the convergence lemma from Theorem~\ref{thm:FfcDef} to deduce that $f(DB_z) u = \lim_{n \to \infty} \psi_n (DB_z) u$ whenever $u\in\overline{\Ran(D)}$, so the result follows.
\end{proof}

We now prove the analytic dependence of the functional calculus bounds for the operators $DB$ from Theorem~\ref{thm:QE}.

\begin{thm} \label{thm:Lipschitz Estimates}
Suppose that $D$ and $B$ satisfy either of the hypotheses in Theorem~\ref{thm:QE}. If $\delta \in (0,\kappa)$, $\widetilde{B} \in L^{\infty} (\R^n ; \mathcal{L} (\C^{n+2}))$
with the component structure in \eqref{eq:bound on R(D)} and $\|\widetilde{B} - B\|_{\infty} < \delta$, then $D\widetilde{B}$ has a bounded $H^\infty(S_{\omega})$-functional calculus, where $\omega = \cos^{-1}((\kappa-\delta)/(K+\delta)) \in [\omega(B), \tfrac{\pi}{2})$, and
\[
    \| f (D\widetilde{B}) u - f(DB) u\|_2 \lesssim \|\widetilde{B} - B\|_{\infty} \|f\|_{\infty} \|u\|_2
\]
for all $f \in H^{\infty} (S_{\mu}^o)$, $\mu\in(\omega,\tfrac{\pi}{2}]$ and $u\in \overline{\Ran(D)}$, where the implicit constant depends only on $n$, $\kappa-\delta$, $K+\delta$, $\upsilon$ and $\mu$.
\end{thm}

\begin{proof}
Let $U=\{ z \in \C \colon |z| < 1 \}$ and define $B_\delta \colon U \to L^{\infty} (\R^n ; \mathcal{L} (\C^{n+2}))$ by 
\[
    B_\delta(z) 
    \coloneqq B + z \delta \|\widetilde{B} - B\|_{\infty}^{-1} (\widetilde{B} - B)
\]
for all $z\in U$. Observe that $B_\delta$ is holomorphic with $\|B_\delta(z)\|_\infty \leq \|B\|_\infty + \delta < K+\delta$ and 
\[
    \Re \langle B_\delta(z) Du , Du \rangle
    = \Re \langle B Du , Du \rangle  - \Re \langle z\delta \|\widetilde{B} - B\|_{\infty}^{-1} ({\widetilde{B} - B}) Du , Du \rangle 
    \geq (\kappa - \delta ) \|Du\|_2^2
\]
for all $u\in\D(D)$ and $z\in U$.
It then follows from Theorem~\ref{thm:QE} and Theorem~\ref{thm:Quadratic estimates and functional calculus} that $DB_\delta(z)$ has a bounded $H^\infty(S_{\omega(B_\delta(z))})$-functional calculus, where $\omega(B_\delta(z)) \leq \cos^{-1}((\kappa-\delta)/(K+\delta)) < \tfrac{\pi}{2}$, and 
\[
\|f(DB_\delta(z)) u\|_2 \leq C_\mu \|f\|_{\infty} \|u\|_2
\]
for all $f \in H^{\infty} (S_{\mu}^o)$, $\mu > \omega(B_\delta(z))$, $u\in \overline{\Ran(D)}$ and $z\in v$, where $C_\mu$ depends only $n$, $\kappa-\delta$, $K+\delta$, $\upsilon$ and $\mu$. Note that $\omega(B_\delta(0))=\omega(B)\leq \cos^{-1}(\kappa/K)$.

Suppose that $f \in H^{\infty} (S_{\mu}^o)$, $\mu\in(\omega,\tfrac{\pi}{2}]$ and $u\in \overline{\Ran(D)}$ to  define $G \colon U \to L^2 (\R^n;\C^{n+2})$ by 
\[
    G (z) \coloneqq \frac{f(DB_\delta(z))u - f(DB)u}{2C_\mu \|f\|_{\infty} \|u\|_2}
\]
for all $z\in U$. This is a holomorphic function, since $z \mapsto f(D B_\delta(z))u$ is holomorphic on $U$ by Lemma~\ref{lem:Analytic dependence}, whilst the functional calculus bounds for $DB_\delta(z)$ imply that
\[
    \|G (z)\|_2 \leq \frac{1}{2 C_\mu \|f\|_{\infty} \|u\|_2} (\|f(DB_\delta(z)) u\|_2 + \|  f(DB) u \|_2) \leq 1.
\]
The mapping $z\mapsto \langle G (z), \varphi \rangle$ is therefore holomorphic on $U$, and also equal to 0 when $z=0$, for all $\varphi \in L^2 (\R^n ; \C^{n+2})$, whilst $|\langle G (z), \varphi \rangle| \leq \|G (z)\|_2 \|\varphi\|_2 \leq 1$ whenever $\|\varphi\|_2 = 1$. The Schwarz lemma then implies that
\[
    \|G (z) \|_2
    = \sup_{\|\varphi\|_2 = 1} |\langle G (z), \varphi \rangle|
    \leq |z|
\]
for all $z\in U$. Thus, choosing  $z = \|\widetilde{B} - B \|_{\infty} / \delta \in U$ so that $B_\delta(z)=\widetilde{B}$, we obtain
\[
    \|f(D\widetilde{B}) u - f(DB) u \|_2
    \leq 2 C_\mu \|\widetilde{B} - B \|_{\infty}  \|f\|_{\infty} \|u\|_2,
\]
as required.
\end{proof}


\section{Initial Value Problems for First-Order Systems}\label{sec:BV}


We ultimately consider solutions $F:\R_+ \to L^2(\R^n;\C^{n+2})$ of Cauchy--Riemann type systems
\begin{equation} \label{eqn:First-Order Equation}
    \p F(t) + DB F(t) = 0
\end{equation}
adapted to the first-order operators $DB$ from Theorem~\ref{thm:QE} acting on $t$ in $\R_+$. These are shown to characterise a class of weak solutions for second-order equations $\HAaV  u = 0$ on $\R^{n+1}_+$ in Proposition~\ref{prop:Reduction to first-order}. The main benefit then comes from obtaining semigroup representations for solutions to initial value problems for these first-order systems in Theorem ~\ref{Thm:First order solutions are semigroups}. The relevant semigroups are generated by the action of $DB$ on certain spectral subspaces, which are also the trace spaces for the initial data. Well-posedness for the associated second-order boundary value problems is then equivalent to certain  mappings on the boundary $\partial\R^{n+1}_+\cong\R^n$ being isomorphisms, which is shown to be the case for block and Hermitian coefficients in Section~\ref{sec:ADD}. The more general equivalence, which requires the non-tangential maximal function bounds in Section~\ref{sec:NT}, is postponed to Section~\ref{Sec:Return to the Second-Order Equation}.

We adopt the convention for functions $\phi:\R^{n+1}_+ \to \C^{n+2}$ whereby $\phi(t):\R^{n} \to \C^{n+2}$ is defined by $(\phi(t))(x) \coloneqq \phi(t,x)$ for all $x\in\R^n$ and $t \in \R_+$. We shall write that \textit{$F$ is a weak solution of $\p F + DB F = 0$ in $\R_+$}, or simply  \textit{$\p F + DB F = 0$ in $\R_+$}, if $F \in L^2_{\Loc} (\R_+ ; L^2 (\R^{n} ; \C^{n+2}))$ and
\begin{equation}\label{eq:weak1st}
\int_0^{\infty} \langle F(t), \p \varphi(t) \rangle  \d t
= \int_0^{\infty} \langle BF(t), D \varphi(t) \rangle \d t
\end{equation}
for all $\varphi \in L^2(\mathbb{R}_+;\D(D)) \cap \test(\R_+ ; L^2 (\R^{n} ; \C^{n+2}))$ with the graph norm on $\D(D)$ given by~\eqref{graph norm}.

The space of test functions specified above is strictly larger than the space $\test(\R^{n+1}_+; \C^{n+2})$ used in the distributional sense considered by Auscher--Axelsson in~\cite[Proposition~4.1]{AA11}. The tangential smoothness for test functions has been weakened to the requirement that $\varphi(t)\in \D(D)$ for all $t\in\supp{\varphi}\subseteq[a,b]\subset\mathbb{R}_+$ with $\int_0^\infty \|\varphi(t)\|_{\D(D)}^2 \d t
= \int_a^b (\|\varphi(t)\|_2^2 + \|D\varphi(t)\|_2^2) \d t < \infty$. This still allows us to establish the equivalence with weak solutions $\HAaV  u = 0$ in Proposition~\ref{prop:Reduction to first-order}. It does not seem possible, however, to obtain the semigroup characterisation for weak solutions $\p F + DB F = 0$ in Theorem~\ref{Thm:First order solutions are semigroups}, and in particular Lemma~\ref{lem:weak derivatives are 0}, by considering only $\test$ test functions, as ultimately it is not clear whether such functions are dense in $\D(\nabla_\mu^*)$. Instead, we appeal to the abstract approach taken by Auscher--Axelsson in \cite[Proposition~2.2]{AA2011Remarks} to establish such results, as our test functions more closely resemble the type they considered at (7) in \cite{AA2011Remarks}.


\subsection{Equivalent First-Order Systems}


Throughout this section, we suppose that $n\geq3$ and $V\in L^1_{\Loc}(\R^n)$ with bounded elliptic coefficients $(A,a)$ and $\cA=\cA_{A,a,V}$ as in~\eqref{eq:AcA}--\eqref{eq:bddellA} with constants $0<\lambda\leq\Lambda<\infty$. The operator $D$ is defined as in~\eqref{eq:Ddef} and we will now construct the required coefficient operator $B$ satisfying \eqref{eq:bound on R(D)}--\eqref{eq:elliptic on R(D)}. We adapt the approach initiated by Auscher--Axelsson--McIntosh in~ \cite[Proposition~3.2]{AAMc10-2} and refined by Auscher--Axelsson in~\cite[Proposition~4.1]{AA11} by considering the bounded operators $\cA$, $\ola{\cA}$ and $\ula{\cA}$ defined pointwise as multiplication operators on $L^2(\R^n;\C^{n+2})$ by
\[
\cA  \coloneqq  \begin{bmatrix}
\App & \Apv & 0 \\
\Avp & \Avv & 0 \\
0 & 0 & a e^{i\arg{V}}
\end{bmatrix},\
\ola{\cA} \coloneqq 
\begin{bmatrix}
\App & \Apv & 0 \\
0 & I & 0 \\
0 & 0 & I
\end{bmatrix}
\text{ and }\ 
\ula{\cA} \coloneqq
\begin{bmatrix}
I & 0 & 0 \\
\Avp & \Avv & 0 \\
0 & 0 & a e^{i\arg{V}}
\end{bmatrix}.
\]
The operator $\ola{\cA}$ is injective, since $\Re\App\geq\lambda$ almost everywhere on $\R^n$ by \eqref{eq:ellimp}, with
\begin{equation}\label{eq:Ahatinv}
\ola{\cA}^{-1} = 
\begin{bmatrix}
\App^{-1} & -\App^{-1} \Apv & 0 \\
0 & I & 0 \\
0 & 0 & I
\end{bmatrix}.
\end{equation}
We prove below that the operator $\widehat{\mathcal{A}}$ defined by
\begin{equation}\label{eq:Ahatdef}
\widehat{\cA} \coloneqq  \ula{\cA} \ola{\cA}^{-1} = 
\begin{bmatrix}
\App^{-1} & - \App^{-1} \Apv & 0 \\
\Avp \App^{-1} & \Avv - \Avp \App^{-1} \Apv & 0 \\
0 & 0 & a e^{i\arg{V}}
\end{bmatrix}
\end{equation}
inherits a bound and ellipticity from $\cA$. Moreover, we will see in Proposition~\ref{prop:Reduction to first-order} how solutions of the Schr\"odinger equation~\eqref{eqn:Schrodinger Equation} can be characterised by the Cauchy--Riemann type system~ \eqref{eqn:First-Order Equation} for the operator $DB$ with coefficients $B = \widehat{\cA}_{A,a,V}$.

\begin{prop}\label{prop:bddellAhat} If $n\geq3$, $V\in L^1_{\Loc}(\R^n)$ and $\cA$ satisfies the bound and ellipticity in \eqref{eq:bddellA} with constants $0<\lambda\leq \Lambda<\infty$, then 
\[
\|\widehat{\cA}\|_\infty \leq \frac{\max\{\Lambda^2,1\}}{\lambda}
\quad\text{ and }\quad
\Re \langle\widehat{\cA}u,u\rangle \geq \frac{\lambda}{\Lambda} \|u\|_2^2
\]
for all $u\in\overline{\Ran(D)}$, whilst $\hhat{\cA} = \cA$.
\end{prop}

\begin{proof}
The definition $\widehat{\cA}\coloneqq  \ula{\cA}\ola{\cA}^{-1}$ implies the bound for $\|\widehat{\cA}\|_\infty$ and a computation reveals that $\hhat{\cA} = \cA$. If $u$ is in $\overline{\Ran(D)}$, then Lemma~\ref{lem:RanD} shows that $v = \left[\App^{-1} (u_{\perp} - \Apv u_{\parallel}), u_{\parallel}, u_\oo\right]$ is in $\overline{\Ran(D)}$
with $\ola{\cA} v =u$, so $\ola{\cA} \colon \overline{\Ran(D)} \to \overline{\Ran(D)}$ is a bijection. Moreover, since $\widehat{\cA} \ola{\cA}=\ula{\cA}$, we have
\begin{align*}
\Re \langle \widehat{\cA} u , u \rangle 
& = \Re \langle \ula{\cA} v , \ola{\cA} v \rangle \\
& = \Re \left \langle 
\begin{bmatrix}
    v_{\perp} \\
    \Avp v_{\perp} + \Avv v_{\parallel} \\
    a e^{i\arg{V}} v_\oo
\end{bmatrix},
\begin{bmatrix}
    \App v_{\perp} + \Apv v_{\parallel} \\
    v_{\parallel} \\
    v_\oo
\end{bmatrix}
\right\rangle \\
&= \Re \left\langle 
\begin{bmatrix}
	\App v_{\perp} + \Apv v_{\parallel} \\
    \Avp v_{\perp} + \Avv v_{\parallel} \\
    a e^{i\arg{V}} v_\oo
\end{bmatrix},
\begin{bmatrix}
	v_{\perp} \\
    v_{\parallel} \\
    v_\oo
\end{bmatrix}
\right\rangle \\
&=\Re \langle \cA v , v \rangle,
\end{align*}
whence \eqref{eq:ellimp} implies that
$
\Re \langle \widehat{\cA} u , u \rangle
\geq \lambda \|v\|_2^2 
= \lambda \|\ola{\cA}^{-1} u\|_2^2
\geq (\lambda/\Lambda) \|u\|_2^2,
$
as required.
\end{proof}

The above construction shows that $B = \widehat{\cA}_{A,a,V}$ satisfies the bound, component structure and strict accretivity on $\Ran(D)$ stated in \eqref{eq:bound on R(D)} and \eqref{eq:elliptic on R(D)} with $\|B\|_\infty\leq\max\{\Lambda^2,1\}/\lambda$ and $\kappa(B)\geq\lambda/\Lambda$. The technical lemma below will be used to obtain the aforementioned correspondence between solutions of the Schr\"odinger equation and Cauchy--Riemann type system.

\begin{lem}\label{lem:pusol}
If $n\geq3$, $V\in L^1_{\Loc}(\R^n)$ and $\HAaV  u = 0$ in $\R^{n+1}_+$, then $\HAaV (\p u)=0$ in $\R^{n+1}_+$ with $\nabla_\mu(\p u)=\p(\nabla_\mu u)$ and 
\[
\iint_W |\nabla_\mu(\p u)|^2
\lesssim \frac{1}{ l(W)^2} \iint_{2W} |\nabla u|^2 
\]
for all Whitney cubes $W\subset\R^{n+1}_+$, where the implicit constant depends only on $n$, $\lambda$ and $\Lambda$.
\end{lem}

\begin{proof}
Suppose that $\HAaV  u = 0$ in $\R^{n+1}_+$. Let $W$ denote an arbitrary Whitney cube in $\R^{n+1}_+$, so $W=W(x,t)=(t,2t) \times Q(x,t)$ for some $x \in \R^n$ and $t>0$, where $Q(x,t)$ is the cube of side length $t$ centred at $x$ in $\mathbb{R}^n$. To prove that $\p u$ is in $\mathcal{V}^{1,2}_{\Loc}(\R^{n+1}_+)$, consider the difference quotient $\mathbb{D}_h u(s,y) \coloneqq \tfrac{1}{h}[u(s+h,y)-u(s,y)]$ for all $(s,y)\in W(x,t)$ and $0<|h|< \frac{1}{4}t$.

The $t$-independence of the equation coefficients implies that $\HAaV (\mathbb{D}_h u) = 0$ in $2W$, so we can use Caccioppoli's inequality (see Proposition~\ref{prop:Caccioppoli}) to obtain
\begin{align*}
\iint_{W} |\mathbb{D}_h(\partial_i u)|^2
=\iint_{W} |\partial_i(\mathbb{D}_h u)|^2
\lesssim \frac{1}{ l(W)^2} \iint_{\frac{3}{2}W} |\mathbb{D}_h u|^2
\leq \frac{1}{ l(W)^2} \iint_{2W} |\partial_t u|^2 =:K
\end{align*}
for all $0<|h|< \frac{1}{4}t$ and $i\in\{0,\ldots,n\}$, where the constants depend only on $\lambda$, $\Lambda$ and $n$. The final bound above also holds uniformly with respect to $h$ because $u$ is in $W^{1,2}_{\Loc}(\R^{n+1}_+)$ and $2W=(\frac{t}{2},\frac{5t}{2})\times Q(x,2t)\Subset\R^{n+1}_+$ (see \cite[Lemma 7.23]{GT77}). We can then use \cite[Lemma~7.24]{GT77} to deduce that $\partial_i u$ is weakly differentiable in the $t$-direction on $W$ with $\|\partial_t(\partial_i u)\|_{L^2(W)}^2\lesssim K$ for all $i\in\{0,\ldots,n\}$. 
The $t$-independence of the potential allows us to  likewise obtain
\begin{align*}
\iint_{W} |\mathbb{D}_h (V^{1/2}u)|^2
=\iint_{W} |V||\mathbb{D}_h u|^2
\lesssim \frac{1}{ l(W)^2} \iint_{\frac{3}{2}W} |\mathbb{D}_h u|^2
\leq \frac{1}{ l(W)^2} \iint_{2W} |\p u|^2
= K
\end{align*}
for all $0<|h|< \frac{1}{4}t$, where the implicit constant depends only on $\lambda$, $\Lambda$ and $n$, hence $V^{1/2}u$ is weakly differentiable in the $t$-direction on $W$ with $\|\partial_t(|V|^{1/2}u)\|_{L^2(W)}^2\lesssim K$.

The equality $\partial_i(\partial_t u)=\partial_t(\partial_iu)$ for weak derivatives is now immediate for all $i\in\{0,\ldots,n\}$,  whilst $|V|^{1/2}(\p u)=\p(|V|^{1/2}u)$ follows from the $t$-independence of the potential. This shows that $\gradV(\p u)=\p(\gradV u)$ and $\p u\in\mathcal{V}^{1,2}_{\Loc}(\R^{n+1}_+)$, with the required estimate on Whitney cubes, from which it is straightforward to verify that $\HAaV (\p u)=0$ in $\R^{n+1}_+$.
\end{proof}

We now make precise the correspondence between weak solutions of the second-order equation as defined in \eqref{eq:weak2nd} and weak solutions of the first-order system as defined in \eqref{eq:weak1st}. The level of detail included in the proof obtained for parabolic equations by  Auscher--Egert--Nystr\"{o}m in  \cite[Theorem~2.2]{AEN16-1} is especially helpful here, since we adapt it using Lemma~\ref{lem:pusol} to account for the potential and to allow for the larger class of test functions associated with \eqref{eq:weak1st}. Note that 
$\gradAV u = (\partial_{\nu_A} u, \gradVp u) = ((A \nabla u)_\perp,\gradx u, \V u)$ for all $u\in\mathcal{V}^{1,2}_{\Loc}(\R^{n+1}_+)$, whilst Lemma~\ref{lem:NT Control Implies Second-Order Has First-Order Solution} shows that the control of $\gradAV u$ required in (1) below holds when ${N}_* (\gradIV u) \in L^2 (\R^n)$.

\begin{prop} \label{prop:Reduction to first-order}
If $n\geq3$ and $V\in L^1_{\Loc}(\R^n)$, then there is the following correspondence between weak solutions of the second-order equation $\HAaV  u = 0$ and weak solutions of the first-order system $\p F + DB F = 0$ when $B = \widehat{\cA}_{A,a,V}$:
\begin{enumerate}
\item If $\HAaV  u = 0$ in $\R^{n+1}_+$ and $\gradAV u \in L^2_{\Loc} (\R_+ ; L^2 (\R^n ; \C^{n+2}))$, then $F \coloneqq \gradAV u$ is in $L^2_{\Loc} (\R_+ ; \Hil)$ and $\p F + DB F = 0$ in $\R_+$;
\item If $\p F + DB F = 0$ in $\R_+$ and $F \in L^2_{\Loc} (\R_+ ; \Hil)$, then there exists $u$ in $\Wv^{1,2}_{\Loc}(\R^{n+1}_+)$ such that $F = \gradAV u$ and $\HAaV  u = 0$ in $\R_+^{n+1}$.
\end{enumerate}
\end{prop}

\begin{proof}
We first prove that (1) holds. Suppose that $\HAaV  u = 0$ in $\R^{n+1}_+$ and that $F\coloneqq \gradAV u$ is in $L^2_{\Loc} (\R_+ ; L^2 (\R^n ; \C^{n+2}))$. In that case, Lemma~\ref{lem:RanD} implies that $F (t) \in \overline{\Ran(D)}$ for almost every $t\in\mathbb{R}_+$, so $F \in L^2_{\Loc} (\R_+ ; \Hil)$.

Now consider an arbitrary function $\varphi$ in $L^2(\mathbb{R}_+;\D(D)) \cap \test(\R_+ ; L^2 (\R^{n} ; \C^{n+2}))$ as required by \eqref{eq:weak1st}. Observe that $\varphi$ is supported in $[a,b]\times\R^n$, for some $0<a<b<\infty$, and that $\varphi_\perp(t) \in \D(\gradVp) = \mathcal{V}^{1,2}(\R^n)$ and $(\p \varphi_\perp)(t) \in L^2(\R^n)$ for almost every $t\in\mathbb{R}_+$, with 
\[
\int_0^\infty \|\gradVp\varphi_\perp(t)\|_{L^2(\R^n;\C^{n+1})}^2 \d t
+ \int_0^\infty
\|\partial_t\varphi_\perp(t)\|_{L^2(\R^n)}^2 \d t < \infty,
\]
so $\varphi_\perp\in \mathcal{V}^{1,2}(\R^{n+1}_+)$. The nature of the support of $\varphi_\perp$ then ensures that its extension by zero to the full space $\R^{n+1}$ is in $\mathcal{V}^{1,2}(\R^{n+1})$, and we know from Lemma~\ref{lem:testdense} that $\test(\R^{n+1})$ is dense in $\mathcal{V}^{1,2}(\R^{n+1})$. Moreover, an inspection of that proof shows that the mollifiers used therein can be chosen here to obtain a sequence of functions in $\test([\tfrac{1}{2}a,2b]\times\mathbb{R}^n)$ that converges to $\varphi_\perp$ in $\mathcal{V}^{1,2}(\R^{n+1}_+)$. This provides a sequence of test functions which can be substituted into \eqref{eq:weak2nd} and combined with the fact that $\HAaV u = 0$ in $\R^{n+1}_+$ to deduce that
\[
\int_0^{\infty} \int_{\R^n} \left(A \nabla_{t,x} u \cdot \overline{\nabla_{t,x} \varphi_\perp} + a V u \overline{\varphi_\perp}\right) \d x \d t = 0,
\]
where the required convergence is guaranteed because $\gradAV u$, and hence by ellipticity $\gradV u$, are in $L^2_{\Loc} (\R_+ ; L^2 (\R^n ; \C^{n+2}))$. We rearrange this identity to obtain
\begin{align*}
\int_0^{\infty} \langle F_{\perp} , \p \varphi_{\perp} \rangle
& = \int_0^{\infty} \langle (A \nabla_{t,x} u)_{\perp} , \p \varphi_{\perp} \rangle \d t \\
& = - \int_0^{\infty} \langle (A \nabla_{t,x} u)_{\parallel} , \gradx \varphi_\perp \rangle \d t - \int_0^{\infty} \langle a e^{i\arg{V}} \V u , \V \varphi_\perp \rangle \d t \\
& = \int_0^{\infty} \langle (B F)_{\parallel} , (D \varphi)_{\parallel} \rangle \d t + \int_0^{\infty} \langle (BF)_\oo , (D \varphi)_\oo \rangle \d t.
\end{align*}

Next, since $(\varphi_{\parallel}, \varphi_\oo) \in W^{1,2}(\R_+ ; L^2 (\R^{n} ; \C^{n+1}))$ with support in $[a,b]\times\R^n$, a standard regularisation (cf. Lemma~\ref{lem:testdense} in the case $V\equiv 0$) gives a sequence of functions in  $\test([\tfrac{1}{2}a,2b]\times\mathbb{R}^n)$ that converges to $(\varphi_{\parallel}, \varphi_\oo)$ in $W^{1,2}(\R_+,L^2 (\R^n ; \C^{n+1}))$. This provides a sequence of test functions which combined with Lemma~\ref{lem:pusol} justifies integrating-by-parts, with respect to $t$, to obtain
\begin{align*}
\int_0^{\infty} \langle
\begin{bmatrix}
F_{\parallel} \\
F_\oo
\end{bmatrix}
, \begin{bmatrix}
\p \varphi_{\parallel} \\
\p \varphi_\oo
\end{bmatrix}
\rangle \d t 
= \int_0^{\infty} \langle \gradVp  u , \p \begin{bmatrix}
\varphi_{\parallel} \\
\varphi_\oo\end{bmatrix} \rangle \d t 
= -\int_0^{\infty} \langle \p (\gradVp u) ,  \begin{bmatrix}
\varphi_{\parallel} \\
\varphi_\oo\end{bmatrix} \rangle \d t,
\end{align*}
where the required convergence is guaranteed because $\gradVp u$ is in $L^2_{\Loc} (\R_+ ; L^2 (\R^n ; \C^{n+1}))$, and hence $\p(\gradVp u)=\gradVp(\p u)$ is in $L^2_{\Loc} (\R_+ ; L^2 (\R^n ; \C^{n+1}))$ with 
\[
\int_{\tfrac{1}{2}a}^{2b} \|\gradVp(\p u)\|_{L^2 (\R^n ; \C^{n+1})}^2 \d t
\lesssim \frac{1}{a^2} \int_{\tfrac{1}{8}a}^{8b} \|\nabla u\|_{L^2 (\R^n ; \C^{n+1})}^2 \d t
<\infty
\]
by Lemma~\ref{lem:pusol}. Finally, since $(\varphi_{\parallel}(t), \varphi_\oo(t))
\in\D(\gradV^*)$, and algebra shows that $(B F)_{\perp}(t) = (\p u)(t)$, for almost every $t\in\R_+$, we obtain
\begin{align*}
\int_0^{\infty} \langle
\begin{bmatrix}
F_{\parallel} \\
F_\oo
\end{bmatrix}
, \begin{bmatrix}
\p \varphi_{\parallel} \\
\p \varphi_\oo
\end{bmatrix}
\rangle \d t 
&= -\int_0^{\infty} \langle \gradVp (\p u) , \begin{bmatrix}
\varphi_{\parallel} \\
\varphi_\oo\end{bmatrix} \rangle \d t \\
& = \int_0^{\infty} \langle \p u , -(\gradVp)^* \begin{bmatrix}
\varphi_{\parallel} \\
\varphi_\oo\end{bmatrix} \rangle \d t \\
& = \int_0^{\infty} \langle (B F)_{\perp} , (D \varphi)_{\perp} \rangle \d t.
\end{align*}
Altogether, we have $
\int_0^{\infty} \langle F , \p \varphi \rangle \d t = \int_0^{\infty} \langle B F , D \varphi \rangle \d t$, so $\p F + DB F = 0$ in $\R_+$, as required.

We now prove that (2) holds. Suppose that $\p F + DB F = 0$  in $\R_+$ and that $F \in L^2_{\Loc} (\R_+ ; \Hil)$. If $t \in \R_+$, then Lemma~\ref{lem:RanD} and \eqref{eq:dotVinj} show that there exists $g_t$ in $\Wvo^{1,2} (\R^n) \subseteq L^2_{\Loc}(\R^n)$ such that
\begin{equation}\label{eq:Fg}
 \begin{bmatrix} F_{\parallel} (t,x) \\ F_\oo (t,x) \end{bmatrix}
 = \begin{bmatrix} \gradx g_t (x) \\ \V g_t (x) \end{bmatrix}
\end{equation}
for all $x\in\R^n$. We set $g(t,x) \coloneqq g_t (x)$ for all $(t,x)\in\R^{n+1}_+$. Now consider $c_0\in(0,\infty)$, to be chosen later, and define $u$ in $L^2_{\Loc}(\R^{n+1}_+;\C^{n+2})$, since $F \in L^2_{\Loc} (\R_+ ; \Hil)$ and $g_{c_0}\in L^2_{\Loc}(\R^n)$, by
\begin{equation}\label{eq:udef}
u (t,x) \coloneqq \int_{c_0}^t (BF)_{\perp} (s,x) \d s + g(c_0,x)
\end{equation}
for all $(t,x)\in\R^{n+1}_+$ (as usual $\int_{c_0}^{c_0}\coloneqq 0$ and $\int_{c_0}^t\coloneqq -\int_{t}^{c_0}$ when $t\in(0,c_0)$). For each $x \in \R^n$, Lebesgue's differentiation theorem guarantees that the function $t\mapsto u(t,x)$ is differentiable almost everywhere on $\R_+$ with $(\p u)(t,x) = (B F)_{\perp} (t , x)$ for almost every $t\in\R_+$.

Now consider test functions $\psi$ in $\mathcal{C}_c^{\infty} (\R_+)$,  $\eta$ in $\mathcal{C}_c^{\infty} (\R^n ; \C^{n+1})$ and $\varphi$ in $\mathcal{C}_c^{\infty} (\R^{n+1}_+ ; \C^{n+2})$ given by $\varphi (t,x) \coloneqq (0 , \psi (t) \eta (x))$ for all $(t,x)\in\R^{n+1}_+$. Observe that 
\[
\int_0^{\infty} \langle F (t) , (\p \varphi) (t) \rangle \d t = \int_0^{\infty} \int_{\R^n} \begin{bmatrix} F_{\parallel} (t,x) \\ F_\oo (t,x) \end{bmatrix} \cdot \overline{\eta (x) (\p \psi)(t)} \d x \d t.
\]
Meanwhile, since $\p u = (B F)_{\perp}$, we have
\begin{align*}
    \int_0^{\infty} \langle (B F) (t) , (D \varphi) (t) \rangle \d t 
    & = -\int_0^{\infty} \int_{\R^n} (B F)_{\perp} (t,x) \overline{\psi (t) (\gradV^* \eta) (x)} \d x \d t \\
    & = -\int_{\R^n} \left( \int_0^{\infty} (\p u) (t,x) \overline{\psi (t)} \d t \right) \overline{(\gradV^* \eta) (x)} \d x \\ 
    & = \int_{\R^n} \left( \int_0^{\infty} u (t,x) \overline{(\p \psi) (t)} \d t \right) \overline{(\gradV^* \eta) (x)} \d x.
\end{align*}
Therefore, since $\p F + DB F = 0$  in $\R_+$, we deduce that
\begin{align*}
    \int_0^{\infty} \left( \int_{\R^n}  \begin{bmatrix} F_{\parallel} (t,x) \\ F_\oo (t,x) \end{bmatrix} \cdot \overline{\eta (x)} \d x \right) \overline{(\p \psi)(t)} \d t
    = -\int_0^{\infty} \left( \int_{\R^n} (BF)_{\perp} (t,x) \overline{\gradV^* \eta (x)} \d x \right) \overline{\psi (t)} \d t
\end{align*}
for all $\psi \in \mathcal{C}_c^{\infty} (\R_+)$. This shows that the function $\displaystyle t\mapsto \int_{\R^n} \begin{bmatrix} F_{\parallel} (t,x) \\ F_\oo (t,x) \end{bmatrix} \cdot \overline{\eta (x)} \d x$ is weakly differentiable on $\R_+$ with
\[
\p \left( \int_{\R^n}  \begin{bmatrix} F_{\parallel} (t,x) \\ F_\oo (t,x) \end{bmatrix} \cdot \overline{\eta (x)} \d x \right) = \int_{\R^n} (BF)_{\perp} (t,x) \overline{\gradV^* \eta (x)} \d x
\]
for all $t\in\R_+$. Furthermore, we deduce that
\[
\int_0^{\infty} \left( \int_{\R^n}  \begin{bmatrix} F_{\parallel} (t,x) \\ F_\oo (t,x) \end{bmatrix} \cdot \overline{\eta (x)}- u (t,x) \overline{(\gradV^* \eta) (x)} \d x \right) \overline{(\p \psi) (t)} \d t = 0
\]
for all $\psi \in \mathcal{C}_c^{\infty} (\R_+)$, so integrating by parts shows that there exists $C_0\in\R$ such that
\begin{align*}
C_0 &= \int_{\R^n} \left( \begin{bmatrix} F_{\parallel} (t,x) \\ F_\oo (t,x) \end{bmatrix} \cdot \overline{\eta (x)} - u (t,x) \overline{(\gradV^* \eta) (x)}\right) \d x \\
&= \int_{\R^n} \left(  \begin{bmatrix} F_{\parallel} (t,x) \\ F_\oo (t,x) \end{bmatrix} \cdot \overline{\eta (x)} - \begin{bmatrix} F_{\parallel} (c_0,x) \\ F_\oo (c_0,x) \end{bmatrix} \cdot \overline{\eta (x)} - \left( \int_{c_0}^t (BF)_{\perp} (s,x) \d s \right) \overline{(\gradV^* \eta) (x)}\right) \d x
\end{align*}
for almost every $t\in\R_+$, where we used \eqref{eq:Fg}-\eqref{eq:udef} to obtain the second equality.

The function $\displaystyle t\mapsto \int_{\R^n} \begin{bmatrix} F_{\parallel} (t,x) \\ F_\oo (t,x) \end{bmatrix} \cdot \overline{\eta (x)} \d x$ is continuous almost everywhere on $\R_+$, since any weakly differentiable function is equal almost everywhere to a locally absolutely continuous function (see, for instance, \cite[Section~7.3]{GT77}). Therefore, we now choose a point $c_0\in\R_+$ at which this function is continuous. Meanwhile, the function $t \mapsto \int_{c_0}^t \int_{\R^n}  (BF)_{\perp} (s,x) \overline{(\gradV^* \eta) (x)} \d x \d s$ is continuous on $\R_+$, since $F \in L^2_{\Loc} (\R_+ ; \Hil)$. Therefore, for each $\varepsilon >0$, there exists $\delta \in (0 , c_0)$ such that 
\[
   \left| \dashint_{c_0 - \delta}^{c_0 + \delta} \left( \int_{\R^n}  \begin{bmatrix} F_{\parallel} (t,x) \\ F_\oo (t,x) \end{bmatrix} \cdot \overline{\eta (x)} \d x \right) \d t - \int_{\R^n} \begin{bmatrix} F_{\parallel} (c_0,x) \\ F_\oo (c_0,x) \end{bmatrix} \cdot \overline{\eta (x)} \d x \right| < \varepsilon
\]
and
\[
   \left| \dashint_{c_0 - \delta}^{c_0 + \delta} \int_{c_0}^t \int_{\R^n} (BF)_{\perp} (s,x) \overline{(\gradV^* \eta) (x)} \d x \d s \d t \right| < \varepsilon.
\]
This allows us to deduce that $C_0 = 0$, since $|C_0| = |\dashint_{c_0 - \delta}^{c_0 + \delta} C_0 \d s| < 2\varepsilon$.

This proves that $(F_{\parallel}, F_\mu) = \gradVp u$, and as $(BF)_{\perp} = \p u$, we have $F_{\perp} = (A\nabla u)_\perp = \partial_{\nu_A} u$ by \eqref{eq:Ahatdef}, so $F=\gradAV u$. This in turn shows that $\gradV u \in L^2_{\Loc}(\R_+;L^2(\R^n;\C^{n+2}))$, so $u\in \Wv^{1,2}_{\Loc}(\R^{n+1}_+)$. Now consider test functions $\psi$ in $\mathcal{C}_c^{\infty} (\R^{n+1}_+)$ and $\varphi \coloneqq  (\psi,0)$ in $\mathcal{C}_c^{\infty} (\R^{n+1}_+ ; \C^{n+2})$. Observe that, since $\int_0^{\infty} \langle F , \p \varphi \rangle \d t = \int_0^{\infty} \langle B F , D \varphi \rangle \d t$, we have
\begin{align*}
    \int_0^{\infty} \int_{\R^n} (A\nabla u)_\perp \, \overline{\p \psi} \d x \d t & 
    = - \int_0^{\infty} \int_{\R^n} ((B F)_{\parallel} \cdot \overline{\gradx \psi} + (BF)_\oo \V\overline{\psi}) \d x \d t \\
    & = - \int_0^{\infty} \int_{\R^n} ((A \nabla u )_{\parallel} \cdot \overline{\gradx \psi} + a V u \overline{\psi} ) \d x \d t,
\end{align*}
where $(B F)_{\parallel} = (A \nabla u)_{\parallel}$ and $(BF)_\oo = a e^{i\arg V} \V u$ by \eqref{eq:Ahatdef}. Together, we have $\HAaV  u = 0$, so the proof is complete.
\end{proof}


\subsection{Semigroup Solutions}\label{ssect:GWP}


Throughout this section, suppose that $n\geq3$ and $V\in L^1_{\Loc}(\R^n)$. The operator $DB$ is defined as in~\eqref{eq:Ddef} with $B$ satisfying the bound (but not necessarily the component structure) in~\eqref{eq:bound on R(D)} and the accretivity in~\eqref{eq:elliptic on R(D)} with angle of accretivity $\omega(B)\in[0,\frac{\pi}{2})$. We rely on the holomorphic functional calculus for bisectorial operators outlined in Appendix~\ref{sec:app}. The theory of analytic semigroups in Proposition~\ref{prop:SGfromFC} allows us to generate solutions to the Cauchy--Riemann type system~\eqref{eqn:First-Order Equation} for the operator $DB$. Theorem~\ref{Thm:First order solutions are semigroups} will show that when $DB$ has a bounded $H^{\infty} (S_{\omega(B)})$-functional calculus, there is a Hardy-type splitting of $\overline{\Ran(D)}$ into spectral subspaces which characterise the trace spaces for initial value problems given by~\eqref{eqn:First-Order Equation}.

We start by defining the following holomorphic functions on $S_{\frac{\pi}{2}}^o$ (see also $[\cdot]$ in~\eqref{eq:[]def}):
\[
\chi^{\pm} (z) \coloneqq
\begin{cases}
1, \quad & \text{if} \, \pm \Re (z) > 0 \\
0, \quad & \text{if} \, \pm \Re(z) < 0
\end{cases},
\quad
\sgn(z) \coloneqq \chi^+(z) - \chi^-(z)
\quad\text{and}\quad
[z] \coloneqq z \sgn(z).
\]
Now consider the Hilbert space $\Hil \coloneqq \overline{\Ran(D)}=\overline{\Ran(DB)}$ and let $DB|_\Hil$ denote the injective part of $DB$ (see Remark~\ref{rem:injpart}) whereby
\[
DB|_\Hil: \D(DB|_\Hil) \subseteq \Hil \to \Hil
\quad\text{ and }\quad
DB|_\Hil u \coloneqq  DBu
\]
for all $u\in \D(DB|_\Hil)\coloneqq \D(DB)\cap\Hil$. The adjoint operator $(DB|_\Hil)^*$ is  defined with respect to $\mathcal{H}=\overline{\Ran(D)}$, which has orthogonal complement $\Hil^\perp = N(D)$ in $L^2(\R^n;\C^{n+2})$, hence
\begin{equation}\label{eq:adwP}
(DB|_\Hil)^*u=\mathbb{P}_{\Hil}B^*Du
\end{equation}
for all $u\in\D(D) \cap \Hil = \D(B^*D) \cap \Hil = \D((DB|_\Hil)^*)$, as $\langle u, DB v \rangle =\langle u, DB\mathbb{P}_\Hil v \rangle = \langle \mathbb{P}_{\Hil}B^*D u, v \rangle$ when $v\in\D(DB) \cap \Hil$. In what follows, $DB$ will  denote $DB|_\Hil$ and thus $(DB)^*=\mathbb{P}_{\Hil}B^*D|_\Hil$.

Now suppose that $DB\coloneqq DB|_\Hil$ has a bounded $H^{\infty}(S_{\omega(B)})$-functional calculus and let $T$ denote either $DB$ or $(DB)^*$. The operators $\chi^{\pm}(T)$ and $\sgn(T)$ then belong to $\mathcal{L}(\Hil)$, whilst Proposition~\ref{prop:SGfromFC} shows that $[T]$ is a sectorial operator of type $S_{\omega(B)+}$ on $\Hil$, and (2) in Theorem~\ref{thm:FfcDef} implies that $\sgn(T)T\subseteq[T]=T\sgn(T)$ with
\begin{equation}\label{eq:T[T]}
\D(T) = \D(\sgn(T)T) = \D([T])\cap \D(T) = \D(\sgn(T)[T]) = \D([T]). 
\end{equation}
It follows that $\sgn(T)T=[T]=T\sgn(T)$ with
\begin{equation}\label{eq:[DB]DB}
[T] \chi^{\pm}(T) f 
= T (\chi^+(T)-\chi^-(T))\chi^{\pm}(T) f 
= \pm T \chi^{\pm}(T) f
\end{equation}
for all $f \in \D(T) = \D([T])$.

We can then define the Hardy-type spectral subspaces $\Hil^{\pm}_{DB} \coloneqq \{ \chi^{\pm}(DB) f : f \in \Hil \}$ to obtain the topological direct sum decomposition
\begin{equation}\label{eq:tds}
\Hil = \Hil^+_{DB} \oplus \Hil^-_{DB},
\end{equation}
since $f = \chi^+(DB) f + \chi^-(DB) f$ with 
$\norm{f}{2}{} \eqsim \norm{\chi^+(DB) f}{2}{} + \norm{\chi^-(DB) f}{2}{}$ for all $f\in \Hil$. The operators $\chi^{\pm}(DB)$ are thus known as the Hardy-type projections for $DB$. There is an analogous decomposition for $(DB)^*$ but this is not needed for our purposes.

We collate below the essential properties of the solutions to~\eqref{eqn:First-Order Equation} provided by the analytic semigroup of operators $e^{-t [DB]}\coloneqq (e^{-t[\cdot]})(DB)$, which are defined for all $t\in[0,\infty)$ using the $\mathcal{F}$-functional calculus for $DB$ as in Proposition~\ref{prop:SGfromFC}.

\begin{prop} \label{Prop:Semigroup solves first order equation}
Suppose that $DB\coloneqq DB|_\Hil$ has a bounded $H^{\infty} (S_{\omega(B)})$-functional calculus. If $f \in \Hil^+_{DB}$ and $F(t)=e^{-t [DB]}f$ for all $t\in\R_+$, then $F \in L^2_{\Loc}(\R_+;\Hil^+_{DB}) \cap \mathcal{C}^{\infty} (\R_+;\Hil)$ with $\p F + DB F = 0$ in $\R_+$, $\lim_{t \to 0^+} \|F(t)-f\|_2=0$ and $\lim_{t \to \infty} \|F(t)\|_2 = 0$, as well as
\begin{equation}\label{eq:L2W}
    \lim_{t \to 0^+} \dashint_t^{2t} \|F (s) - f \|_2^2 \d s
    = 0
    \quad\text{ and }\quad
    \lim_{t \to \infty} \dashint_t^{2t} \| F (s) \|_2^2 \d s
    = 0.
\end{equation}
Moreover, it holds that
\[
\|f\|_2^2 \leq \sup_{t>0} \dashint_t^{2t} \| F (s) \|_2^2 \d s \leq \sup_{t>0} \|F (t)\|_2^2 \leq C \|f\|_2^2,
\]
where $C\coloneqq \sup_{t>0}\|e^{-t[DB]}\|_{\mathcal{L}(\Hil)}^2 < \infty$.
\end{prop}

\begin{proof}
Theorem~\ref{thm:FfcDef} shows that $e^{-t [DB]}\chi^+(DB)=\chi^+(DB)e^{-t [DB]}$, since $\chi^+(DB)$ and $e^{-t [DB]}$ are in $\mathcal{L}(\Hil)$, so Proposition~\ref{prop:SGfromFC} implies that $e^{-t [DB]}f = e^{-t [DB]}\chi^+(DB)f = \chi^+(DB) e^{-t [DB]}f$ belongs to $\Hil_{DB}^+\cap\D([DB]) = \Hil_{DB}^+\cap\D(DB)$ and $\mathcal{C}^{\infty}(\R_+;\Hil)$ with
\[
 \p (e^{-t[DB]}f)
=-[DB] e^{-t [DB]}f
=- DB e^{-t [DB]}f,   
\]
where we used \eqref{eq:T[T]} and \eqref{eq:[DB]DB} for $DB$. This proves that  $F \in L^2_{\Loc}(\R_+,\Hil_{DB}^+)$ and $\p F + DB F = 0$ on $\R_+$, whilst $\lim_{t \to 0^+} \|F (t)-f\|_2=0$ and $\lim_{t \to \infty} \|F(t)\|_2 = 0$ by Proposition~\ref{prop:SGfromFC}, which also implies \eqref{eq:L2W}. Moreover, the $L^2$-convergence implies that 
\[
\|f\|_2^2 \leq \sup_{t > 0} \dashint_t^{2t} \| F(s) \|_2^2 \d s \leq \sup_{t > 0} \|F (t)\|_2^2,
\]
as $\|f\|_2^2 \leq \dashint_\delta^{2\delta} \| f - F(s) \|_2^2 \d s + \sup_{t>0} \dashint_t^{2t} \| F(s) \|_2^2 \d s$ for all $\delta>0$. Finally, as $DB$ has a bounded $H^{\infty} (S_{\omega(B)})$-functional calculus, for each $\mu\in(\omega(B),\tfrac{\pi}{2})$, there exists $C_\mu\in(0,\infty)$ such that 
\[
C\coloneqq \sup_{t > 0} \|e^{-t [DB]}\|_{\mathcal{L}(\Hil)}
\leq C_\mu\sup\{|e^{-t z}| : z\in S_{\mu+}^o,\, t\in\R_+ \}
= C_\mu < \infty,
\]
which completes the proof.
\end{proof}

The following lemma extends the result obtained by Auscher--Axelsson in \cite[Proposition~4.4]{AA11} for the case $V\equiv0$. We need to avoid the mollification argument used therein, however, because it is not clear if $\mathcal{C}^\infty_c(\R^n;\C^{n+1})$ is dense in $\D(\gradV^*)$. This is achieved by requiring weak solutions to satisfy \eqref{eq:weak1st} for the larger class of test functions in $L^2(\mathbb{R}_+;\D(D)) \cap \test(\R_+ ; L^2 (\R^{n} ; \C^{n+2}))$. 

\begin{lem} \label{lem:weak derivatives are 0}
Suppose that $DB\coloneqq DB|_\Hil$ has a bounded $H^{\infty} (S_{\omega(B)})$-functional calculus. If $t\in\R_+$ and $F \in L^2_{\Loc} (\R_+ ; \Hil)$ with $\p F + DB F = 0$ in $\R_+$, then 
\[
    \int_0^t \eta_+' (s) e^{-(t-s) [DB]} \chi^+(DB) F(s) \d s = 0 = \int_t^{\infty} \eta_-' (s) e^{-(s-t) [DB]} \chi^-(DB) F(s) \d s
\]
for any $[0,\infty)$-valued functions $\eta_+\in\mathcal{C}_c^{\infty}((0,t))$ and $\eta_-\in\mathcal{C}_c^{\infty}((t,\infty))$ with derivatives $\eta_\pm'$.
\end{lem}

\begin{proof}
Suppose that $\psi \in \Hil$ and $t\in\R_+$. Now $\varphi(s)\coloneqq 0$ for all $s \geq t$ and 
\[
\varphi(s) \coloneqq \eta_+ (s) \left( e^{-(t-s)[DB]}\chi^+(DB) \right)^* \psi
=\eta_+ (s) e^{-(t-s) [(DB)^*]} \chi^+ ((DB)^*) \psi
\]
for all $s\in(0,t)$. It is important to recall that here the operator adjoints are with respect to the Hilbert space $\mathcal{H}=\overline{\Ran(D)}$, which has orthogonal complement $\Hil^\perp = N(D)$ in $L^2(\R^n;\C^{n+2})$, and that $(DB)^*=\mathbb{P}_{\Hil}B^*D|_\Hil$ by \eqref{eq:adwP}. Moreover, the equality above follows from the properties of the functional calculus for $(DB)^*$ from Theorem~\ref{thm:FfcDef} and~\eqref{eq:adjFC}. These considerations and the semigroup properties in Proposition~\ref{prop:SGfromFC} for $(DB)^*$ also ensure that
\[
\varphi\in L^2(\mathbb{R}_+;\D(D)) \cap \test(\R_+ ; L^2 (\R^{n} ; \C^{n+2})),
\]
and since $F \in L^2_{\Loc} (\R_+ ; \Hil)$ satisfies $\p F + DB F = 0$ in $\R_+$, we have 
\[
    \int_0^t \langle F(s), \partial_s \varphi(s)  \rangle \d s
    = \int_0^t \langle BF(s), D \varphi(s) \rangle \d s
    = \int_0^t \langle F(s), \mathbb{P}_\Hil B^*D \varphi(s) \rangle \d s.
\]
The semigroup properties in Proposition~\ref{prop:SGfromFC} and \eqref{eq:[DB]DB} for $(DB)^*=\mathbb{P}_{\Hil}B^*D|_\Hil$ also show that
\begin{align*}
    \int_0^t & \langle F(s) , \partial_s \varphi(s) \rangle \d s \\
    & = \int_0^t \eta_+' (s) \left\langle F(s) ,\left( e^{-(t-s) DB} \chi^+(DB) \right)^* \psi \right\rangle \d s \\
    & \qquad \qquad + \int_0^t \eta_+ (s) \left\langle F(s) , \partial_s \left( e^{-(t-s) [(DB)^*]} \chi^+ ((DB)^*) \psi \right) \right\rangle \d s \\
    & = \int_0^t  \eta_+' (s) \langle e^{-(t-s) DB} \chi^+(DB) F(s) , \psi \rangle \d s \\
    & \qquad \qquad + \int_0^t \eta_+ (s) \langle F(s) , \mathbb{P}_{\Hil}B^*D e^{-(t-s) [(DB)^*]} \chi^+ ((DB)^*) \psi \rangle \d s \\
    & = \int_0^t  \eta_+' (s) \langle e^{-(t-s) DB} \chi^+(DB) F(s) , \psi \rangle \d s + \int_0^t \langle F(s) , \mathbb{P}_\Hil B^*D \varphi(s) \rangle \d s.
\end{align*}
Therefore, by Fubini's Theorem and rearranging above, we conclude that
\[ 
    \left\langle \int_0^t \eta_+' (s) e^{-(t-s) [DB]} \chi^+(DB) F(s) \d s , \psi \right\rangle = \int_0^t  \eta_+' (s) \langle e^{-(t-s) DB} \chi^+(DB) F(s) , \psi \rangle \d s 
    = 0
\]
for all $\psi \in \mathcal{H}$, and since $\Hil^\perp \cap \Hil = N(D) \cap \overline{\Ran(D)} = \{0\}$, it follows that
\[
\int_0^t \eta_+' (s) e^{-(t-s) [DB]} \chi^+(DB) F(s) \d s = 0.
\]
Using the test function $\widetilde{\varphi}(s) \coloneqq \eta_- (s) \left( e^{-(s-t)[DB]} \chi^-(DB) \right)^* \psi$, a similar argument shows that 
\[
    \int_t^{\infty} \eta_-' (s) e^{-(s-t) [DB]} \chi^-(DB) F(s) \d s = 0.
\]
This completes the proof.
\end{proof}

We now prove that the Hardy-type spaces for $DB$ in \eqref{eq:tds} characterise the trace spaces for initial value problems given by~\eqref{eqn:First-Order Equation}. The preceding lemma makes this a straightforward application of the abstract framework developed by Auscher--Axelsson in~\cite{AA2011Remarks,AA11}.

\begin{thm} \label{Thm:First order solutions are semigroups}
Suppose that $DB\coloneqq DB|_\Hil$ has a bounded $H^{\infty} (S_{\omega(B)})$-functional calculus. If $F\in L_{\Loc}^2 (\R_+;\Hil)$ with $\p F + DB F = 0$ in $\R_+$ and 
$
    \sup_{t > 0} \dashint_t^{2t} \| F (s) \|_2^2 \d s < \infty,
$
then there exists $f \in \Hil_{DB}^+$ such that $\lim_{t \to 0^+} \|F (t)-f\|_2=0$ and $F (t) = e^{-t [DB]} f$.
\end{thm}

\begin{proof}
For each $\varepsilon > 0$, we construct functions $\eta_{\varepsilon}^{\pm}$ as in~\cite[Section~6]{AA11}. Let $\eta^0 \colon [0,\infty) \to [0,1]$ denote a smooth function that is supported in $[1,\infty)$ and satisfies $\eta^0 (t) = 1$ for all $t \in (2,\infty)$, set $\eta_{\varepsilon} (t) \coloneqq \eta^0 (t/\varepsilon) (1 - \eta^0 (2 \varepsilon t))$ and
\[
\eta_{\varepsilon}^{\pm} (s,t) \coloneqq \eta^0 \left( \pm (t - s)\varepsilon \right) \eta_{\varepsilon} (t) \eta_{\varepsilon} (s)
\]
for all $s,\,t\in(0,\infty)$. The functions  $\eta_{\varepsilon}^+$ are compactly supported in $\{(s,t) \in \R^2 : 0 < s < t < \infty\}$ and $\lim_{\varepsilon\to0^+} \eta_{\varepsilon}^+=\mathds{1}_{\{ (s,t) \in \R^2 : 0 < s < t < \infty \}}$.  Similarly, the functions $\eta_{\varepsilon}^-$ are compactly supported in $\{(s,t) \in \R^2 : 0 < t < s < \infty \}$ and $\lim_{\varepsilon\to0^+} \eta_{\varepsilon}^-=\mathds{1}_{\{ (s,t) \in \R^2 : 0 < t < s < \infty \}}$. If $t\in\R_+$ and $F \in L^2_{\Loc} (\R_+ ; \Hil^+_{DB})$ with $\p F + DB F = 0$ in $\R_+$, then Lemma~\ref{lem:weak derivatives are 0} shows that
\[
-\int_0^t (\p[s] \eta_{\varepsilon}^+) (t,s) e^{-(t-s) [DB]} \chi^+(DB) F(s) \d s + \int_t^{\infty} (\p[s] \eta_{\varepsilon}^-) (t,s) e^{-(s-t) [DB]} \chi^-(DB) F(s) \d s = 0.
\]
The result then follows as in the proof of \cite[Theorem 8.2(i)]{AA11} with $\mathcal{E}\equiv 0$ (cf. ~\cite[Proposition~2.2]{AA2011Remarks} with $\beta=0$, $f\equiv0$ and $A=\pm[DB]$).
\end{proof}


\subsection{Boundary Isomorphisms} \label{sec:ADD}


Throughout this section, we suppose that $D$ and $B$ satisfy either of the hypotheses in Theorem~\ref{thm:QE}, so the injective part $DB\coloneqq DB|_\Hil$ has a bounded $H^{\infty} (S_{\omega(B)})$-functional calculus on $\mathcal{H}=\overline{\Ran(D)}$ by Theorem~\ref{thm:Quadratic estimates and functional calculus}. We can then use Lemma~\ref{lem:RanD} and the topological decomposition $\Hil = \Hil^+_{DB} \oplus \Hil^-_{DB}$ from \eqref{eq:tds} to define boundary mappings $\Phi_N \colon \Hil_{DB}^+ \to L^2 (\R^n)$ and $\Phi_R \colon \Hil_{DB}^+ \to \{ \gradV g : g \in \Wvo^{1,2} (\R^n) \}$ by
\begin{equation} \label{eqn:linear mappings}
\Phi_N (f) \coloneqq  f_{\perp}
\quad\text{ and }\quad
\Phi_R (f) \coloneqq  (f_{\parallel}, f_\oo )
\end{equation}
for all $f\in \Hil_{DB}^+$. We will see in Section~\ref{Sec:Return to the Second-Order Equation} that the Neumann $\Neu$ and Regularity $\Reg$ problems are well-posed precisely when these mappings are isomorphisms. For now, we combine the analytic perturbation result in Theorem~\ref{thm:Lipschitz Estimates} with a method of continuity to show that such properties are stable under small $L^{\infty}$-perturbations of the coefficients $B$. We then use this result to prove that the boundary mappings are isomorphisms when the coefficients $A$ are Hermitian and the product $aV$ is real-valued in Proposition~\ref{prop:Self-Adjoint Isomorphisms}. We use the analogous result for block coefficients in Proposition~\ref{prop:Block-Type Isomorphisms}, which does not rely on perturbation and even allows for complex-valued potentials, as the basis for the perturbation in the Hermitian case. The proof proceeds via Rellich-type estimates in both cases.

\begin{thm}\label{thm:absper}
Suppose that $D$ and $B$ satisfy either of the hypotheses in Theorem~\ref{thm:QE}. If $\widetilde{B}$ is in $L^{\infty} (\R^n ; \mathcal{L} (\C^{n+2}))$ with the component structure in \eqref{eq:bound on R(D)}, then the following properties hold:
\begin{enumerate}
\item If $\Phi_N \colon \Hil_{DB}^+ \to L^2 (\R^n)$ is an isomorphism, then there exists $\varepsilon\in(0,\kappa)$, depending only on $n$, $\kappa$, $K$, $\upsilon$ and the constant
\[
\|(\Phi_N)^{-1}\| \coloneqq  \inf\{ C \in (0,\infty): \|u\|_2 \leq C \|u_\perp\|_2 \text{ for all } u\in \Hil_{DB}^+\},
\]
such that $\Phi_N \colon \Hil_{D\widetilde{B}}^+ \to L^2 (\R^n)$ is an isomorphism whenever $\|\widetilde{B} - B\|_{\infty} < \varepsilon$.
\item If $\Phi_R \colon \Hil_{DB}^+  \to \{ \gradV g : g \in \Wvo^{1,2} (\R^n) \}$ is an isomorphism, then there exists $\varepsilon\in(0,\kappa)$, depending only on $n$, $\kappa$, $K$,  $\upsilon$ and the constant
\[
\|(\Phi_R)^{-1}\|\coloneqq \inf\{ C \in (0,\infty): \|u\|_2 \leq C \|(u_\|,u_\mu)\|_2 \text{ for all } u\in \Hil_{DB}^+\},
\]
such that $\Phi_R:\! \Hil_{D\widetilde{B}}^+ \to \{ \gradV g:\! g \!\in\! \Wvo^{1,2} (\R^n) \}$ is an isomorphism whenever $\|\widetilde{B} - B\|_{\infty} < \varepsilon$.
\end{enumerate}
Observe that $D\widetilde{B}$ has a bounded $H^\infty(S_{\omega})$-functional calculus whenever $\|\widetilde{B} - B\|_{\infty} < \varepsilon<\kappa$, where $\omega = \cos^{-1}((\kappa-\varepsilon)/(K+\varepsilon)) \in [\omega(B), \tfrac{\pi}{2})$, by Theorem~\ref{thm:Lipschitz Estimates}, so $\Hil_{D\widetilde{B}}^+$ is defined as in \eqref{eq:tds}.
\end{thm}

\begin{proof}
We first suppose that $\|\widetilde{B} -  B\|_{\infty} < \varepsilon\leq\tfrac{1}{2}\kappa$, where $\varepsilon$ will be chosen later. If $E^+\coloneqq \chi^+(DB)$ and  $\widetilde{E}^+\coloneqq \chi^+(D\widetilde{B})$, then the analytic perturbation result in Theorem~\ref{thm:Lipschitz Estimates}, applied with  $\delta=\tfrac{1}{2}\kappa$ and $\mu= \cos^{-1}(\tfrac{1}{4}\kappa/K)$, shows that there exists $C_0\in(0,\infty)$ such that
\begin{equation}\label{eq:mc0N}
\|\widetilde{E}^+ u - E^+ u\|_2=
\|\chi^+(D\widetilde{B}) u - \chi^+(DB) u\|_2
\leq C_0 \|\widetilde{B} - B\|_{\infty} \|u\|_2
< \varepsilon C_0 \|u\|_2
\end{equation}
for all $u\in\Hil$, as $\|\widetilde{B} - B\|_{\infty} <\varepsilon \leq\tfrac{1}{2}\kappa=\delta<\kappa$ and $\cos^{-1}((\kappa-\delta)/(K+\delta))<\mu<\tfrac{\pi}{2}$. Meanwhile, the functional calculus bounds for $D\widetilde{B}$ provide $C_1\in(0,\infty)$ such that
\begin{equation}\label{eq:mcmuN}
\|\widetilde{E}^+u\|_2 =
\|\chi^+(D\widetilde{B})u\|_2
\leq C_1 \|u\|_2
\end{equation}
for all $u\in\Hil$. Moreover, the constants $C_0$ and $C_1$ depend only on $n$,  $\kappa$, $K$ and  $\upsilon$, as the choices for $\delta$ and $\mu$ depend only on $\kappa$ and $K$.

Next, observe that if $\varepsilon C_0C_1<1$, then $\widetilde{E}^+ \colon \Hil_{DB}^+ \to \Hil_{D\widetilde{B}}^+$ is surjective. This is because \eqref{eq:mcmuN} and \eqref{eq:mc0N} imply that $\| \widetilde{E}^+(\widetilde{E}^+ - E^+) u \|_2 < \varepsilon C_0 C_1 \|u\|_2 < \|u\|_2$
for all $u\in\Hil$, so a Neumann series shows that $I - \widetilde{E}^+(\widetilde{E}^+ - E^+):\Hil \to\Hil$ is bijective, and as
$\widetilde{E}^+ E^+ = \widetilde{E}^+ (I - \widetilde{E}^+ (\widetilde{E}^+ - E^+))$, we obtain $\widetilde{E}^+ E^+ (I - \widetilde{E}^+ (\widetilde{E}^+ - E^+))^{-1}u = \widetilde{E}^+u = u$
for all $u\in \Hil_{D\widetilde{B}}^+$, which proves the claimed surjectivity. This mapping is actually bijective but we don't use this fact.

Now assume that $\Phi_N \colon \Hil_{DB}^+ \to L^2 (\R^n)$ is an isomorphism, so $C_N\coloneqq \|(\Phi_N)^{-1}\|\in(0,\infty)$ with
\begin{equation}\label{eq:mc1N}
\|(\Phi_N)^{-1}f\|_2 \leq C_N \|f\|_2
\quad\text{ and }\quad
\|u\|_2=\|(\Phi_N)^{-1}(u_\perp)\|_2 \leq C_N\|u_\perp\|_2
\end{equation}
for all $f\in L^2(\R^n)$ and $u\in \Hil_{DB}^+$. We see from \eqref{eq:mc0N} and \eqref{eq:mc1N} that
\[
\|u_\perp\|_2
=\|(E^+u)_\perp\|_2
\leq \|(E^+u)_\perp-(\widetilde{E}^+u)_\perp\|_2 + \|(\widetilde{E}^+u)_\perp\|_2
\leq \varepsilon C_0 C_N \|u_\perp\|_2 + \|(\widetilde{E}^+u)_\perp\|_2
\]
for all $u\in \Hil_{DB}^+$. We now choose $\varepsilon=\tfrac{1}{2}\min\{\kappa,(C_0C_1)^{-1},(C_0C_N)^{-1}\}$, so then $\varepsilon C_0C_N \leq \tfrac{1}{2}$ and $(1-\varepsilon C_0C_N)^{-1} \leq 2$, and the preceding estimates imply that
\begin{equation}\label{eq:mcF}
\|u\|_2
\leq C_N \|u_\perp\|_2
\leq 2 C_N \|(\widetilde{E}^+u)_{\perp}\|_2
= 2 C_N \|\Phi_N\widetilde{E}^+u\|_2
\end{equation}
for all $u\in \Hil_{DB}^+$.

We now use a method of continuity (see, for instance, \cite[Theorem~5.2]{GT77}) to prove that $\Phi_N \colon \Hil_{D\widetilde{B}}^+ \to L^2 (\R^n)$ is a bijection. If $v \in \Hil_{D\widetilde{B}}^+$ and $\Phi_N v = 0$, then the surjectivity of $\widetilde{E}^+ \colon \Hil_{DB}^+ \to \Hil_{D\widetilde{B}}^+$ above gives $u\in \Hil_{DB}^+$ such that $v=\widetilde{E}^+u$, so $\Phi_N \widetilde{E}^+ u = 0$ and \eqref{eq:mcF} implies that $u = 0$, hence $v = 0$ and $\Phi_N \colon \Hil_{D\widetilde{B}}^+ \to L^2 (\R^n)$ is injective. Meanwhile, if $f\in L^2(\R^n)$, then $T_f:\Hil_{DB}^+ \to \Hil_{DB}^+$ given by $T_fu\coloneqq (\Phi_N)^{-1}(f + (E^+u)_\perp-(\widetilde{E}^+u)_\perp)$ for all $u\in \Hil_{DB}^+$ is a contraction mapping, since $\|T_f(u-v)\|\leq \varepsilon C_0C_N \|u-v\|_2$ for all $u,\, v \in \Hil_{DB}^+$ by \eqref{eq:mc0N} and \eqref{eq:mc1N}, whilst $\varepsilon C_0C_N<1$. Therefore, there exists $u\in \Hil_{DB}^+$ such that $T_fu=u$, whence $(E^+ u)_{\perp} = u_{\perp} = f + (E^+ u)_{\perp}  - (\widetilde{E}^+ u)_{\perp}$, so $f=(\widetilde{E}^+u)_\perp$ and $\Phi_N \colon \Hil_{D\widetilde{B}}^+ \to L^2 (\R^n)$ is surjective.

This completes the proof of (1). An analogous argument proves (2) by replacing \eqref{eq:mc1N} with
\[
\|(\Phi_R)^{-1}(\gradV g)\|_2 \leq C_R \|\gradV g\|_2
\quad\text{ and }\quad
\|u\|_2=\|(\Phi_R)^{-1}((u_\|,u_\mu))\|_2 \leq C_R\|(u_\|,u_\mu)\|_2
\]
for all $g\in \Wvo^{1,2} (\R^n)$ and $u\in \Hil_{DB}^+$, when $\Phi_R \colon \Hil_{DB}^+ \to \{ \gradV g : g \in \Wvo^{1,2} (\R^n) \}$ is an isomorphism, where $C_R\coloneqq \|(\Phi_R)^{-1}\|\in(0,\infty)$.
\end{proof}

\subsubsection{Block Coefficients} \label{sec:well-posedness of Block Type Matrices}

We now prove that the boundary mappings $\Phi_N$ and $\Phi_R$ from \eqref{eqn:linear mappings} are isomorphisms when $A$ is block, in which case
\[
    \cA = \begin{bmatrix}
        \App & 0 & 0 \\
        0 & \Avv & 0 \\
        0 & 0 & a e^{i\arg V}
    \end{bmatrix}
    \quad \text{ and } \quad
    B = \widehat{\cA} = \begin{bmatrix}
        \App^{-1} & 0 & 0 \\
        0 & \Avv & 0 \\
        0 & 0 & a e^{i\arg V}
    \end{bmatrix}.
\]
We follow the strategy used by Auscher--Axelsson--McIntosh in \cite[Section~4.2]{AAMc10-2} to establish the Rellich-type estimates below.

\begin{prop} \label{prop:Block-Type Isomorphisms}
If $A$ is block, then $\|f_{\perp}\|_2 \eqsim \| (f_{\parallel} , f_\oo) \|_2$ for all $f\in \Hil_{DB}^+$, where the implicit constants depend only on $n$, $\lambda$,  $\Lambda$ and $\upsilon$. Moreover, the mappings $\Phi_N \colon \Hil_{DB}^+ \to L^2 (\R^n)$ and $\Phi_R \colon \Hil_{DB}^+ \to \{ \gradV g : g \in \Wvo^{1,2} (\R^n) \}$ are isomorphisms.
\end{prop}

\begin{proof}
The multiplication operators $N,\, \mathbb{P}_{\perp},\, \mathbb{P}_{\|,\oo} \colon \Hil \to \Hil$ are defined by the  matrices
\[
    N \coloneqq \begin{bmatrix}
        -1 & 0 & 0\\
        0 & I & 0 \\
        0 & 0 & 1
    \end{bmatrix},
    \quad
    \mathbb{P}_{\perp} \coloneqq \begin{bmatrix}
        1 & 0 & 0\\
        0 & 0 & 0 \\
        0 & 0 & 0
    \end{bmatrix}
    = \tfrac{1}{2} (I - N)
    \quad\text{and}\quad
    \mathbb{P}_{\|,\oo} \coloneqq \begin{bmatrix}
        0 & 0 & 0 \\
        0 & I & 0 \\
        0 & 0 & 1
    \end{bmatrix}
    = \tfrac{1}{2} (I + N).
\]
The block structure implies that $NDBN = -DB$, whilst $N^{-1} = N \in \mathcal{L}(H)$, so we have
\begin{equation} \label{eqn:Symmetry of E_DB when A is block}
    N \sgn(DB) = N \sgn(DB) NN = \sgn(NDBN) N = \sgn(-DB) N = - \sgn(DB) N
\end{equation}
by \eqref{eq:simFC}. If $f \in \Hil_{DB}^+$, then $f=\sgn(DB)f$ and $f=N f+2(f_\perp,0)$, so \eqref{eqn:Symmetry of E_DB when A is block} shows that
\begin{align*}
    2f &= \sgn(DB)f + N f+2(f_\perp,0) \\
    &= \sgn(DB)N f + N \sgn(DB)f + 2(\sgn(DB)+I)(f_\perp,0) \\
    &=2(\sgn(DB)+I)(f_\perp,0)
\end{align*}
so $\|f\|_2\lesssim\|f_\perp\|_2$. The estimate $\|f\|_2\lesssim\|(f_\|,f_\oo)\|_2$ follows similarly using $f=-N f+2(0,f_\|,f_\oo)$.

To complete the proof, it suffices to prove that $\mathbb{P}_{\|,\oo} \colon \Hil_{DB}^+ \to \mathbb{P}_{\|,\oo} \Hil$ and $\mathbb{P}_{\perp}\colon \Hil_{DB}^+ \to \mathbb{P}_{\perp} \Hil$ are isomorphisms, since $\mathbb{P}_{\|,\oo} \Hil = \{(0,\gradV g) : g \in \dot{\mathcal{V}}^{1,2} (\R^n)\}$ and $\mathbb{P}_{\perp}\Hil = \{(f,0) : f\in L^2(\R^n)\}$ by Lemma~\ref{lem:RanD}. The estimates just proved show that these mappings are injective. Meanwhile, if $f \in \mathbb{P}_{\|,\oo} \Hil$, then $Nf = f$, so $(I+N)f=2f$ and $(I-N)f=0$, and  \eqref{eqn:Symmetry of E_DB when A is block} shows that
\[
f = \tfrac{1}{2}\sgn(DB)(I-N)f + \tfrac{1}{2}(I+N)f 
= \tfrac{1}{2}(I+N)(\sgn(DB)+ I)f
= \mathbb{P}_{\|,\oo} (2 \chi^+(DB) f),
\]
where we used $2 \chi^+(DB) = (\chi^+(DB)-\chi^-(DB))+(\chi^+(DB)+\chi^-(DB))$ to obtain the final equality, hence $\mathbb{P}_{\|,\oo} \colon \Hil_{DB}^+ \to \mathbb{P}_{\|,\oo} \Hil$ is surjective. The proof that $\mathbb{P}_{\perp}\colon \Hil_{DB}^+ \to \mathbb{P}_{\perp} \Hil$ is 
surjective follows analogously.
\end{proof}


\subsubsection{Hermitian Coefficients}


We now prove that the boundary mappings $\Phi_N$ and $\Phi_R$ from \eqref{eqn:linear mappings} are isomorphisms when $A$ is Hermitian and $aV$ is real-valued, in which case $aV=(aV)^*$,
\[
\mathcal{A}
= \begin{bmatrix}
A_{\perp \perp} & A_{\perp \parallel} & 0 \\
A_{\parallel \perp} & A_{\parallel \parallel} & 0 \\
0 & 0 & a e^{i\arg V}
\end{bmatrix}
= \begin{bmatrix}
A_{\perp \perp}^* & A_{\parallel \perp}^* & 0 \\
A_{\perp \parallel}^* & A_{\parallel \parallel}^* & 0 \\
0 & 0 & a^*e^{-i\arg V}
\end{bmatrix}
= \mathcal{A}^*
\quad \text{ and } \quad
B = \widehat{\cA}
\]
as in~\eqref{eq:Ahatdef}. We follow the strategy used by  Auscher--Axelsson--McIntosh in \cite[Section~4.1]{AAMc10-2}. The Rellich-type estimates below show that these mappings are injective.

\begin{prop} \label{prop:self adjoint injective}
If $A$ is Hermitian and $aV$ is real-valued, then  $\|f_{\perp}\|_2 \eqsim \| (f_{\parallel} , f_\oo) \|_2$ for all $f\in \Hil_{DB}^+$, where the implicit constants depend only on $\lambda$ and $\Lambda$.
\end{prop}

\begin{proof}
Suppose that $f \in \Hil_{DB}^+$ and define $F(t)=e^{-t [DB]}f$ for all $t\in\R_+$. Let $N$ denote the operator from the proof of Proposition~\ref{prop:Block-Type Isomorphisms}. The Hermitian structure and \eqref{eq:Ahatdef} imply that $(\widehat{\mathcal{A}}\,)^* = B^* = N B N$, whilst $N^2=I$ and $ND=-DN$, so  Proposition~\ref{Prop:Semigroup solves first order equation} provides the identity
\begin{align*}
    \int_0^{\infty} \p \langle N B F (t) , F (t) \rangle \d t & = \int_0^{\infty} ( \langle N B \p F (t) , F (t) \rangle + \langle N B F (t) , \p F (t) \rangle ) \d t \\
	& = -\int_0^{\infty} ( \langle N B DB F (t) , F (t) \rangle + \langle N B F (t) , DB F (t) \rangle ) \d t \\
    & = -\int_0^{\infty} ( \langle (N B N) N DB F (t) , F (t) \rangle + \langle D N B F (t) , B F (t) \rangle ) \d t \\
    & = -\int_0^{\infty} \langle (N D + D N) B F (t) , B F (t) \rangle \d t \\
    & = 0,
\end{align*}
where we used that $\p F =-DB F$ to obtain the second equality, which implies that $F\in\D(DB)$ and thus $F\in\D(DNB)$, as required for the third equality. Meanwhile, the Fundamental Theorem of Calculus and the $L^2$-convergence results in Proposition~\ref{Prop:Semigroup solves first order equation} imply that
\[
    \int_0^{\infty} \p \langle N B F (t) , F (t) \rangle \d t = \lim_{t \to \infty} \langle N B F (t) , F (t) \rangle - \lim_{t\to 0} \langle N B F (t) , F(t) \rangle = - \langle N B f , f \rangle,
\]
hence $\langle N B f , f \rangle = 0$. This shows that $\langle (B f)_{\perp} , f_{\perp} \rangle = \langle (B f)_{\parallel} , f_{\parallel} \rangle + \langle (B f)_\oo , f_\oo \rangle$, thus
\[
    \kappa(B) \norm{f}{2}{2} \leq \Re \langle B f , f \rangle 
    = 2 \Re \langle (B f)_{\perp} , f_{\perp} \rangle 
    \leq 2 \|B\|_\infty \|f\|_2 \norm{f_\perp}{2}{}
\]
and
\[
     \kappa(B) \norm{f}{2}{2} \leq \Re \langle B f , f \rangle
    = 2 \Re \left( \langle (B f)_{\parallel} , f_{\parallel} \rangle + \langle (B f)_\oo , f_\oo \rangle \right)
    \leq 2 \|B\|_\infty \|f\|_2 \|(f_{\parallel},f_\oo)\|_2.
\]
The result follows, since $\kappa(B)$ and $\|B\|_\infty$ depend only on $\lambda$ and $\Lambda$ by  Proposition~\ref{prop:bddellAhat}.
\end{proof}
 
The surjectivity of the boundary mappings for Hermitian coefficients is now established using the method of continuity in Theorem~\ref{thm:absper}.

\begin{prop} \label{prop:Self-Adjoint Isomorphisms}
If $A$ is Hermitian and $aV$ is real-valued, then $\Phi_N \colon \Hil_{DB}^+ \to L^2 (\R^n)$ and $\Phi_R \colon \Hil_{DB}^+ \to \{ \gradV g : g \in \Wvo^{1,2} (\R^n) \}$ are isomorphisms.
\end{prop}

\begin{proof}
Suppose that $\sigma,\, \tau \in [0,1]$ and consider the Hermitian matrix $\mathcal{A}_{\tau} \coloneqq \tau \mathcal{A} + (1 - \tau) I$, which satisfies $ \|\mathcal{A}_{\tau}\|_{\infty}\leq \max\{\Lambda,1\}
$ and $\Re \langle \mathcal{A}_{\tau} v , v \rangle \geq \min \{\lambda,1\} \norm{v}{2}{2}$
for all $v\in\Ran(D)$ by \eqref{eq:bddellA} and Lemma~\ref{lem:RanD}. We set $B_\tau \coloneqq  \widehat{\cA}_{\tau}$, as in \eqref{eq:Ahatdef}, and note that it satisfies \eqref{eq:bound on R(D)}--\eqref{eq:elliptic on R(D)} with $\|B_\tau\|_\infty\leq\max\{\Lambda^2,1\}/\min \{\lambda,1\}$ and $\kappa(B_\tau)\geq\min\{\lambda,1\}/\max\{\Lambda,1\}$.

Now consider $\Hil_{\tau}^+ \coloneqq \{ \chi^+ (DB_{\tau}) f : f \in \Hil \}$ and the mappings $\Phi_N^{\tau} \colon \Hil_{\tau}^+ \to L^2 (\R^n)$ and $\Phi_R^{\tau} \colon \Hil_{\tau}^+ \to  \{ \gradV g : g \in \Wvo^{1,2} (\R^n) \}$ given by
$\Phi_N^{\tau} (f) \coloneqq f_{\perp}$ and $\Phi_R^{\tau} (f) \coloneqq (f_{\parallel}, f_\oo )$
for all $f\in \Hil_{\tau}^+$. Observe that, since $\mathcal{A}_{\tau}$ is Hermitian, the Rellich estimates in Proposition \ref{prop:self adjoint injective} show that there exist $C_N\in(0,\infty)$ and $C_R\in(0,\infty)$, depending only on $\lambda$ and $\Lambda$, such that $\|u\|_2 \leq C_N \|u_\perp\|_2$ and $\|u\|_2 \leq C_R \|(u_\|,u_\mu)\|_2$ for all $u\in \Hil_{\tau}^+$.

The above bounds imply that there exists $C_0\in(0,\infty)$, depending only on $\lambda$ and $\Lambda$, such that
$\|B_\sigma-B_\tau\|_{\infty} \leq C_0 \|\mathcal{A}_\sigma-\mathcal{A}_\tau\|_{\infty} = C_0|\tau-\sigma|\|\mathcal{A}-I\|_\infty$, since \eqref{eq:Ahatinv} shows that
\[
B_\sigma-B_\tau
= \ula{\mathcal{A}}_\sigma\ola{\mathcal{A}}_\sigma^{-1} -  \ula{\mathcal{A}}_\tau\ola{\mathcal{A}}_\tau^{-1} \\
= \ula{\mathcal{A}}_\sigma \ola{\mathcal{A}}_\sigma^{-1} (\ola{\mathcal{A}}_\tau - \ola{\mathcal{A}}_\sigma)\ola{\mathcal{A}}_\tau^{-1} + (\ula{\mathcal{A}}_\sigma - \ula{\mathcal{A}}_\tau) \ola{\mathcal{A}}_\tau^{-1}.
\]
Moreover, by Theorem~\ref{thm:absper}, there exists $\varepsilon\in(0,1)$, depending only on $n$, $\lambda$,  $\Lambda$, $\upsilon$, $C_N$ and $C_R$, such that whenever   $C_0|\tau-\sigma|\|\mathcal{A}-I\|_\infty<\varepsilon$ and both $\Phi_N^{\tau} \colon \Hil_{\tau}^+ \to L^2 (\R^n)$ and $\Phi_R^\tau \colon \Hil_{\tau}^+ \to \{ \gradV g : g \in \Wvo^{1,2} (\R^n) \}$ are isomorphisms, then both $\Phi_N^{\sigma} \colon \Hil_{\sigma}^+ \to L^2 (\R^n)$  and $\Phi_R^\sigma \colon \Hil_{\sigma}^+ \to \{ \gradV g : g \in \Wvo^{1,2} (\R^n) \}$ are isomorphisms. 

Meanwhile, we know from Proposition~\ref{prop:Block-Type Isomorphisms} in the case $\mathcal{A}=I$ that both $\Phi_N^0 \colon \Hil_0^+ \to L^2 (\R^n)$ and $\Phi_R^0 \colon \Hil_0^+ \to \{ \gradV g : g \in \Wvo^{1,2} (\R^n) \}$ are isomorphisms. Therefore, we can iterate the perturbation result above to reach any value of $\tau\in[0,1]$, since at each step the value of $\varepsilon$ will be the same. In particular, after $N$ iterations, where $N\geq 2C_0\|\mathcal{A}-I\|_\infty\varepsilon^{-1}$, we find that $\Phi_N^1 = \Phi_N$ and $\Phi_R^1 = \Phi_R$ are isomorphisms, as required.
\end{proof}


\section{Non-Tangential Maximal Function Bounds}\label{sec:NT}


To prove that the Neumann $\Neu$ and Regularity $\Reg$ problems are well-posed, it remains to establish non-tangential maximal function bounds for weak solutions. This is achieved here via Theorem~\ref{thm:First-Order NT Control}, which provides an equivalence between square function and non-tangential maximal function bounds for semigroup solutions to the Cauchy--Riemann type system~\eqref{eqn:First-Order Equation} for operators $DB$ when $V\in B^{q}(\R^n)$ with $q\geq\max\{\tfrac{n}{2},2\}$. We also establish a Fatou-type result in Proposition~\ref{prop:Convergence on Whitney Averages}, which shows that the Whitney averages of such solutions converge pointwise almost everywhere on $\R^n$ to their initial data. These results rely on the reverse H\"older estimates for solutions obtained below in Proposition~\ref{prop:Reverse Holder of gradients of solutions} and the $L^p$-type off-diagonal bounds for resolvents in Proposition~\ref{prop:Lq off diagonal estimates}.

The method originally developed by Auscher--Axelsson--Hofmann in~\cite[Proposition~2.56]{AAH08} for this purpose when $V\equiv0$ relies crucially on the fact that $\div A \nabla (u-c)=0$ whenever $\div A \nabla u=0$ and $c\in\R$. This is because the reverse H\"{o}lder bounds for $\nabla u=\nabla(u-c)$ then follow from Caccioppoli's inequality and a Sobolev--Poincar\'{e} inequality. More generally, however, we only have $H_{A,a,V}(u-c)=-aVc$ whenever $H_{A,a,V}u=0$ and $c\in\R$. This leads us to derive the Caccioppoli inequality for inhomogeneous Schr\"{o}dinger equations in Proposition~\ref{prop:Caccioppoli}. This is then combined with the Fefferman--Phong inequality~\eqref{eq:FePh} to obtain the reverse H\"{o}lder estimates for $\gradV u$ in Proposition~\ref{prop:Reverse Holder of gradients of solutions}. The detailed treatment for parabolic equations obtained by Auscher--Egert--Nyst\"rom in~\cite[Theorem~2.13]{AEN16-1} was particularly insightful here.


\subsection{Reverse H\"older Estimates}


We consider weak solutions to inhomogeneous Schr\"{o}dinger equations $- \div A \nabla u + aV u = \gradV^* f$ on open sets $\Omega\subset \R^d$ when $d\in\N$, $V \in L^1_{\Loc} (\R^d)$ and $f\in L^2_{\Loc}(\R^d;\C^{d+1})$, where $\gradV = (\partial_1,\ldots,\partial_d,\V)$ as in~\eqref{eq:def_nabla_nu} and $\gradV^*$ has the meaning described in~\eqref{eq:definhomsol} below. The results only require coefficients $A \in L^{\infty} (\R^d ; \mathcal{L} (\C^d))$ and $a \in L^{\infty} (\R^d ; \mathcal{L} (\C))$ with constants $0<\lambda\leq\Lambda<\infty$ that satisfy the bound and G{\aa}rding-type ellipticity
\begin{equation}\label{eq:bddellcoeffRH}
\max\left\{\|A\|_{L^{\infty} (\R^d; \mathcal{L} (\C^d))},
\|a\|_{L^{\infty}(\R^d)}\right\} \leq \Lambda
\text{ and }
\Re \int_{\R^d} ( A \nabla u \cdot \overline{\nabla u} + a V |u|^2) \geq \lambda \int_{\R^d} |\gradV u|^2
\end{equation}
for all $u\in C_c^\infty(\R^d)$ (and hence $u\in \Wv^{1,2}(\R^d)$ by Lemma~\ref{lem:testdense}). In particular, we shall write that 
\begin{equation}\label{eq:definhomsol}
- \div A \nabla u + aV u = \gradV^* f \text{ in } \Omega
\end{equation}
to mean that $u \in \Wv^{1,2}_{\Loc}(\Omega)$ and 
$\int_\Omega (A \nabla u \cdot \overline{\nabla v} + a V u \overline{v}) = \int_\Omega f \cdot \overline{\gradV v}$ for all $v \in \test(\Omega)$. The homogeneous version of the following Caccioppoli inequality (when $f\equiv0$) was used previously in the proof of Lemma~\ref{lem:pusol}.

\begin{prop} \label{prop:Caccioppoli}
Suppose that $d\in\N$, $V \in L^1_{\Loc} (\R^d)$, $(A,a)$ satisfy \eqref{eq:bddellcoeffRH} with $0<\lambda\leq\Lambda<\infty$, and $\alpha>1$. If $f \in L^2_{\Loc}(\R^d;\C^{d+1})$ and $- \div A \nabla u + aV u = \gradV^* f$ in an open set $\Omega\subset\R^d$, then
\[
    \int_{Q} |\nabla u|^2 + \int_{Q} |V| |u|^2 \lesssim \frac{1}{\ell(Q)^2} \int_{\alpha Q} |u|^2 + \int_{\alpha Q} |f|^2
\]
for all cubes $Q\subset\R^d$ such that $\alpha Q \Subset\Omega$, where the implicit constant depends only on $d$, $\lambda$, $\Lambda$ and $\alpha$.
\end{prop}

\begin{proof}
Let $\eta \in \test (\Omega)$ be supported in $\alpha Q\Subset\Omega$ such that $0 \leq \eta(x) \leq 1$ for all $x \in \Omega$ and $\eta (x) = 1$ for all $x \in Q$ whilst $\| \nabla \eta \|_{\infty} \lesssim 1/\ell(Q)$, where the implicit constant depends  on $\alpha>1$. If $- \div A \nabla u + aV u  = \gradV^* f$ in $\Omega$, then the meaning of \eqref{eq:definhomsol} implies that $u \eta^2 \in \Wvc^{1,2} (\Omega)$ and
\begin{equation}\label{eq:soln}
    \int_\Omega (A\nabla u \cdot \overline{\nabla (u \eta^2)} + aV|u|^2\eta^2)= \int_\Omega f \cdot \overline{\gradIV (u \eta^2)},
\end{equation}
since $\test(\Omega)$ is dense in $\Wvc^{1,2}(\Omega)$.

Now consider $\varepsilon \in(0,1)$ to be chosen later. The product rule gives
\[
    \int_Q | \gradV u |^2 \leq \int_{\Omega} | \gradV u |^2 \eta^2 = \int_{\Omega} | \gradV (u\eta) |^2 + \int_\Omega |u|^2 |\nabla \eta|^2
\]
and, when combined with \eqref{eq:soln} and the ellipticity in~\eqref{eq:bddellcoeffRH}, it also gives
\begin{align*}
    \int_\Omega |\gradV (u\eta)|^2
    & \lesssim \left| \int_\Omega (A \nabla (u\eta) \cdot \overline{\nabla (u\eta)} + aV|u|^2\eta^2)\right| \\
    & = \left| \int_\Omega u  A \nabla \eta \cdot \overline{\nabla (u \eta)} + \eta  A \nabla u \cdot \overline{\nabla (u \eta)} + a V |u\eta|^2 \right| \\ 
    & = \left| \int_\Omega |u|^2  A \nabla \eta \cdot \nabla \eta + u \eta A \nabla \eta \cdot \overline{\nabla u} + f \cdot \overline{\gradIV (u \eta^2)} - \overline{u} \eta  A \nabla u \cdot \nabla \eta \right|.
\end{align*}    
Next, we combine the previous two estimates with the bound in~\eqref{eq:bddellcoeffRH} and the $\varepsilon$-version of Young's inequality to obtain
\begin{align*}
    \int_\Omega |\gradV u|^2 \eta^2 
    & \lesssim \int_{\Omega}  |u|^2 |\nabla \eta|^2 + |u| |\eta| |\nabla \eta| |\nabla u| + |f| |\gradIV (u \eta^2)| \\
    & \lesssim \int_\Omega |\nabla u| |\eta \nabla \eta||u| + \int_\Omega |f| \left(|u||\eta \nabla \eta| +  |\nabla u|\eta^2 + |\V u|\eta^2 \right)\\
    &\lesssim \varepsilon \int_\Omega |\nabla u|^2\eta^2 + \left(\frac{1}{\varepsilon} + \varepsilon \right) \int_\Omega |u|^2|\nabla \eta|^2  + \varepsilon \int_\Omega |\V u|^2\eta^2 + \frac{1}{\varepsilon} \int_\Omega |f|^2\eta^2,
\end{align*}
hence
\[
    \int_\Omega |\gradIV u|^2 \eta^2 \lesssim \varepsilon \int_\Omega |\gradIV u|^2 \eta^2 + \frac{1}{\varepsilon} \int_\Omega |u|^2 |\nabla \eta|^2 + \frac{1}{\varepsilon} \int_\Omega |f|^2 \eta^2,
\]
where the implicit constant depends only on $d$, $\lambda$, $\Lambda$ and $\alpha$.

We now choose $\varepsilon \in(0,1)$ sufficiently small, and recall the properties of $\eta$, to obtain
\[
    \int_{Q} |\gradIV u|^2
    \leq \int_\Omega |\gradIV u|^2 \eta^2
    \lesssim \frac{1}{\ell(Q)^2} \int_{\alpha Q} |u|^2 + \int_{\alpha Q} |f|^2,
\]
where the implicit constant depends only on $d$, $\lambda$, $\Lambda$ and $\alpha$, as required.
\end{proof}

We will combine the preceding Caccioppoli inequality with the following reverse H\"{o}lder estimate for solutions to prove Proposition~\ref{prop:Reverse Holder of gradients of solutions} below. This approach is modelled on that developed for parabolic equations by Auscher--Monniaux--Portal in \cite[Corollary~4.2]{AMP15}.

\begin{prop}\label{prop:Reverse Holder}
Suppose that $d\geq 3$, $V \in L^1_{\Loc} (\R^d)$, $(A,a)$ satisfy \eqref{eq:bddellcoeffRH} with $0<\lambda\leq\Lambda<\infty$, and $\alpha>1$. If $\delta\in(0,\infty)$ and $- \div A \nabla u + aV u = 0$ in an open set $\Omega\subset\R^d$, then 
\[
    \left( \dashint_{Q} |u|^{2^*} \right)^{{1}/{2^*}}
    \lesssim \left(\dashint_{\alpha Q} |u|^\delta \right)^{1/\delta}
\]
for all cubes $Q\subset\R^d$ such that $\alpha Q \Subset\Omega$, where $2^*\coloneqq 2d/(d-2)$ denotes the Sobolev exponent for $\R^d$ and the implicit constant depends only on $d$, $\lambda$, $\Lambda$, $\alpha$ and $\delta$.
\end{prop}

\begin{proof}
If $u \in W^{1,2}(Q)$, then using the Sobolev--Poincar\'e inequality (see (7.45) in \cite{GT77}) we have
\begin{align*}
    \left(\dashint_{Q} |u|^{2^*}\right)^{1/2^*}
    \lesssim \left(\dashint_{Q} |u - u_Q|^{2^*}\right)^{1/2^*} + \dashint_Q |u|
    \lesssim \ell(Q)\left(\dashint_{Q} |\nabla u|^2\right)^{1/2} + \left(\dashint_Q |u|^2\right)^{1/2},
\end{align*}
where $u_Q\coloneqq \dashint_Q u$. Therefore, by the Caccioppoli inequality in Proposition~\ref{prop:Caccioppoli} in the case $f\equiv0$, we have the weak reverse H\"older estimate
\[
\left(\dashint_{Q} |u|^{2^*} \right)^{{1}/{2^*}}
\lesssim \left(\dashint_{\alpha Q} |u|^2\right)^{1/2}
\]
for all cubes $Q\subset\R^d$ such that $\alpha Q \Subset\Omega$ for some $\alpha>1$, whenever $- \div A \nabla u + aV u = 0$ in $\Omega$. The self-improvement of the exponent in the right-hand side of such estimates (\cite[Theorem ~2]{IN85}) completes the proof.
\end{proof}
The self-improvement obtained by Iwaniec--Nolder in~ \cite[Theorem~2]{IN85} shows that if $\delta\in(0,\infty)$ and $V \in B^q(\R^d)$ for some $q\in(1,\infty)$ with $\llbracket V\rrbracket_{q}\leq \upsilon < \infty$, then $(\dashint_Q V^q)^{1/q} \lesssim (\dashint_Q V^\delta)^{1/\delta}$ for all cubes $Q \subset \R^d$, where the implicit constant depends only on $\upsilon$ and $\delta$. In particular, we have
\begin{equation} \label{eqn:Changing exponent for FP}
    \dashint_Q V  \eqsim \left(\dashint_Q V^{1/2}\right)^{2}
\end{equation}
for all cubes $Q \subset \R^d$, where the implicit constants depend only on $\upsilon$. We combine this with the Fefferman--Phong inequality ~\eqref{eq:FePh} with $p=1$ to obtain the following result. 

\begin{prop} \label{prop:Reverse Holder of gradients of solutions}
Suppose that  $d\geq 3$, $V \in B^q(\R^d)$ for some $q\in(1,\infty)$ with  $\llbracket V\rrbracket_{q}\leq \upsilon < \infty$, $(A,a)$ satisfy \eqref{eq:bddellcoeffRH} with $0<\lambda\leq\Lambda<\infty$, and $\alpha>1$. If $\delta\in(0,\infty)$ and $- \div A \nabla u + aV u = 0$ in an open set $\Omega\subset\R^d$, then
\[
    \left( \dashint_{Q} | \gradIV u |^2 \right)^{1/2} \lesssim \left( \dashint_{\alpha Q} |\gradIV u |^{\delta} \right)^{1/\delta},
\]
for all cubes $Q\subset\R^d$ such that $\alpha Q \Subset\Omega$, where the implicit constant depends only on $d$, $\lambda$, $\Lambda$, $\upsilon$, $\alpha$ and $\delta$.
\end{prop}

\begin{proof}
It follows from \eqref{eqn:Changing exponent for FP} that $V^{1/2} \in B^{2q}(\R^d)$ with $\llbracket V^{1/2}\rrbracket_{2q}\leq \widetilde{\upsilon} < \infty$, where $\widetilde{\upsilon}$ depends only on $\upsilon$. Now suppose that $- \div A \nabla u + aV u = 0$ in $\Omega$ and let $Q\subset\R^d$ denote a cube such that $\alpha Q \Subset\Omega$ for some $\alpha>1$. We consider the following two cases:

(1) If $\ell(\alpha Q ) \dashint_{\alpha Q} V^{1/2} \geq 1$, then by the Caccioppoli inequality in Proposition~\ref{prop:Caccioppoli} with $f\equiv0$, followed by the reverse H\"older estimate in Proposition~\ref{prop:Reverse Holder} with $\delta=1$, we have 
\begin{align}\label{eq:wrH1}
   \bigg(\dashint_{Q} | \gradIV u |^2\bigg)^{1/2} 
   \lesssim \frac{1}{\ell(Q)}\bigg( \dashint_{\frac{1+\alpha}{2} Q} |u|^2\bigg)^{1/2} 
   \lesssim \frac{1}{\ell(Q)} \dashint_{\alpha Q} |u| 
   \lesssim \dashint_{\alpha Q} |\gradIV u |,
\end{align}
where we used the Fefferman--Phong inequality~\eqref{eq:FePh} with $p=1$ and $V^{1/2} \in B^{2q}(\R^d)$, since $\llbracket V^{1/2}\rrbracket_{2q}\leq \widetilde{\upsilon} < \infty$ and $\widetilde{\upsilon}$ depends only on $\upsilon$, to obtain the final estimate. 

(2) If $\ell(\alpha Q ) \dashint_{\alpha Q} V^{1/2} \leq 1$, then we set $u_{\alpha Q} \coloneqq \dashint_{\alpha Q} u$ and write 
\[
    \dashint_{Q} |\gradIV u|^2
    \lesssim \dashint_{Q} |\gradIV (u - u_{\alpha Q})|^2 + \dashint_{Q} V|u_{\alpha Q}|^2.
\]
Now define $f=(f_1,\ldots,f_{d+1})$ in $L^2_{\Loc}(\R^d;\C^{d+1})$ by setting $f_{d+1} = -aV^{1/2} u_{\alpha Q}$ and $f_j\equiv 0$ when $j\in\{1,\ldots,d\}$. Observe that, since $- \div A \nabla u + aV u  = 0$ in $\Omega$, we have
\[
\int_\Omega \left(A\nabla(u - u_{\alpha Q})\cdot\overline{\nabla v} + aV(u - u_{\alpha Q})\overline{ v}\right) = -\int_\Omega aV u_{\alpha Q}\overline{ v} = \int_\Omega f\cdot\overline{\gradV v}
\]
for all $ v\in\test(\Omega)$, hence
$- \div A \nabla (u - u_{\alpha Q}) + aV (u - u_{\alpha Q}) = \gradV^* f$ in $\Omega$. The inhomogeneous version of Caccioppoli inequality in Proposition~ \ref{prop:Caccioppoli}  thus implies that
\[
    \dashint_{Q} |\gradIV u|^2
    \lesssim \frac{1}{\ell(Q)^2} \dashint_{\alpha Q} |u - u_{\alpha Q}|^2 + \dashint_{\alpha Q} V |u_{\alpha Q}|^2
    \lesssim \left( \dashint_{\alpha Q} | \nabla u |^{2_*} \right)^{{2}/{2_*}} + \dashint_{\alpha Q} V|u_{\alpha Q}|^2,
\]
where we used the Sobolev--Poincar\'e inequality (see (7.45) in \cite{GT77}) in the second estimate with $2_* \coloneqq 2d/(d+2)$. Using \eqref{eqn:Changing exponent for FP} followed by the Fefferman--Phong inequality~\eqref{eq:FePh} with $p=1$ and $V^{1/2} \in B^{2q}(\R^d)$, but now in the case when $\ell(\alpha Q) \dashint_{\alpha Q} V^{1/2} \leq 1$, we have
\begin{align*}
    \left(\dashint_{\alpha Q} V |u_{\alpha Q}|^2\right)^{1/2}
     \leq \left( \dashint_{\alpha Q} V\right)^{1/2} \dashint_{\alpha Q} |u|
    \lesssim \left(\dashint_{\alpha Q} V^{1/2} \right) \dashint_{\alpha Q} |u|
    \lesssim \dashint_{\alpha Q} | \gradV u |.
    \end{align*}
Combining these estimates with Jensen's inequality we get
\begin{equation}\label{eq:wrH2}
    \left(\dashint_{Q} | \gradIV u |^2\right)^{1/2} \lesssim \left( \dashint_{\alpha Q} | \gradIV u |^{2_*} \right)^{{1}/{2_*}},
\end{equation}
since $1<2_*<2$.

The weak reverse H\"older estimate \eqref{eq:wrH2} obtained in case (2) also holds in case (1) by \eqref{eq:wrH1}. The self-improvement of the exponent in the right-hand side of such estimates (\cite[Theorem~2]{IN85}) completes the proof.
\end{proof}


\subsection{Off-Diagonal Estimates}


The next step in proving the non-tangential maximal function bounds is to establish $L^p$-type off-diagonal estimates for $DB$ when $p$ is in an open interval that contains $2$, where $D$ is defined as in~\eqref{eq:Ddef} and $B=\widehat{\cA}_{A,a,V}$ as in~\eqref{eq:Ahatdef} (cf. \eqref{eq:BA5} below) with constants $0<\lambda\leq\Lambda<\infty$. We  adapt the method developed by Auscher--Axelsson in~\cite[Lemma~10.3]{AA11} (cf.~\cite[Lemma~6.1]{ARR15} and \cite[Lemma~2.10]{AEN16-1}) to account for a potential $V$, beginning with some interpolation results for the adapted Sobolev spaces $\Wv^{1,p}(\R^n)$ from Section~\ref{sec:pre}. It will be essential for us to track how the  interpolation constants depend on the reverse H\"{o}lder constant $\llbracket V\rrbracket_q$ when $q\in(1,\infty)$. This is because the interpolation results are ultimately used to prove $L^p$-type resolvent bounds in Lemma~\ref{lem:Resolvent bounds} via a certain scaling argument which preserves the reverse H\"{o}lder constant of the potential.

The following notation is needed to state the proceeding interpolation theory. If $A_0$ and $A_1$ are Banach spaces embedded in a common Hausdorff topological vector space, and $\theta\in(0,1)$, then $[A_0,A_1]_{[\theta]}$ denotes the Banach space given by the complex interpolation functor (see, for instance, \cite[Chapter~4]{BerghLofstrom76}). We write $A\eqsim B$ when $A$ and $B$ are equivalent Banach spaces in the sense that $A=B$ as sets with equivalent norms $\|a\|_{A}\eqsim\|a\|_{B}$ for all $a\in A$. Finally, the dual space of bounded linear functionals on a Banach space $A$ is denoted by $A^\prime$.

\begin{lem}\label{lem:interpolation}
If $n\in\N$, $q\in(1,\infty)$ and $V\in B^q(\R^n)$ with $\llbracket V\rrbracket_{q} \leq \upsilon < \infty$, then there exists $\delta\in(0,1)$, depending only on $q$, such that the adapted Sobolev spaces on $\R^n$ satisfy
\[
[\Wv^{1,p_0},\Wv^{1,p_1}]_{[\theta]}\eqsim\Wv^{1,p}
\qquad\text{and}\qquad
[(\Wv^{1,p_0})^\prime,(\Wv^{1,p_1})^\prime]_{[\theta]}\eqsim(\Wv^{1,p})^\prime
\]
whenever $\theta\in(0,1)$, $p_0,\, p_1\in [2-\delta,2+\delta]$ and $\tfrac{1}{p}=\tfrac{1-\theta}{p_0}+\tfrac{\theta}{p_1}$, where the implicit constants depend only on $n$, $\theta$, $p_0$, $p_1$ and $\upsilon$.
\end{lem}

\begin{proof}
We adopt the strategy outlined in  \cite{badr09}. Indeed, using \cite[Theorem~1.2]{AB07}, we have
\[
\|(-\Delta+V+1)^{1/2}f\|_{p}
\eqsim \|\nabla f\|_p+\|(V+1)^{1/2}f\|_p
\eqsim \|f\|_{\Wv^{1,p}}
\]
for all $f\in C_c^\infty(\R^n)$ and $p\in(1,2q]$, where the implicit constants depend only on $n$, $p$ and $\upsilon$. Moreover, the fractional powers $(-\Delta+V+1)^{1/2}$ and $(-\Delta+V+1)^{-1/2}$, defined as unbounded operators in $L^2$ using (for example) the $\mathcal{F}$-Functional Calculus for the self-adjoint operator $-\Delta+V+1$ in $L^2$, have unique
bounded extensions $S_p\in\mathcal{L}(\mathcal{V}^{1,p},L^p)$ and $R_p\in\mathcal{L}(L^p,\mathcal{V}^{1,p})$, with operator norms depending only on $n$, $p$ and $\upsilon$, and $R_pS_p=I$ on $\mathcal{V}^{1,p}$ for each $p\in(1,2q]$. The extended operators are also compatible across different values of $p$, in the sense that $S_{p_0}=S_{p_1}$ on $\mathcal{V}^{1,p_0}\cap\mathcal{V}^{1,p_1}$ and $R_{p_0}=R_{p_1}$ on $L^{p_0}\cap L^{p_1}$ whenever $p_0,\, p_1\in (1,2q]$, since they are defined from a common dense subspace (such as $\test$).

We now fix  $\delta\in(0,\min\{1,2q-2\})$ and consider $p_0,\, p_1\in [2-\delta,2+\delta]$. The compatibility allows us to define linear operators
$S:\mathcal{V}^{1,p_0}+\mathcal{V}^{1,p_1}\to L^{p_0}+L^{p_1}$ and $
R:L^{p_0}+L^{p_1}\to\mathcal{V}^{1,p_0}+\mathcal{V}^{1,p_1}$ such that $S=S_{p_i}$ on $\mathcal{V}^{1,p_i}$ and $R=R_{p_i}$ on $L^{p_i}$ for each $i\in\{0,1\}$, hence $RS=I$ on  $\mathcal{V}^{1,p_0}+\mathcal{V}^{1,p_1}$. Using that $[L^{p_0},L^{p_1}]_{[\theta]}=L^{p}$ with equal norms (see, for instance,  \cite[Theorem~5.1.1]{BerghLofstrom76}) and the action of the complex interpolation functor on a retract between interpolation couples (see \cite[Theorem~6.4.2]{BerghLofstrom76} or \cite[Section~1.2.4]{Triebel78}), we obtain $[\Wv^{1,p_0},\Wv^{1,p_1}]_{[\theta]}=\Wv^{1,p}$ with equivalent norms, where the norm equivalence depends only on $n$, $\theta$, $p_0$, $p_1$ and $\upsilon$.

The results for the dual spaces then follow from
the duality theorem for complex interpolation (see \cite[Theorem 4.5.1 and Corollary 4.5.2]{BerghLofstrom76}). The required density properties are in Lemma~\ref{lem:testdense}, whilst the fact that $\Wv^{1,p}$ is reflexive when $p\in(1,\infty)$ can be seen by considering its natural embedding into $n+2$ copies of $L^p$, following the proof for Sobolev spaces in \cite[Section~7.5]{GT77}.
\end{proof}

These interpolation results imply  $L^p$-type resolvent bounds for $DB$ when $p$ is close to $2$. The proof below relies on Shneiberg's original stability theorem in~\cite{Sne74}, and the quantitative version obtained by Auscher--Bortz--Egert--Saari in~\cite[Theorem~A.1]{ABES2019} is needed to ensure the uniform dependency, with respect to $t\in\R\setminus\{0\}$, of implicit constants when a certain scaling $V\mapsto V_t$ is applied to the potential.

\begin{lem} \label{lem:Resolvent bounds}
If $n\geq 3$, $q\in(1,\infty)$ and $V\in B^q(\R^n)$ with $\llbracket V\rrbracket_{q} \leq \upsilon < \infty$, then there exists $\varepsilon\in(0,1)$, depending only on $n$, $\lambda$, $\Lambda$, $q$ and $\upsilon$, such that whenever $|p-2|<\varepsilon$ it holds that $\|(I + it DB)^{-1}f\|_p \lesssim \|f\|_p$ for all $f \in L^p (\R^n; \C^{n+2}) \cap L^2 (\R^n; \C^{n+2})$ and $t\in\R$, where the implicit constant depends only on $n$, $\lambda$, $\Lambda$, $p$, $q$ and $\upsilon$.
\end{lem}

\begin{proof}
There exists $\delta\in(0,1)$, depending only on $q$, such that the results in Lemma~\ref{lem:interpolation} hold. Now choose $p_0\in[2-\delta,2)$ and $p_1\in(2,2+\delta]$ such that $p_0',\, p_1'\in[2-\delta,2+\delta]$. These choices depend on $\delta$ and thus can be chosen depending only on $q$. Now suppose that $t\in\R\setminus\{0\}$ and define the potential $V_t(x)\coloneqq t^2V(tx)$ for all $x\in\R^n$. A change of variables shows that  $V_t \in B_q(\R^n)$ with $\llbracket V_t \rrbracket_q = \llbracket V \rrbracket_q$.
Therefore, the Sobolev spaces $\Wv^{1,p}_t(\R^n) \coloneqq  \{f\in L^p(\R^n) : \gradVt f \in L^p(\R^n;\C^{n+1})\}$ adapted to  $\gradVt \coloneqq (\nabla,|V_t|^{1/2})$ are defined as in \eqref{eq:def_nabla_nu} for all $p\in[1,\infty)$, as is the homogeneous space $\Wvo^{1,2}_t(\R^n)$. Moreover, by Lemma~\ref{lem:interpolation}, we have
\begin{equation}\label{eq:interp}
[\Wv_t^{1,p_0},\Wv_t^{1,p_1}]_{[\theta]}\eqsim\Wv_t^{1,p}
\qquad\text{and}\qquad
[(\Wv_t^{1,p_0'})^\prime,(\Wv_t^{1,p_1'})^\prime]_{[\theta]}\eqsim(\Wv_t^{1,p'})^\prime
\end{equation}
whenever $\theta\in(0,1)$ and $\tfrac{1}{p}=\tfrac{1-\theta}{p_0}+\tfrac{\theta}{p_1}$, where the implicit constants depend only on $n$, $\theta$, $q$ and~$\upsilon$. Importantly, these constants do not depend on $t$, since $\llbracket V_t \rrbracket_q = \llbracket V \rrbracket_q \leq \upsilon$.

To facilitate changing variables, define the operator $D_t$, associated with $\gradVt$ as in \eqref{eq:Ddef}, by
\[
    D_t=\begin{bmatrix}
    0 & -(\gradVt)^* \\
    -\gradVt & 0 
    \end{bmatrix}.
\]
Observe that for each $f \in L^2(\R^n; \C^{n+2})$, with $f_t(x)\coloneqq f(tx)$ for all $x\in\R^n$, we have $f\in\D(D)$ if and only if $f_t\in\D(D_t)$; in which case 
\begin{equation}\label{eq:cov}
    tDf(tx)=D_tf_t(x).
\end{equation}
This is immediate when $f\in\test(\R^n)$ and the general result follows from the density of $\test(\R^n)$ in $\D(\gradV)=\Wv^{1,2}(\R^n)$ (see Lemma~\ref{lem:testdense}) and the definition of $\D(\gradV^*)$. Also, recall from \eqref{eq:Ahatdef} that
\begin{equation}\label{eq:BA5}
\mathcal{A} 
=\begin{bmatrix}
\App & \Apv & 0 \\
\Avp & \Avv & 0 \\
0 & 0 & a
\end{bmatrix}
\qquad\text{and}\qquad
B=\widehat{\mathcal{A}}= 
\begin{bmatrix}
\App^{-1} & - \App^{-1} \Apv & 0 \\
\Avp \App^{-1} & \Avv - \Avp \App^{-1} \Apv & 0 \\
0 & 0 & a
\end{bmatrix}
\end{equation}
in order to define $A_t(x)\coloneqq A(tx)$,  $\mathcal{A}_t(x)\coloneqq \mathcal{A}(tx)$ and $B_t(x)\coloneqq B(tx)$
for all $x\in\R^n$.

Using \eqref{eq:cov}, and the density of $\mathcal{C}_c^{\infty}(\R^n)$ in $\Wvo^{1,2}_t(\R^n)$ as in \eqref{eq:bddellA}, we find that
\begin{equation}\label{eq:Atbddell}
\|\cA_t\|_\infty =
\|\cA\|_\infty \leq \Lambda
\quad\text{ and }\quad
\Re \left\langle \cA_t \begin{bmatrix}  f\\ \gradVt g \end{bmatrix} , \begin{bmatrix}  f\\ \gradVt g \end{bmatrix}\right\rangle \geq \lambda (\|f\|_2^2+\|\gradVt g\|_2^2)
\end{equation}
for all $f\in L^2(\R^n)$ and $g \in \Wvo^{1,2}_t(\R^n)$. Proposition~\ref{prop:bddellAhat} then shows that $B_t$ satisfies \eqref{eq:bound on R(D)} and \eqref{eq:elliptic on R(D)} on $\Ran(D_t)$ with $\|B_t\|_\infty\leq\max\{\Lambda^2,1\}/\lambda$ and $\kappa(B_t)\geq\lambda/\Lambda$. This allows us to apply the results from Section~\ref{sec:QE} to the first-order operator $D_tB_t$ with uniform control over the constants $\|B_t\|_\infty$ and $\kappa(B_t)$. This is useful because \eqref{eq:cov} implies the change of variables 
\begin{equation}\label{eq:covfinal}
(I + itD B)^{-1}f(tx)=(I + iD_{t} B_{t})^{-1}f_t(x)
\end{equation}
for all $f \in L^2 (\R^n; \C^{n+2})$ and $x\in\R^n$.

For each $p\in(1,\infty)$, define $L_{t,p} \colon \Wv_t^{1,p}\to (\Wv_t^{1,p'})'$ by
\[
    (L_{t,p} u) (\varphi) \coloneqq \langle \mathcal{A}_t \begin{bmatrix}
    u \\ i \gradVt u
    \end{bmatrix} , \begin{bmatrix}
    \varphi \\ i \gradVt \varphi
    \end{bmatrix} \rangle
\]
for all $u \in \Wv_t^{1,p}$ and $\varphi \in \Wv_t^{1,p'}$. Using \eqref{eq:Atbddell}, we have
\[
    |(L_{t,p} u) (\varphi) | 
    \leq \Lambda \|u\|_{\Wv_t^{1,p}} \|\varphi\|_{\Wv_t^{1,p'}}
    \quad\text{ and }\quad
    |(L_{t,2} v) (v)|
    \geq \Re (L_{t,2} v)(v)
    \geq \lambda \|v\|_{\Wv_t^{1,2}}^2,
\]
so $\|L_{t,p} u \|_{(\Wv_t^{1,p'})'} \leq \Lambda \|u\|_{\Wv_t^{1,p}}$ and $\|L_{t,2} v \|_{(\Wv_t^{1,2})'} \geq \lambda \|v\|_{\Wv_t^{1,2}}$, for all $u \in \Wv_t^{1,p}$, $\varphi \in \Wv_t^{1,p'}$, $v\in \Wv_t^{1,2}$. This shows that $L_{t,2}$ injective, so the Lax--Milgram Lemma (see, for instance, \cite[Lemma~1.3]{Ouhabaz2005}) implies that $L_{t,2}$ is bijective. Therefore, the quantitative version of Shneiberg's stability theorem from \cite[Theorem~A.1]{ABES2019} applied on the complex interpolation scales in \eqref{eq:interp} shows that there exists $\varepsilon\in(0,\delta)$,  depending only on $n$, $\lambda$, $\Lambda$, $q$ and $\upsilon$, such that whenever $|p-2|<\varepsilon$ it holds that $L_{t,p}$ is also a bijection with $(L_{t,p})^{-1}=(L_{t,2})^{-1}$ on $(\Wv_t^{1,p'})'\cap (\Wv_t^{1,2})'$ and
\[
\|L_{t,p} u \|_{[(\Wv_t^{1,p_0'})^\prime,(\Wv_t^{1,p_1'})^\prime]_{[\theta_p]}} \geq \tfrac{1}{5}\lambda\|u\|_{[\Wv_t^{1,p_0},\Wv_t^{1,p_1}]_{[\theta_p]}},
\text{ where } \tfrac{1}{p}=\tfrac{1-\theta_p}{p_0}+\tfrac{\theta_p}{p_1},
\]
for all $u\in\Wv_t^{1,p}$. Moreover, the value of $\theta_p\in (0,1)$ depends ultimately only on $p$ and $q$, so the equivalence in \eqref{eq:interp} implies further that
\begin{equation}\label{eq:Ltbdd}
\|L_{t,p} u \|_{(\Wv_t^{1,p'})'} \gtrsim \|u\|_{\Wv_t^{1,p}}
\end{equation}
for all $u\in\Wv_t^{1,p}$, where the implicit constant depends only on $n$, $\lambda$, $p$, $q$ and $\upsilon$. Importantly, it is here that we rely on $\llbracket V_t \rrbracket_q = \llbracket V \rrbracket_q \leq \upsilon$, since this ensures that the implicit constants associated with the complex interpolation do not depend on $t$, as explained beneath~\eqref{eq:interp}.

Now suppose that $|p-2|<\varepsilon$ and $f \in L^p (\R^n; \C^{n+2}) \cap L^2 (\R^n; \C^{n+2})$. Let $\tilde{f}=(I + iD_t B_t)^{-1}f$,
\[
	g = \begin{bmatrix} (B_t f)_{\perp} \\ f_{\parallel} \\ f_\oo \end{bmatrix}
	\text{ and }
	\tilde{g} = \begin{bmatrix} (B_t \tilde{f})_{\perp} \\ \tilde{f}_{\parallel} \\ \tilde{f}_\oo \end{bmatrix},
	\text{ so then }  
    f = \begin{bmatrix} (\mathcal{A}_t g)_{\perp} \\ g_{\parallel} \\ g_\oo \end{bmatrix}
    \text{ and }
    \tilde{f} = \begin{bmatrix} (\mathcal{A}_t \tilde{g})_{\perp} \\ \tilde{g}_{\parallel} \\ \tilde{g}_\oo \end{bmatrix}.
\]
If $\psi \in \mathcal{C}_c^{\infty} (\R^n ; \C^{n+2})$, then 
\begin{equation} \label{eqn:OD integral equation}
     \langle f , \psi \rangle
     = \langle  (I + i D_tB_t) \tilde{f} , \psi \rangle
     = \langle  \tilde{f} , \psi  \rangle + \langle  B_t \tilde{f} , -i D_t \psi \rangle.
\end{equation}
In particular, if $\varphi \in \mathcal{C}_c^{\infty} (\R^n)$, then setting $\psi = (\varphi,0)$ in \eqref{eqn:OD integral equation}, we obtain
\begin{equation} \label{eqn:OD perp}
    \begin{split}
        \langle  (\mathcal{A}_t g)_{\perp}, \varphi \rangle
        & = \langle  \tilde{f}_{\perp}, \varphi \rangle + \langle  (B_t \tilde{f} )_{\parallel} , i \nabla\varphi \rangle + \langle  ( B_t \tilde{f})_\oo, i |V_t|^{1/2} \varphi \rangle \\ 
	    & = \langle (\mathcal{A}_t \tilde{g} )_{\perp}, \varphi \rangle + \langle  (\mathcal{A}_t \tilde{g} )_{\parallel} , i \nabla\varphi \rangle + \langle  ( \mathcal{A}_t \tilde{g} )_\oo ,i |V_t|^{1/2} \varphi\rangle,
\end{split}
\end{equation}
since $(B_t \tilde{f})_{\parallel} = (\mathcal{A}_t \tilde{g} )_{\parallel}$ and $(B_t \tilde{f})_\oo = (\mathcal{A}_t \tilde{g} )_\oo$. Moreover, if $\varphi \in \mathcal{C}_c^{\infty} (\R^n ; \C^{n+1})$, then setting $\psi = (0,\varphi)$ in \eqref{eqn:OD integral equation}, we obtain
\[
    \langle  
    \begin{bmatrix}g_{\parallel}\\g_\oo\end{bmatrix} , \varphi \rangle
    = \langle  \begin{bmatrix}\tilde{g}_{\parallel} \\ \tilde{g}_\oo\end{bmatrix} , \varphi \rangle + \langle  (B_t \tilde{f})_{\perp} , i (\gradVt)^* \varphi \rangle
    = \langle  \begin{bmatrix}\tilde{g}_{\parallel} \\ \tilde{g}_\oo\end{bmatrix} , \varphi \rangle + \langle  -i \gradVt \tilde{g}_{\perp} , \varphi \rangle,
\]
since $\tilde{f}\in\D(D_t B_t)$, so $B_t\tilde{f}\in\D(D_t)$ with  $\tilde{g}_\perp=(B_t\tilde{f})_\perp\in \D(\gradVt)=\Wv_t^{1,2}(\R^n)$. It follows that
\begin{equation} \label{eqn:OD parallel mu}
i \gradVt \tilde{g}_{\perp} = \begin{bmatrix}
\tilde{g}_{\parallel}-g_{\parallel}\\
\tilde{g}_\oo-g_\oo
\end{bmatrix}.
\end{equation}

Next, define $F_t\colon L^p (\R^n ; \C^{n+2}) \to (\Wv_t^{1,p'} (\R^n))'$, with $\|F_t u \|_{(\Wv_t^{1,p'})'} \leq \Lambda \|u\|_p$, by
\[
    (F_t u) (\varphi) \coloneqq \langle \begin{bmatrix} (A_{\perp \perp})_t & 0 & 0 \\ 0 &-(A_{\parallel \parallel})_t & 0 \\ 0 & 0 &-a_t \end{bmatrix}
    \begin{bmatrix} u_{\perp} \\  u_{\parallel} \\ u_\oo \end{bmatrix}, \begin{bmatrix} \varphi \\ i \gradx \varphi \\ i |V_t|^{1/2} \varphi \end{bmatrix} \rangle
\]
for all $u \in L^p (\R^n ; \C^{n+2})$ and $\varphi \in \Wv_t^{1,p'} (\R^n)$. Applying \eqref{eqn:OD parallel mu} followed by \eqref{eqn:OD perp} we find that
\begin{align*}
    (L_{t,2} \tilde{g}_{\perp}) (\varphi)
    & = \langle \mathcal{A}_t \begin{bmatrix} \tilde{g}_{\perp} \\ i \gradVt \tilde{g}_{\perp} \end{bmatrix} , \begin{bmatrix} \varphi \\ i \gradVt \varphi \end{bmatrix} \rangle \\
    & = \langle \mathcal{A}_t \begin{bmatrix} \tilde{g}_{\perp} \\  \tilde{g}_{\parallel} - g_{\parallel}\\ \tilde{g}_{\oo} - g_{\oo} \end{bmatrix} , \begin{bmatrix} \varphi \\ i \nabla \varphi \\ i |V_t|^{1/2} \varphi \end{bmatrix} \rangle\\
    & = \langle \begin{bmatrix} (\mathcal{A}_tg)_{\perp} \\  0 \\ 0 \end{bmatrix} - \mathcal{A}_t \begin{bmatrix} 0 \\  g_{\parallel}\\ g_{\oo} \end{bmatrix} , \begin{bmatrix} \varphi \\ i \nabla \varphi \\ i |V_t|^{1/2} \varphi \end{bmatrix} \rangle \\
    &= F_tg(\varphi)
\end{align*}
for all $\varphi \in \test(\R^n)$, hence $L_{t,2} \tilde{g}_{\perp}=F_t g$ by density (see Lemma~\ref{lem:testdense}). Now observe that $F_tg$ is in $(\Wv_t^{1,p'}(\R^n))'\cap (\Wv_t^{1,2}(\R^n))'$, since $g\in L^p(\R^n ; \C^{n+2})\cap L^2(\R^n ; \C^{n+2})$, which allows us to write  $\tilde{g}_{\perp} = (L_{t,p})^{-1}(F_t g)=(L_{t,2})^{-1}(F_t g)$ to deduce that $\tilde{g}_{\perp}$ is in $\Wv_t^{1,p}(\R^n)\cap \Wv_t^{1,2}(\R^n)$ with
\[
\|\tilde{g}_{\perp}\|_{\Wv_t^{1,p}}
=\|(L_{t,p})^{-1}(F_t g)\|_{\Wv_t^{1,p}}
\lesssim \|F_t g\|_{(\Wv_t^{1,p'})'}
\leq \Lambda \|g\|_p,
\]
where we used \eqref{eq:Ltbdd} to obtain the first estimate. We combine this with \eqref{eqn:OD parallel mu} to obtain
\[
    \|\tilde{f}\|_p
    \lesssim \|\tilde{g}\|_p 
    \leq \|\tilde{g}_{\perp}\|_p + \|\begin{bmatrix}
    \tilde{g}_{\parallel}-g_{\parallel}\\
    \tilde{g}_\oo-g_\oo
    \end{bmatrix}\|_p + \|\begin{bmatrix}
    g_{\parallel}\\
    g_\oo
    \end{bmatrix}\|_p
    = \|\tilde{g}_{\perp}\|_{\Wv_t^{1,p}} + \|\begin{bmatrix}
    g_{\parallel}\\
    g_\oo
    \end{bmatrix}\|_p
    \lesssim \|g\|_p \lesssim \|f\|_p, 
\]
thus $\|(I + iD_tB_t)^{-1} f\|_p \leq C \|f\|_p$ for some $C\in(0,\infty)$ depending only on $n$, $\lambda$, $\Lambda$, $p$, $q$ and $\upsilon$. This completes the proof, since $f_t \in L^p(\R^n ; \C^{n+2})\cap L^2(\R^n ; \C^{n+2})$, and the change of variables in \eqref{eq:covfinal} finally implies that
\[
\|(I + itD B)^{-1}f\|_p
= t^n\|(I + iD_{t} B_{t})^{-1}f_t\|_p
\leq C t^n\|f_t\|_p
= \|f\|_p,
\]
as required.
\end{proof}

A final interpolation now combines the above resolvent bounds with the off-diagonal bounds in Proposition~\ref{prop:Off diagonal estimates} to deduce $L^p$-type off-diagonal estimates for $DB$ when $p$ is close to $2$.

\begin{prop} \label{prop:Lq off diagonal estimates}
If $n\geq 3$, $q\in(1,\infty)$ and $V\in B^q(\R^n)$ with $\llbracket V\rrbracket_{q} \leq \upsilon < \infty$, then there exists $\delta\in(0,1)$, depending only on $n$, $\lambda$, $\Lambda$, $q$ and $\upsilon$, such that whenever $|p-2|<\delta$ and $M\in(0,\infty)$ it holds that
\[
    \|\mathds{1}_E (I+itDB)^{-1} \mathds{1}_F f\|_p \lesssim \left( 1 + \frac{\dist (E,F)}{t} \right)^{-M} \|\mathds{1}_F f\|_p,
\]
for all $f \in L^p (\R^n; \C^{n+2}) \cap L^2 (\R^n; \C^{n+2})$, $t\in\R$ and measurable sets $E,\, F \subseteq \R^n$, where the implicit constant depends only on $n$, $\lambda$, $\Lambda$, $p$, $q$, $\upsilon$ and $M$.
\end{prop}

\begin{proof}
Using Lemma~\ref{lem:Resolvent bounds}, there exists $\varepsilon\in(0,1)$ such that whenever $|p-2|<\varepsilon$ it holds that
\[
    \|\mathds{1}_E (I+itDB)^{-1} \mathds{1}_F f\|_p \lesssim \|f\|_p,
\]
where the implicit constant depends only on $n$, $\lambda$, $\Lambda$, $p$, $q$ and $\upsilon$. Meanwhile, if $N\in\N$, then the off-diagonal estimates from Proposition~\ref{prop:Off diagonal estimates} show that 
\[
    \|\mathds{1}_E (I+itDB)^{-1} \mathds{1}_F f\|_2 \lesssim \left( 1 + \frac{\dist (E,F)}{t} \right)^{-N} \|f\|_2,
\]
where the implicit constant depends only on $\lambda$, $\Lambda$ and $N$. The result follows for any $\delta\in(0,\varepsilon)$ from the Riesz--Thorin interpolation theorem by choosing $N$ large enough, depending on $M$.
\end{proof}


\subsection{Non-Tangential Estimates} \label{Sec: Non-Tangential Maximal Estimate}


We now combine the weak reverse H\"older estimates for solutions with the $L^p$-type off-diagonal estimates to prove the non-tangential maximal function estimates and Fatou-type result for semigroup solutions of~\eqref{eqn:First-Order Equation} when $V\in B^{q}(\R^n)$ with $q\geq\max\{\tfrac{n}{2},2\}$. The operator $D$ is defined as in~\eqref{eq:Ddef} and $B=\widehat{\cA}_{A,a,V}$ as in~\eqref{eq:Ahatdef} with constants $0<\lambda\leq\Lambda<\infty$. It is helpful to first isolate the following elementary bounds.

\begin{lem} \label{lem:Initial bounds for NT}
If $n\in\N$ and $F \in L^2_{\Loc} (\R^{n+1}_+; \C^{n+2})$, then
\[
    \sup_{t>0} \dashint_t^{2t} \|F (s) \|_2^2 \d s \lesssim \|{N}_*F\|_2^2 \lesssim \int_0^{\infty} \|F (s) \|_2^2 \frac{\d s}{s},
\]
where the implicit constants depend only on $n$.
\end{lem}

\begin{proof}
The definition of the non-tangential maximal function in \eqref{eq:N*def} shows that
\[
    |{N}_* F (x)|^2 = \sup_{t>0} \dashint_t^{2t} \dashint_{Q(x,t)} |F(s,y)|^2 \d y \d s \lesssim \int_0^{\infty} \int_{\R^n} \mathds{1}_{Q(x,s)} (y) |F(s,y)|^2 \frac{\d y \d s}{s^{n+1}},
\]
so Tonelli's Theorem implies that
\begin{align*}
    \|{N}_* F \|_2^2  \lesssim \int_0^{\infty} \int_{\R^n}   \left(\frac{1}{s^n}\int_{\R^n}\mathds{1}_{Q(y,s)}(x)\d x\right) |F(s,y)|^2 \d y \frac{\d s}{s}
    \eqsim \int_0^{\infty} \|F (s)\|_2^2 \frac{\d s}{s},
\end{align*}
where the implicit constants depend only on $n$. Meanwhile, for each $t>0$, we have
\[
    |{N}_* F (x)|^2
    \geq \int_{t}^{2 t} \dashint_{Q(x,t)} |F(s,y)|^2 \frac{\d y \d s}{s}
    \eqsim \int_{t}^{2 t} \frac{1}{t^n} \int_{\R^n} \mathds{1}_{Q(x,t)}(y) |F(s,y)|^2 \frac{\d y \d s}{s},
\]
so Tonelli's Theorem implies that
\[
    \|{N}_* F \|_2^2
    \gtrsim \int_{t}^{2 t} \int_{\R^n} \left(\frac{1}{t^n}\int_{\R^n}\mathds{1}_{Q(y,t)}(x)\d x\right) |F(s,y)|^2 \frac{\d y \d s}{s}
    \eqsim \dashint_{t}^{2t} \|F (s) \|_2^2 \frac{\d s}{s},
\]
where the implicit constants depend only on $n$, and the result follows.
\end{proof}

The quadratic estimates in Theorem~\ref{thm:QE} are now used to obtain the non-tangential maximal function bounds following the approach of
Auscher--Axelsson--Hofmann in~\cite[Proposition~2.56]{AAH08}.

\begin{thm} \label{thm:First-Order NT Control}
Suppose that $n\geq 3$, $q\geq\max\{\tfrac{n}{2},2\}$ and $V \in B^{q}(\R^n)$ with $\llbracket V\rrbracket_{q}\leq \upsilon<\infty$. If $f \in \Hil_{DB}^+$ and $F(t) = e^{-t[DB]} f$,
then
\[
    \int_0^{\infty} \|t \partial_t F \|_2^2 \frac{\d t}{t}
    \eqsim \|f\|_2^2
    \eqsim \|{N}_* F\|_2^2,
\]
where the implicit constants depend only on $n$, $\lambda$, $\Lambda$ and $\upsilon$.
\end{thm}

\begin{proof}
We choose $\mu\in(\omega(B),\tfrac{\pi}{2})$, depending only on $\lambda$ and $\Lambda$ by Proposition~\ref{prop:bddellAhat}, so then the semigroup properties in Proposition~\ref{prop:SGfromFC} and the quadratic estimates from Theorem~\ref{thm:QE} and Theorem~\ref{thm:Quadratic estimates and functional calculus} imply that
\begin{align*}
    \int_0^{\infty} \|t \partial_t F \|_2^2 \frac{\d t}{t}
    = \int_0^{\infty} \| t [DB] e^{-t[DB]} f \|_2^2 \frac{\d t}{t}
    = \int_0^{\infty} \| \psi(t DB) f \|_2^2 \frac{\d t}{t}
    \eqsim \|f\|_2^2,
\end{align*}
where $\psi$ in $\Psi(S_{\mu}^o)$ is given by $\psi(z) \coloneqq  [z]e^{-[z]}$ for all $z\in S_{\mu}^o$, and the implicit constants depend only on $n$, $\lambda$, $\Lambda$ and $\upsilon$, since $B=\widehat{A}_{A,a,V}$ (and recall Proposition~\ref{prop:bddellAhat}).

It remains to prove that $\|{N}_* F\|_2^2 \eqsim \|f\|_2^2$. Lemma \ref{lem:Initial bounds for NT} and Proposition \ref{Prop:Semigroup solves first order equation} show that
\[
    \|{N}_* F\|_2^2 \gtrsim \sup_{t>0} \dashint_t^{2t} \| F (s) \|_2^2 \d s \geq \lim_{t \to 0^+} \dashint_t^{2t} \| F (s) \|_2^2 \d s = \|f\|_2^2,
\]
where the implicit constant depends only on $n$. 

To prove the reverse estimate, consider a Whitney cube $W \coloneqq [t , 2t] \times Q(x,t)\subset\R^{n+1}_+$, where $(t,x)\in\R^{n+1}_+$. Proposition~\ref{Prop:Semigroup solves first order equation} shows that $F$ is in $L^2_{\Loc}(\R_+,\Hil)$ and $\p F+DBF=0$ in $\R_+$, so Proposition~\ref{prop:Reduction to first-order} implies that there exists $u$  in $\Wv^{1,2}_{\Loc}(\R^{n+1}_+)$ such that $F = \gradAV u$ and $\HAaV u = 0$ in $\R^{n+1}_+$. Next, since $V\in B^{2}(\R^n)$, we choose $p\in(2-\delta,2)$, depending only on $n$, $\lambda$, $\Lambda$ and $\upsilon$, such that the off-diagonal estimates in Proposition~\ref{prop:Lq off diagonal estimates} hold.

After applying the reverse H\"{o}lder estimate from Proposition~\ref{prop:Reverse Holder of gradients of solutions} in the case $d=n+1$ and on the cube $W$, since $2 W = [t/2,5t/2]\times Q(x,2t) \Subset \R_+^{n+1}$, and recalling \eqref{eq:Ahatinv}, we obtain
\begin{align*}
   \dashiint_{W} |F|^2
     = \dashiint_{W} |\gradAV u|^2 
     \lesssim \dashiint_{W} |\gradIV u|^2  
     \lesssim_p \left( \dashiint_{2 W} |\gradIV u|^{p} \right)^{\frac{2}{p}} 
     \lesssim \left(\dashiint_{2 W} |F |^{p} \right)^{\frac{2}{p}}.
\end{align*}
Note that this application of Proposition~\ref{prop:Reverse Holder of gradients of solutions} is justified because the bound and ellipticity in \eqref{eq:bddellA} on $\R^n$ imply the bound and G{\aa}rding-type ellipticity in \eqref{eq:bddellcoeffRH} for the $t$-independent extensions of $(A,a)$ on $\R^{n+1}$, whilst the $t$-independent extension of $V$ is in $B^{q}(\R^{n+1})$ when $V\in B^{q}(\R^n)$. Now, since $p<2$, we have
\begin{align*}
\left(\dashiint_{2 W} |F|^{p}\right)^{\frac{2}{p}}
    & \lesssim \dashiint_{2 W} | \psi(sDB) f (y)|^{2} \d y \d s  + \left(\dashiint_{2 W} | R_s f (y)|^{p} \d y \d s \right)^{\frac{2}{p}},
\end{align*}
where $R_s \coloneqq  (I + is DB)^{-1}$ and now $\psi(sDB)\coloneqq e^{-s [DB]} - R_s$ can also be defined by the functional calculus using $\psi$ in $\Psi(S_{\mu}^o)$ given by $\psi (z) \coloneqq e^{-[z]} - (1 + i z)^{-1}$ for all $z\in S_{\mu}^o$, hence
\begin{align*}
    \| {N}_* F \|_2^2
    & \lesssim \| {N}_*(\psi(sDB)f)\|_2^2 + \bigg\|\sup_{t>0} \dashiint_{2 W(x,t)} |R_s f (y)|^{p} \d y \d s\bigg\|_{\frac{2}{p}}^{\frac{2}{p}}
    =: I_1 + I_2.
\end{align*}
Lemma \ref{lem:Initial bounds for NT} and the quadratic estimates from Theorem~\ref{thm:QE} and Theorem~\ref{thm:Quadratic estimates and functional calculus} imply that
\[
    I_1 = \| {N}_*(\psi(sDB)f)\|_2^2
    \lesssim \int_0^{\infty} \| \psi (s DB) f \|_2^2 \frac{\d s}{s}
    \lesssim \|f\|_2^2.
\]

To estimate $I_2$, observe that
\[
    \dashiint_{2W(x,t)} |R_s f (y) |^{p} \d y \d s
    \eqsim t^{-n}\dashint_{\frac{t}{2}}^{\frac{5t}{2}} \|\mathds{1}_{Q(x,2t)}R_s f\|_p^p \d s 
\]
If $s\in[t/2,5t/2]$ and $M>np$, then the off-diagonal estimates in Proposition~\ref{prop:Lq off diagonal estimates} show that
\begin{align}\begin{split}\label{eq:NTOD}
    \| \mathds{1}_{Q(x,2t)} R_s f \|_p 
    & \leq \| \mathds{1}_{Q(x,2t)} R_s \mathds{1}_{Q(x,4t)} f \|_p  + \sum_{j=1}^{\infty} \| \mathds{1}_{Q(x,2t)} R_s \mathds{1}_{Q(x,2^{j+2}t)\setminus Q(x,2^{j+1}t)} f \|_p \\
    & \lesssim \sum_{j=0}^{\infty} \left( 1 + \frac{2^{j}t}{s} \right)^{-M} \| \mathds{1}_{2^j Q(x,4t)} f \|_p \\
    & \lesssim \sum_{j=0}^{\infty} 2^{-j(M-np)} t^{\frac{n}{p}}  \bigg(\dashint_{2^{j+2} Q(x,t)} |f|^p\bigg)^{\frac{1}{p}} \\
    & \lesssim t^{\frac{n}{p}} [M_*(|f|^p)(x)]^{\frac{1}{p}},
\end{split}\end{align}
where $M_*$ is the Hardy--Littlewood maximal operator, hence
\begin{align*}
\dashiint_{2W(x,t)} |R_s f (y) |^{p} \d y \d s
\lesssim \dashint_{t/2}^{5t/2} M_* (|f|^p) (x) \d s
= M_* (|f|^p) (x).
\end{align*}
Therefore, since $\frac{2}{p} > 1$, the bounds for the Hardy--Littlewood maximal operator imply that
\[
I_2 = \bigg\|\sup_{t>0} \dashiint_{2 W(x,t)} |R_s f (y)|^{p} \d y \d s\bigg\|_{\frac{2}{p}}^{\frac{2}{p}}
    \lesssim \|M_* (|f|^p)\|_{\frac{2}{p}}^{\frac{2}{p}} 
    \lesssim \|f\|_2^2,
\]
hence $\|{N}_* (F) \|_2 \lesssim I_1 + I_2 \lesssim \|f\|_2$, where an inspection of the preceding estimates shows that the implicit constants depend only on $n$, $\lambda$, $\Lambda$ and $\upsilon$, as required.
\end{proof}

The quadratic estimates in Theorem~\ref{thm:QE} also imply a Fatou-type result following the approach of Auscher--Ros\'{e}n--Rule in~\cite[Theorem~4.9]{ARR15}.

\begin{prop} \label{prop:Convergence on Whitney Averages}
Suppose that $n\geq 3$, $q\geq\max\{\tfrac{n}{2},2\}$ and $V \in B^{q}(\R^n)$ with $\llbracket V\rrbracket_{q}\leq \upsilon<\infty$. If $f \in L^2 (\R^n ; \C^{n+2} )$ and $F(t) = e^{-t [DB]} f$, then
\[
    \lim_{t \to 0^+} \dashiint_{W(x,t)} | F(s,y) - f(x)|^2 \d y \d s = 0
\]
for almost every $x \in \R^n$.
\end{prop}

\begin{proof} 
If $f\in \Nul(DB)$, then $e^{-s [DB]} f = f$ and the result follows by the Lebesgue differentiation theorem, so by \eqref{eq:Hodge} it remains to prove the result on $\Hil=\overline{\Ran(DB)}$. We follow \cite[Lemma~8.12]{AEN16-1} to reduce to a dense subspace $\mathcal{E}$ of $\Hil$. First, since $V\in B^{2}(\R^n)$, there exists $\delta\in(0,1)$ such that the off-diagonal estimates in Proposition~\ref{prop:Lq off diagonal estimates} hold. We choose $p\in(2,2+\delta)$ and claim that
$
\mathcal{E}\coloneqq \{ u \in \Ran(DB) \cap \D(DB) : DB u \in L^p (\R^n ; \C^{n+2}) \}
$
is dense in $\Hil$. To see this, define $T_m \coloneqq im R_{\frac{1}{m}} DB R_m$ for each $m\in\N$, where $R_m \coloneqq (I + i m DB)^{-1}$.

The following two facts imply that $\lim_{m\to\infty} \| (I - T_m) u \|_2 = 0$ for all $u\in \Hil$:

(i) If $u \in \D(DB)$, then $ \| (I - R_{\frac{1}{m}}) u \|_2 = \| \tfrac{i}{m} DB R_{\frac{1}{m}} u \|_2 \lesssim \tfrac{1}{m} \| DB u \|_2$, so a limit argument shows that $\lim_{m\to\infty} \|(I - R_{\frac{1}{m}}) u\|=0$ for all $u\in\overline{\D(DB)}=\Hil$, since $DB$ is densely defined.

(ii) If $u = DB v$ for some $v \in \D (DB)$, then
\[
    \| (I - im DB R_m) u \|_2 = \| R_m u \|_2 = \tfrac{1}{m} \| im DB R_m v\|_2 = \tfrac{1}{m} \| (I - R_m) v \|_2 \lesssim \tfrac{1}{m} \|v\|_2,
\]
so another limit argument shows that $\lim_{m\to\infty} \|(I - im DB R_m) u\|=0$ for all $u\in\overline{\Ran(DB)}=\Hil$.

Meanwhile, for each $u \in \Hil$, there exists $u_m$ in $\test(\R^n;\C^{n+2})$ such that $\lim_{m\to\infty} \|u-u_m\|_2=0$, whilst the $L^2$-resolvent bounds imply that $\sup_{m\in\N}\|T_m\|_{\mathcal{L}(\Hil)}<\infty$, hence
\[
    \|u-T_m u_m \|_2
    \lesssim \|(I-T_m)u\|_2 + \|u - u_m\|_2
\]
and $\lim_{m\to\infty}\|u-T_m u_m \|_2=0$. Observe that $T_m u_m$ is in $\Ran(DB)\cap\D(DB)$, whilst the resolvent bounds in Proposition~\ref{prop:Lq off diagonal estimates} show that
\[
    \| DB (T_m u_m) \|_p =  \|DB R_{\frac{1}{m}} (I - R_m) u_m \|_p = m \| (R_{\frac{1}{m}} - I) (I - R_m)  u_m \|_p \lesssim m \|u_m\|_p < \infty,
\]
so $T_m u_m \in \mathcal{E}$ and the claimed density holds.

We now prove the result when $f \in \mathcal{E}$ by writing
\begin{align*}
     &\dashiint_{W(x,t)} | F(s,y) - f(x)|^2 \d y \d s \\
     &\ \lesssim \dashiint_{W(x,t)} | \psi (s DB)|^2 \d y \d s 
     +\dashiint_{W(x,t)} | (R_s - I)f(y)|^2 \d y \d s 
     +\dashiint_{W(x,t)} | f(y) - f(x)|^2 \d y \d s \\
     &\ =: I_1(t,x) + I_2(t,x) + I_3(t,x),
\end{align*}
where $\psi$ in $\Psi(S_{\mu}^o)$ is given by $\psi (z) \coloneqq e^{-[z]} - (1 + i z)^{-1}$ for all $z\in S_{\mu}^o$ and $\mu\in(\omega(B),\tfrac{\pi}{2})$. The Lebesgue differentiation theorem implies that $\lim_{t \to 0^+} I_3(t,x)$ for almost every $x \in \R^n$.

To estimate $I_1(t,x)$, observe that
\[
    I_1(t,x) \leq \iint_{\R^{n+1}_+} \mathds{1}_{|y-x|<2s}(y) |\psi (s DB) f(y) |^2 \frac{\d y \d s}{s^{n+1}}
    \coloneqq J(x).
\]
The quadratic estimates from Theorem~\ref{thm:QE} and Theorem~\ref{thm:Quadratic estimates and functional calculus} show that
\[
\int_{\R^n} J(x) \d x
\lesssim \int_0^\infty \|\psi (s DB) f \|_2^2 \frac{\d s}{s}
\lesssim \|f\|_2^2.
\]
Therefore, applying Lebesgue's dominated convergence theorem twice, we obtain
\[
0 \leq \int_{\R^n} \lim_{t \to 0^+} I_1(t,x) \d x
  = \lim_{t \to 0^+} \int_{\R^n} I_1(t,x) \d x
  \lesssim \lim_{t \to 0^+} \int_0^{2t} \|\psi (s DB) f \|_2^2 \frac{\d s}{s}
  = 0,
\]
so $\lim_{t \to 0^+} I_1(t,x)=0$ for almost every $x \in \R^n$.

To estimate $I_2(t,x)$, we use the off-diagonal estimates from Proposition~\ref{prop:Lq off diagonal estimates}, as in \eqref{eq:NTOD} but with $p=2$, to obtain 
\begin{align*}
I_2(t,x) \lesssim t^{2-n} \|\mathds{1}_{Q(x,t)} R_s (DBf)\|_2^2
\lesssim t^2 M_*(|DB f|^2)(x),
\end{align*}
where $M_*$ is the Hardy--Littlewood maximal operator. We know that $ M_*(|DB f|^2) \in  L^{\frac{p}{2}} (\R^{n})$, since $p>2$ and $DB f \in L^p (\R^n ; \C^{n+2})$, so $M_*(|DB f|^2)(x) < \infty$ and $\lim_{t \to 0^+} I_3(t,x)=0$ for almost every $x \in \R^n$, so the result is proved when $f\in\mathcal{E}$.

More generally, for $f \in \Hil$, consider $(f_m)_{m\in\N}$ in $\mathcal{E}$ converging to $f$ in $L^2(\R^n)$ and write
\begin{align*}
\dashiint_{W(x,t)} |F(s,y) - f(x)|^2 \d y \d s 
& \lesssim \dashiint_{W(x,t)} | e^{-s [DB]}(f-f_m)(y)|^2 \d y \d s \\
&\quad+\dashiint_{W(x,t)} | (e^{-s [DB]}-I)f_m (y)|^2 \d y \d s \\
&\quad+\dashiint_{W(x,t)} | (f_m -f)(y)|^2 \d y \d s.
\end{align*}
The convergence result just proved on $\mathcal{E}$ shows that
\[
\limsup_{t\to 0^+} \dashiint_{W(x,t)} | F(s,y) - f(x)|^2 \d y \d s
\leq | {N}_* (e^{-(\cdot)[DB]}[f-f_m])(x)|^2 + |M_*([f-f_m]^2)(x)|
\]
for almost every $x \in \R^n$. Therefore, by Chebyshev's inequality, the non-tangential maximal function estimate from Theorem~\ref{thm:First-Order NT Control}, and the weak-type (1,1) estimate for the Hardy--Littlewood maximal operator, we have
\begin{align*}
|\{ x &\in \R^n: \limsup\nolimits_{t\to 0^+} \textstyle{\dashiint_{W(x,t)}} | F(s,y) - f(x)|^2 \d y \d s > \delta\}| \\
&\leq |\{ x \in \R^n: |{N}_*(e^{-(\cdot)[DB]}[f-f_m])(x)|^2 > \delta/2 \}|
+ |\{ x \in \R^n: |M_*([f-f_m]^2)(x)| > \delta/2 \}| \\
&\lesssim \delta^{-1} \|{N}_*(e^{-(\cdot)[DB]}[f-f_m])\|_{L^2(\R^n)}^2 + \delta^{-1} \|f-f_m\|_2^2 \\
&\lesssim \delta^{-1} \|f-f_m\|_2^2
\end{align*}
for all $\delta>0$, so the result follows by choosing $m\in\N$ such that $\|f-f_m\|_2<\delta$.
\end{proof}


\section{Return to the Second-Order Equation} \label{Sec:Return to the Second-Order Equation}\label{sec:con}


The purpose of the section is to make precise the equivalence between well-posedness of the boundary value problems $\Neu$ and $\Reg$ for the Schr\"{o}dinger equation $H_{A,a,V}u=0$ and the property that the boundary trace mappings $\Phi_N$ and $\Phi_R$ for the associated Cauchy--Riemann type initial value problems $\p F + DB F = 0$ are isomorphisms. These ideas, in the case $V\equiv0$, have their genesis in the work of Auscher--Axelsson--McIntosh (see~\cite[Section~4]{AAMc10-2}) and the subsequent development by Auscher--Ros\'{e}n (see~\cite[Section~3.1]{AA11}). We highlight our characterisation for Schr\"{o}dinger equations in Proposition~\ref{prop:WP equivalence} below. The short detailed proof brings together ideas developed throughout the paper and is included for the benefit of readers unfamiliar with the first-order treatment. This also allows us to conclude the proofs of the four theorems stated in the introduction.


\subsection{Equivalences for Well-Posedness}


It is convenient to note here that the non-tangential control ${N}_* (\gradIV u) \in L^2 (\R^n)$ required for solvability of $\Neu$ and $\Reg$ is sufficient to provide the correspondence between first-order solutions and second-order solutions in Proposition~\ref{prop:Reduction to first-order}.

\begin{lem} \label{lem:NT Control Implies Second-Order Has First-Order Solution}
Suppose that $n\in \N$, $V \in L^1_{\Loc}(\R^n)$ and $\ola{A}$ is the injective operator given by~
\eqref{eq:Ahatinv}. If $u \in \Wv^{1,2}_{\Loc}(\R^{n+1}_+)$ and ${N}_* (\gradIV u) \in L^2 (\R^n)$, then $
\gradAV u \in L^2_{\Loc} (\R_+ ; L^2 (\R^n ; \C^{n+2}))$.
\end{lem}

\begin{proof}
If $0<a<b<\infty$ and $l\in\N$ such that $2^l\geq b$, then Lemma \ref{lem:Initial bounds for NT} shows that
\begin{align*}
    \int_a^b \|\gradAV u\|_2^2 \d s
    \leq \sum_{k = 1}^{l} \int_{2^{k-1} a}^{2^{k} a} \| \ola{\cA}\gradIV u \|_2^2 \d s
    \lesssim \sum_{k = 1}^{l} 2^k a \left(\sup_{t>0} \dashint_{t}^{2 t} \|\gradIV u\|_2^2 \d s\right)
    \lesssim \|{N}_* (\gradIV u)\|_2^2,
\end{align*}
so the result follows.
\end{proof}

We now show how well-posedness of boundary value problems for  Schr\"{o}dinger equations manifests in the framework developed here for Cauchy--Riemann type initial value problems.

\begin{prop} \label{prop:WP equivalence}
Suppose that $n\geq 3$, $q\geq\max\{\tfrac{n}{2},2\}$ and $V \in B^{q}(\R^n)$ whilst $(A,a)$ are $t$-independent, bounded and elliptic coefficients satisfying \eqref{eq:bddcoeff} and \eqref{eq:elliptest}. If $D$ is the operator defined by~\eqref{eq:Ddef}, the coefficients  $B=\widehat{\mathcal{A}}_{A,a,V}$ are given by \eqref{eq:Ahatdef}, and the mappings $\Phi_N$ and $\Phi_R$ are given by \eqref{eqn:linear mappings}, then the following properties hold:
\begin{enumerate}
    \item The mapping $\Phi_N \colon \Hil_{DB}^+ \to L^2 (\R^n)$ is an isomorphism if and only if the Neumann problem $\Neu$ is well-posed.
    \item The mapping $\Phi_R \colon \Hil_{DB}^+ \to \{ \gradV g : g \in \Wvo^{1,2} (\R^n)\}$ is an isomorphism if and only if the Regularity problem $\Reg$ is well-posed.
\end{enumerate}
\end{prop}

\begin{proof}
We only prove (2), as (1) follows analogously. The following two properties prove that $\Phi_R \colon \Hil_{DB}^+ \to \{ \gradV g : g \in \Wvo^{1,2} (\R^n) \}$ is surjective if and only if $\Reg$ is solvable:

(a) If $g \in \Wvo^{1,2} (\R^n)$ and $u$ is a solution of $\Reg$ with data $\gradV g$, then Lemma~\ref{lem:NT Control Implies Second-Order Has First-Order Solution}, Proposition~\ref{prop:Reduction to first-order} and Lemma~\ref{lem:Initial bounds for NT}
imply that $F\coloneqq \gradAV u$ is in $L^2_{\Loc}(\R_+,\Hil)$ and $\p F+DBF=0$ in $\R_+$ with
\[
\sup_{t>0} \dashint_t^{2t} \|F (s) \|_2^2 \d s
\lesssim \|{N}_* (F) \|_2^2
=\|{N}_*(\ola{\cA}\gradV u)\|_2
\lesssim \|{N}_*(\gradV u)\|_2
< \infty.
\]
Theorem~\ref{Thm:First order solutions are semigroups} then provides $f \in \Hil_{DB}^+$ such that $F = e^{-t[DB]}f$ and $\lim_{t \to 0^+} \|F(t)-f\|_2=0$, so
\[
\Phi_R(f) = (f_{\parallel}, f_\oo) = \Nlim_{t \to 0^+} (F_{\parallel}(t), F_\oo(t)) =  \Nlim_{t \to 0^+} \gradVp u(t,\cdot) =  \gradV g,
\]
where the Whitney averages of $F$ converge to $f$ almost everywhere on $\R^n$ by Proposition~\ref{prop:Convergence on Whitney Averages} and the limit notation has the meaning given by \eqref{eqn:convergence on Whitney averages}.

(b) If $f\in \Hil_{DB}^+$, $g \in \Wvo^{1,2} (\R^n)$ and $\Phi_R (f) = \gradIV g$, then Proposition~\ref{Prop:Semigroup solves first order equation} shows that $F$ given by $F(t)\coloneqq e^{-t[DB]}f$ is in $L^2_{\Loc}(\R_+,\Hil)$ and $\p F+DBF=0$ in $\R_+$ with ${\lim_{t \to 0^+} \|F(t)-f\|_2=0}$. Proposition~\ref{prop:Reduction to first-order} then implies that there exists $u$ in $\Wv^{1,2}_{\Loc}(\R^{n+1}_+)$ such that $F = \gradAV u$ and $\HAaV u = 0$ in $\R^{n+1}_+$, hence \[
\gradV g = \Phi_R(f) = (f_{\parallel}, f_\oo) = \Nlim_{t \to 0^+} (F_{\parallel}(t), F_\oo(t)) =  \Nlim_{t \to 0^+} \gradVp u(t,\cdot),
\]
where the Whitney averages of $F$ converge to $f$ almost everywhere on $\R^n$ by Proposition~\ref{prop:Convergence on Whitney Averages}. Moreover, using \eqref{eq:Ahatinv} and Theorem~\ref{thm:First-Order NT Control}, we obtain
\[
\|{N}_*(\gradIV u)\|_2 = \|{N}_*(\ola{\cA}^{-1} \gradAV u)\|_2
\lesssim \|{N}_*(F)\|_2 \leq \|f\|_2 < \infty,
\]
so $u$ is a solution of $\Reg$ with data $\gradV g$. 

The following two facts prove that $\Phi_R \colon \Hil_{DB}^+ \to \{ \gradV g : g \in \Wvo^{1,2} (\R^n) \}$ is injective if and only if $\Reg$ is uniquely posed:

(c) If $\gradV g = 0$ in (a), then $\Phi_R(f)=0$. Therefore, if $\Phi_R \colon \Hil_{DB}^+ \to \{ \gradV g : g \in \Wvo^{1,2} (\R^n) \}$ is injective, then $f=0$, hence $\gradAV u = e^{-t[DB]}f=0$ and $\gradV u=0$. In that case, either $V\not\equiv0$ and $u=0$, or $V\equiv 0$ and $u$ is constant, so $\Reg$ is uniquely posed.

(d) If $\Phi_R (f) = 0$ and $\nabla_\mu g=0$ in (b), then $u$ is a solution of $\Reg$ with data $\nabla_\mu g=0$. Therefore, if $\Reg$ is uniquely posed, either $V\not\equiv0$ and $u=0$, or $V\equiv 0$ and $u$ is constant, hence $F=\gradAV u=0$, $f = \lim_{t \to 0^+} F (t) = 0$ and $\Phi_R \colon \Hil_{DB}^+ \to \{ \gradV g : g \in \Wvo^{1,2} (\R^n) \}$ is injective.
\end{proof}


\subsection{Proofs of the Main Theorems}


We conclude by completing the proofs of the four theorems stated in the introduction.

\begin{proof}[Proof of Theorem \ref{thm:Second-Order WP}]
The fact that $\Neu$ and $\Reg$ are well-posed follows immediately from the equivalence in Proposition~\ref{prop:WP equivalence}, given the isomorphisms in Propositions~\ref{prop:Block-Type Isomorphisms} and~\ref{prop:Self-Adjoint Isomorphisms}. Now observe that if $H_{A,a,V}u = 0$ in $\R^{n+1}_+$ with ${N}_* (\gradIV u) \in L^2 (\R^n)$, then proceeding as in the proof of Theorem~\ref{thm:Second-Order Fatou} below, we obtain $f\in \Hil_{DB}^+$ such that $\lim_{t \to 0^+} \|\partial_{\nu_{A}} u(t,\cdot) - f_\perp\|_2 =0$, $\lim_{t \to 0^+} \|\gradVp u (t,\cdot)- (f_\|,f_\oo)\|_2 = 0$
and
\[
    \int_0^{\infty} \|t \partial_t (\gradIV u) \|_2^2 \frac{\d t}{t}
    \eqsim \|{N}_*(\gradIV u)\|_2^2
    \eqsim \|f\|_2^2,
\]
where the implicit constants depend only on $n$, $\lambda$, $\Lambda$ and $\upsilon$. Meanwhile, the Rellich estimates from Propositions~\ref{prop:Block-Type Isomorphisms} and~\ref{prop:self adjoint injective} imply that $\|f\|_2 \eqsim \|f_{\perp}\|_2 \eqsim \|(f_{\parallel}, f_\oo)\|_2$ for all $f\in \Hil_{DB}^+$, where the implicit constants depend only on $n$, $\lambda$, $\Lambda$ and $\upsilon$. The required estimates follow, since if $u$ is the unique solution of $\Neu$ for data $h\in L^2(\R^n)$, then $f_\perp=h$, whilst if $u$ is the unique solution of $\Reg$ for data $g$ in $\Wvo^{1,2}(\R^n)$, then $(f_{\parallel}, f_\oo)=\gradV g$.
\end{proof}

\begin{proof}[Proof of Theorem \ref{thm:Second-Order WP pert}]
Suppose that $A_0$ and $a_0$ satisfy the bound and ellipticity in \eqref{eq:bddcoeff}
and \eqref{eq:elliptest} with constants $0<\lambda_0\leq\Lambda_0<\infty$. Next, consider $\varepsilon \in (0,\lambda_0/2]$ to be chosen later. Suppose that $A \in L^{\infty} (\R^n; \mathcal{L} (\C^{n+1}))$ and $a \in L^{\infty} (\R^n)$ satisfy $\|A - A_0 \|_{\infty} < \varepsilon$ and $\|a - a_0 \|_{\infty} < \varepsilon$, so $A_0$ and $a_0$ satisfy \eqref{eq:bddcoeff}
and \eqref{eq:elliptest} with constants $\lambda_0/2$ and $2\Lambda_0$. Now, recalling \eqref{eq:Ahatdef}, set $\mathcal{A}=\mathcal{A}_{A,a}$,  $\mathcal{A}_0=\mathcal{A}_{A_0,a_0}$, $B=\widehat{\mathcal{A}}$ and $B_0=\widehat{\mathcal{A}}_0$. There exists $C_0\in(0,\infty)$, depending only on $\lambda$ and $\Lambda$, such that
$\|B-B_0\|_{\infty} \leq C_0 \|\mathcal{A}-\mathcal{A}_0\|_{\infty} < \varepsilon C_0$, since \eqref{eq:Ahatinv} implies that
\[
B - B_0 
= \ula{\mathcal{A}}\ola{\mathcal{A}}^{-1} -  \ula{\mathcal{A}}_0\ola{\mathcal{A}}_0^{-1} \\
= \ula{\mathcal{A}} \ola{\mathcal{A}}^{-1} (\ola{\mathcal{A}}_0 - \ola{\mathcal{A}})\ola{\mathcal{A}}_0^{-1} + (\ula{\mathcal{A}} - \ula{\mathcal{A}}_0) \ola{\mathcal{A}}_0^{-1}.
\]
Therefore, the result follows from  Proposition~\ref{prop:WP equivalence} by choosing $\varepsilon$ sufficiently small so that the perturbation properties in Theorem~\ref{thm:absper} hold.
\end{proof}

\begin{proof}[Proof of Theorem \ref{thm:Second-Order Fatou}]
We have $\Hil = \Hil^+\oplus \Hil^-$ with $\Hil^\pm \coloneqq \Hil^\pm_{DB}$ by~\eqref{eq:tds} and Lemma~\ref{lem:RanD}. Now suppose that $H_{A,a,V}u = 0$ in $\R^{n+1}_+$ with ${N}_* (\gradIV u) \in L^2 (\R^n)$. Lemma~\ref{lem:NT Control Implies Second-Order Has First-Order Solution}, Proposition~\ref{prop:Reduction to first-order} and Lemma~\ref{lem:Initial bounds for NT} imply that $F\coloneqq \gradAV u$ is in $L^2_{\Loc}(\R_+,\Hil)$ and $\p F+DBF=0$ in $\R_+$ with
 \[
 \sup_{t>0} \dashint_t^{2t} \|F (s) \|_2^2 \d s
 \lesssim \|{N}_* (F) \|_2^2
 =\|{N}_*(\ola{\cA}\gradV u)\|_2
 \lesssim \|{N}_*(\gradV u)\|_2
 < \infty.
 \]
Theorem~\ref{Thm:First order solutions are semigroups} then provides $f \in \Hil_{DB}^+$ such that $F = e^{-t[DB]}f$, hence
\[
    \int_0^{\infty} \|t \partial_t (\gradIV u) \|_2^2 \frac{\d t}{t}
    \eqsim \int_0^{\infty} \|t \partial_t F \|_2^2 \frac{\d t}{t}
    \eqsim \|f\|_2^2
    \eqsim \|{N}_*(F)\|_2^2
    \eqsim \|{N}_*(\gradIV u)\|_2^2
\]
by Theorem \ref{thm:First-Order NT Control} and \eqref{eq:Ahatinv}, where the implicit constants depend only on $n$, $\lambda$, $\Lambda$ and $\upsilon$. Moreover, for almost every $x \in \R^n$, we have
\[
\lim_{t \to 0^+} \|F (t) - f \|_2 = 0 = \lim_{t \to 0^+} \dashiint_{W(x,t)} | F(s,y) - f(x)|^2 \d y \d s
\]
by Propositions~\ref{Prop:Semigroup solves first order equation} and~ \ref{prop:Convergence on Whitney Averages}. The result follows, since $f = (h,\gradV g)$ for some $h \in L^2 (\R^n)$ and $g \in \Wvo^{1,2} (\R^n)$ by Lemma~\ref{lem:RanD}, whilst $F=(\partial_{\nu_{A}} u,\gradVp u)$.
\end{proof}

\begin{proof}[Proof of Theorem \ref{thm:SFBVP}]
First consider when $\|V\|_{n/2} < \varepsilon_n$, where $\varepsilon_n\in(0,1)$ is from Lemma~\ref{lem:Riesz Transform Estimates with small norm}, so Theorem~\ref{thm:QE} and the results in Section~\ref{sec:ADD} hold with constants depending only on $n$, $\lambda$ and~$\Lambda$. We then follow the proof of Theorem~\ref{thm:Second-Order WP} above, deducing well-posedness of $\SReg$ and $\SNeu$ as in Proposition~\ref{prop:WP equivalence} from the isomorphisms in Propositions~\ref{prop:Block-Type Isomorphisms} and~\ref{prop:Self-Adjoint Isomorphisms}, without requiring the non-tangential maximal function bounds from Section~\ref{sec:NT}. Observe that Proposition~\ref{prop:Self-Adjoint Isomorphisms} allows for when the product $aV$ is real-valued and $A$ is Hermitian. Moreover, the implicit constants in the associated square function bounds \eqref{eq:SFEsol} depend only on $n$, $\lambda$ and $\Lambda$. 

Now consider when $\|V\|_{n/2} \geq \varepsilon_n$ and set  $c_n\coloneqq\tfrac{\varepsilon_n}{2} \|V\|_{n/2}^{-1} \in (0,1)$. The case above provides unique solutions of $\SReg$ and $\SNeu$ for coefficients $(\widetilde{A},\widetilde{a})\coloneqq (c_n A,c_n a)$ and potential $\widetilde{V}\coloneqq c_n V$ with boundary data $\gradV g \in L^2 (\R^n;\C^{n+1})$ and $\widetilde{h}\coloneqq c_n h \in L^2 (\R^n)$, since $\|\widetilde{V}\|_{n/2}<\varepsilon_n$. The implicit constants in the associated square function bounds \eqref{eq:SFEsol} for these solutions will depend only on $n$, $\lambda$, $\Lambda$ and $\upsilon$, since $(\widetilde{A},\widetilde{a})$ satisfy the bound \eqref{eq:bddcoeff} and ellipticity \eqref{eq:elliptest} with constants depending only on $\lambda$, $\Lambda$ and $\upsilon$. These solutions also uniquely solve $\SReg$ and $\SNeu$ for the original coefficients $(A,a)$ and potential $V$ with boundary data $\gradV g$ and $h$, as required.
\end{proof}


\appendix

\section{Holomorphic Functional Calculus for Bisectorial Operators}\label{sec:app}


Here we provide an overview of the holomorphic functional calculus for bisectorial operators and its relationship with analytic semigroups and quadratic estimates that is used throughout the paper. These are included to make the paper self-contained and for the benefit of readers unfamiliar with holomorphic functional calculus. The ideas originate in the pioneering work of McIntosh~\cite{Mc88}. More recent developments and further details can be found in, for instance, the work of Albrecht--Duong--McIntosh~ \cite{ADMc96}, Egert~\cite{EgertPhD} and Haase~\cite{Haase06}.

We start by introducing notation for closed operators and spaces of holomorphic functions. An \textit{operator} $T$ on a Hilbert space $\Hil$ refers to a linear mapping $T:\D (T)\to\Hil$ with a domain $\D (T)$ that is a linear subspace of $\Hil$. We will also consider the range $\Ran(T)\coloneqq\{Tu : u \in \D(T)\}$ and null space $\Nul(T)\coloneqq\{u\in\D(T): Tu = 0\}$ of such an  operator. An operator is \textit{closed} if its graph is closed, or equivalently, if its domain is complete with respect to the \textit{graph norm}
\begin{equation}\label{graph norm}
\|u\|_{\D (T)} \coloneqq \|u\|_{\Hil} + \|Tu\|_{\Hil}
\end{equation}
for all $u\in\D (T)$. The set of all closed operators on $\Hil$ is denoted by $\mathcal{C}(\Hil)$. If $S,T\in\mathcal{C}(\Hil)$, then we write $S\subseteq T$ when $\D(S)\subseteq\D(T)$ and $Su=Tu$ for all $u\in\D(S)$, whilst $S=T$ means that $S\subseteq T$ and $T\subseteq S$. The subspace of all bounded operators, that is, those $T$ in $\mathcal{C}(\Hil)$ with $\D(T)=\Hil$ and finite \textit{operator norm} \[\|T\|\coloneqq\sup\{\|Tu\|_\Hil: u\in\D(T),\, \|u\|_\Hil=1\} <\infty,\] is denoted by $\mathcal{L}(\Hil)$. If $T \in  \mathcal{C}(\Hil)$, then the \textit{resolvent} is the set 
\[
\rho(T)\coloneqq\{\lambda\in \mathbb{C} : (\lambda I - T):\D(T)\to\Hil\text{ is bijective and }(\lambda I - T)^{-1} \in \mathcal{L}(\Hil)\},
\]
whilst the \textit{spectrum} $\sigma (T)$ is the subset of the extended complex plane $\C_\infty\coloneqq\mathbb{C}\cup\{\infty\}$ defined to be $\C\setminus\rho(T)$, when $T\in\mathcal{L}(\Hil)$, and $(\C\setminus\rho(T))\cup\{\infty\}$, when $T\notin\mathcal{L}(\Hil)$.

\subsection{\texorpdfstring{The $\mathcal{F}$-Functional Calculus}{The F-Functional Calculus}}

We now define sectorial and bisectorial operators on a Hilbert space. If $0 \leq \omega < \mu \leq \pi$, then the closed sector $S_{\omega +}$ in $\C_\infty$, and the open sector $S_{\mu+}^o$ in $\C$, are
\begin{align*}
    S_{\omega +} \coloneqq \{ z \in \C : | \arg (z) | \leq \omega \} \cup \{ 0, \infty \} \quad\text{and}\quad
    S_{\mu +}^o  \coloneqq \{ z \in \C : | \arg (z) | < \mu,\ z \neq 0 \}.
\end{align*}
A closed operator $T$ on $\Hil$ is called \textit{sectorial of type $S_{\omega +}$} if $\sigma (T) \subseteq S_{\omega+}$ and, for each $\mu \in (\omega,\pi]$, there exists $C_{\mu} \in (0,\infty)$ such that the uniform resolvent bound $\norm{(\lambda I-T)^{-1}}{}{} \leq C_{\mu} |\lambda|^{-1}$ holds for all $\lambda \in \C \setminus S_{\mu+}$. If $0 \leq \omega < \mu \leq \tfrac{\pi}{2}$, then we can also consider the related bisectors
\[
S_{\omega} \coloneqq S_{\omega+} \cup (- S_{\omega+})
\quad\text{and}\quad
S_{\mu}^o \coloneqq S_{\mu+}^o \cup (- S_{\mu+}^o).
\]
A closed operator $T$ on $\Hil$ is called \textit{bisectorial of type $S_{\omega}$} if $\sigma (T) \subseteq S_{\omega}$ and, for each $\mu \in (\omega,\frac{\pi}{2}]$, there exists $C_{\mu} \in (0,\infty)$ such that $\norm{(\lambda I-T)^{-1}}{}{} \leq C_{\mu} |\lambda|^{-1}$ for all $\lambda \in \C \setminus S_{\mu}$.

The following function spaces are needed to develop the holomorphic functional calculus for bisectorial operators below. If $0<\theta\leq\frac{\pi}{2}$, then $H(S_{\theta}^o)$ denotes the space of $\C$-valued holomorphic functions on $S_{\theta}^o$, and we consider the following subspaces $\Psi(S_{\theta}^o) \subseteq H^{\infty} (S_{\theta}^o) \subseteq \mathcal{F} (S_{\theta}^o) \subseteq H(S_{\theta}^o)$: 
\begin{align*}
    H^{\infty} (S_{\theta}^o) & \coloneqq H(S_{\theta}^o) \cap L^\infty(S_{\theta}^o); \\
    \Psi(S_{\theta}^o) & \coloneqq \{ \psi \in H^{\infty} (S_{\theta}^o) : \exists\, s , C \in (0,\infty), \, | \psi (z) | \leq C |z|^s(1 + |z|^{2s})^{-1} \, \forall z \in S_{\theta}^o \}; \\
    \mathcal{F} (S_{\theta}^o) & \coloneqq \{ f \in H(S_{\theta}^o) : \exists\, s,C \in (0,\infty), \, | f (z) | \leq C ( |z|^s + |z|^{-s}) \, \forall z \in S_{\theta}^o\}.
\end{align*}
If $0<\theta\leq\pi$, then the spaces $\Psi(S_{\theta+}^o) \subseteq H^{\infty} (S_{\theta+}^o) \subseteq \mathcal{F} (S_{\theta+}^o) \subseteq H(S_{\theta+}^o)$ are defined analogously by replacing the bisector $S_{\theta}^o$ with the sector $S_{\theta+}^o$.

The following result is well-known. Although it is not stated explicitly as a stand-alone result in the literature, it can be proved by combining results in \cite[Theorem 2.3]{CDMcY}, \cite[Lecture 2]{ADMc96} and \cite[Section 2]{AMcN97} or \cite[Sections 1.2--1.3, 2.1--2.3, 5.1]{Haase06}.

\begin{thm}\label{thm:FfcDef}
If $0 \leq \omega < \tfrac{\pi}{2}$ and $T$ is an injective bisectorial operator of type $S_{\omega}$ on $\Hil$, then $T$ has dense domain and dense range (written $\overline{D(T)}=\overline{\Ran(T)}=\Hil)$. Moreover, for each $\mu\in(\omega,\tfrac{\pi}{2}]$, there exists a unique mapping from $\mathcal{F}(S_{\mu}^o)$ into $\mathcal{C}(\Hil)$, called \textit{the $ \mathcal{F}(S_{\mu}^o)$-functional calculus for $T$}, and denoted by $f\mapsto f(T)$ for each $f\in\mathcal{F}(S_{\mu}^o)$, with the following properties:
\begin{enumerate}
\item
If $f$ is a rational function that has all of its poles in $\C\setminus S_\mu$, then $f(T)$ is consistent with the usual definition for polynomials and resolvent operators of $T$ in $\mathcal{C}(\Hil)$. Specifically, if $N\in\mathbb{N}$, $\{p_0,p_1,\ldots,p_N\}\subset \C$ and $\lambda \in \C\setminus S_\mu$, with $ p(z)\coloneqq\sum_{i=0}^N p_i z^i$ and $r(z)\coloneqq(\lambda-z)^{-1}$ for all $z\in S_\mu^o$, then $p(T)= (p_0I + p_1 T + \ldots + p_N T^N)$ and $r(T)=(\lambda I-T)^{-1}$;
\item
If $f$ and $g$ are in $\mathcal{F}(S_{\mu}^o)$, then $f(T)+g(T) \subseteq (f+g)(T)$ and $f(T)g(T) \subseteq (fg)(T)$ with $\D(f(T)g(T)) = \D((fg)(T)) \cap \D(g(T))$ (so both relations are equality when $g(T)\in\mathcal{L}(H)$);
\item
If $(f_n)_{n\in\mathbb{N}}$ is a sequence in $H^\infty(S_\mu^o)$ that converges uniformly on compact sets to some $f \in  H^\infty(S_\mu^o)$ with $\textstyle \sup_{n\in\mathbb{N}} \|f_n\|_{\infty} < \infty$ and $\textstyle \sup_n \|f_n(T)\| < \infty$, then $f(T)\in\mathcal{L}(H)$, $\textstyle \lim_{n\to\infty} f_n(T)u = f(T)u$ for all $u\in \Hil$, and  $\textstyle\|f(T)\| \leq \sup_{n\in\mathbb{N}} \|f_n(T)\|$.
\end{enumerate}
If $0 \leq \omega < \pi$ and $T$ is an injective sectorial operator of type $S_{\omega+}$ on $\Hil$, then the analogous results hold for each $\mu\in(\omega,\pi]$, upon replacing $S_{\mu}^o$ with $S_{\mu+}^o$, and the resulting mapping $f\mapsto f(T)$ for each $f\in\mathcal{F}(S_{\mu+}^o)$ is called \textit{the $ \mathcal{F}(S_{\mu+}^o)$-functional calculus for $T$}.
\end{thm}

The above result concerns bisectorial operators that are injective. The following remark shows how to restrict the domain of a general bisectorial operator to define its injective part.

\begin{rem}\label{rem:injpart}
If $T$ is a bisectorial operator of type $S_{\omega}$ on $\mathcal{H}$, which is not necessarily injective, then $\Hil = \overline{\Ran(T)} \oplus \Nul(T)$, where the direct sum may not be orthogonal. In that case, if $\widetilde{T}$ denotes the restriction of $T$ to $\overline{\Ran (T)}$, then $\widetilde{T}$ is an injective bisectorial operator of type $S_{\omega}$ on $\overline{\Ran(T)}$, called the \textit{injective part of $T$}, and $(\lambda I - \widetilde{T})^{-1}u = (\lambda I - T)^{-1}u$ for all $u\in\overline{\Ran(T)}$ and $\lambda \in\sigma(T)\cap\C$ (see, for instance, \cite[Theorem~3.8]{CDMcY} or \cite[Example~3.2.16]{EgertPhD}).
\end{rem}

Now suppose that $T$ is an injective bisectorial operator of type $S_{\omega}$ on $\Hil$ as in Theorem~\ref{thm:FfcDef}. If $\psi \in \Psi (S_{\mu}^o)$, then $\psi(T) \in \mathcal{L}(H)$ and we have the Cauchy integral representation
\begin{equation}\label{eq:CIF}
\psi (T) u = \frac{1}{2 \pi i} \int_{+\partial S_\theta^o} \psi (z) (zI-T)^{-1}u \d z
\end{equation}
for all $u\in \Hil$, where $+\partial S_\theta^o$ is the boundary of $S_\theta^o$ with positive (counterclockwise) orientation for any $\theta \in (\omega,\mu)$. Moreover, if $f \in \mathcal{F} (S_{\mu }^o)$, then $f(T) \in \mathcal{C}(H)$ with
$
f(T)u = (\psi (T))^{-1} (f \psi) (T)u
$
for all $u\in \D(f(T)) = \{u\in\Hil: (f \psi) (T)u \in \Ran(\psi (T))\}$, where $\psi$ is any function in $\Psi (S_{\mu}^o)$ with the property that $\psi(T)$ is injective and $f \psi \in \Psi (S_{\mu}^o)$. The injectivity of $T$ guarantees that such a regularising function $\psi$ exists. For example, if $k\in\mathbb{N}$ and $|f (z)| \lesssim |z|^k + |z|^{-k}$ for all $z\in S^o_\theta$, then we can use $\psi(z) \coloneqq ({z}/(i+z)^{2})^{k+1}$ for all $z\in S^o_\theta$, since $\psi(T)=T^{k+1}(i+T)^{-2(k+1)}$ by Theorem~\ref{thm:FfcDef}, and thus $\psi(T)$ is injective.

In this context, the {\it Hermitian adjoint}  $T^*:\D(T^*)\to\Hil$ is also an injective bisectorial operator of type $S_{\omega}$ on $\Hil$, so there exists a unique $\mathcal{F}(S_{\mu}^o)$-functional calculus for $T^*$ by Theorem~\ref{thm:FfcDef}. Moreover, it can be shown that 
\begin{equation}\label{eq:adjFC}
f(T^*)=(\conj{f}(T))^*
\end{equation}
for all $f\in\mathcal{F}(S_\mu^o)$, where  $\conj{f}(z)\coloneqq\conj{f(\conj{z})}$ for all $z\in S_\mu^o$ (see, for instance, \cite[Corollary 3.7]{CDMcY} or  \cite[Proposition~2.6.3]{Haase06}). It is also useful to note here that if $U\in\mathcal{L}(\Hil)$ is bijective and  $U^{-1}\in\mathcal{L}(\Hil)$, then $U^{-1}TU$ is  an injective bisectorial operator of type $S_{\omega}$ on $\Hil$, and we have the representation
\begin{equation}\label{eq:simFC}
f(U^{-1}TU) = U^{-1} f(T) U
\end{equation}
for all $f \in\mathcal{F}(S_\mu^o)$ (see, for instance, \cite[Proposition~3.2.10]{EgertPhD}).

\subsection{Analytic Semigroups}
Sectorial operators can be characterised as the generators of bounded analytic semigroups, and their holomorphic functional calculus provides a convenient framework to exploit this (see, for instance, \cite[Section~2]{AMcN970} or \cite[Proposition~ 3.4.4]{Haase05}). The related result for bisectorial operators in Proposition~\ref{prop:SGfromFC} below is a principal motivation for the first-order approach to boundary value problems adopted in this paper, as it provides the semigroup representations for solutions.

To state the result, we introduce the holomorphic function $[\cdot]:S_{\frac{\pi}{2}}^o\to\C$ given by
\begin{equation}\label{eq:[]def}
[\cdot](z) \coloneqq [z] 
\coloneqq 
\begin{cases}
+z, \quad & \text{if} \, \Re (z) > 0; \\
-z, \quad & \text{if} \, \Re(z) < 0.
\end{cases}
\end{equation}
Observe that $[z] = \sqrt{z^2}$, using the principal branch of the square root function, whilst $[\cdot]\in \mathcal{F}(S_\theta^o)$ for each $\theta\in(0,\frac{\pi}{2})$, hence $[T]\coloneqq[\cdot](T)$ is a well-defined closed operator whenever $T$ is injective and bisectorial.

The result below is well-known (see, for instance, \cite[Proposition~8.1]{AMcN970}) but a self-contained proof is provided for convenience. The proof that $[T]$ is sectorial also avoids the usual need to prove that $[T]=\sqrt{T^2}$, which requires a composition rule connecting the functional calculus for $T$ with that for $T^2$. Instead, we argue directly, considering the operator $[T]$ using only the functional calculus for $T$.

\begin{prop}\label{prop:SGfromFC}
If $0 \leq \omega < \tfrac{\pi}{2}$ and $T$ is an injective bisectorial operator of type $S_{\omega}$ on $\Hil$, then $[T]\coloneqq [\cdot](T)$ is a sectorial operator of type $S_{\omega+}$ on $\Hil$. Moreover, the family of operators $e^{-\zeta [T]}\coloneqq(e^{-\zeta[\cdot]})(T)$, $\zeta \in S_{(\frac{\pi}{2} - \omega)+}^o\cup\{0\}$, is an analytic semigroup on $\Hil$ with generator $-[T]$, that is, the following properties hold:
\begin{enumerate}
\item If $\zeta_1,\zeta_2\in S_{(\frac{\pi}{2} - \omega) +}^o\cup\{0\}$, then $e^{-\zeta_1 [T]} e^{-\zeta_2 [T]} = e^{-(\zeta_1+\zeta_2) [T]}$;
\item The mapping $\zeta \mapsto e^{-\zeta [T]}$ from $ S_{(\frac{\pi}{2}-\omega)+}^o$ into $\mathcal{L} (\Hil)$ is holomorphic, and
its derivative satisfies $\frac{\partial}{\partial\zeta} (e^{-\zeta [T]}u) = - [T] e^{-\zeta [T]}u$
for all $u\in\Hil$ (so $e^{-\zeta [T]}u \in \D([T])$ for all $u\in\Hil$);
\item If $\theta \in (0,\frac{\pi}{2}-\omega)$, then $\sup \{ \|e^{-\zeta [T]}\| : \zeta \in S_{\theta+}^o \} < \infty$, $\lim_{|\zeta| \to 0,\, \zeta \in S_{\theta+}^o} \|e^{-\zeta [T]} u - u\|_\Hil = 0$ and $\lim_{|\zeta| \to \infty,\, \zeta \in S_{\theta+}^o} \|e^{-\zeta [T]} u\|_\Hil = 0$ for all $u \in \Hil$;
\item If $\lambda\in S_{\frac{\pi}{2}+}^o$, then $\displaystyle (\lambda+[T])^{-1}u=\int_0^\infty e^{-\lambda t} e^{-t[T]}u \d t$ for all $u \in \Hil$.
\end{enumerate}
\end{prop}

\begin{proof}
We first prove that $[T]$ is sectorial. Suppose that $\omega \in[0,\tfrac{\pi}{2})$ and that $\omega<\theta<\mu\leq\pi$. Suppose that $\lambda\in \mathbb{C}\setminus S_{\mu+}$ and define $f_\lambda(z)\coloneqq1/(\lambda-[z])$ for all $z\in S_{\theta}^o$. Observe that $f_\lambda$ is in $H^\infty(S_\theta^o)$, since $\|f_\lambda\|_{L^\infty(S_{\theta}^o)}\leq 1/\dist(\lambda,S_{\theta+}) < \infty$, so the $\mathcal{F} (S_{\theta}^o)$-functional calculus for $T$ allows us to write $
\lambda f_\lambda(T) = \psi_\lambda(T) + R_\lambda(T)$, where 
\[
\psi_\lambda(z)
\coloneqq \frac{\lambda}{\lambda-[z]} - \frac{i|\lambda|}{i|\lambda|-z}
=  \frac{i\lambda z + |\lambda|[z]}{(\lambda-[z])(|\lambda|+iz)}
\quad\text{and}\quad
R_\lambda(z)
\coloneqq\frac{i|\lambda|}{i|\lambda|-z}
\]
for all $z\in S_{\theta}^o$. The resolvent bounds for $T$ imply there exists $C_\mu\in(0,\infty)$ such that 
\[
\|\psi_\lambda(T)\| \lesssim \frac{C_\mu}{\sqrt{1-\cos(\mu-\theta)}}
\quad\text{and}\quad
\|R_\lambda(T)\|
=|i\lambda|\|(i|\lambda|-T)^{-1}\|
\leq C_\mu,
\]
where the implicit constants do not depend on $\lambda$. To see this, observe that $\psi_\lambda \in \Psi(S_{\theta}^o)$ with
\[
|\psi_\lambda(z)| \leq 2\left(\frac{|\lambda z|^{1/2}}{|\lambda-[z]|} \right)
\left( \frac{|\lambda z|^{1/2}}{||\lambda|+iz|}\right)
\lesssim \frac{1}{\sqrt{1-\cos(\mu-\theta)}} \min\left\{\frac{|z|}{|\lambda|},\frac{|\lambda|}{|z|}\right\}^{\tfrac{1}{2}},
\]
since $
|\lambda- [z]|^2 \geq 2  (1-\cos(\mu-\theta)) |\lambda z|$ for all $z \in S_\theta^o$. The resolvent bounds for $T$ then give 
\[
\|\psi_\lambda (T)\|
\lesssim \frac{C_\mu}{\sqrt{1-\cos(\mu-\theta)}} 
\int_{+\partial S_\theta} \min\{|\zeta|,|\zeta|^{-1}\}^{\tfrac{1}{2}}\ \frac{|\d\zeta|}{|\zeta|}
\lesssim \frac{C_\mu}{\sqrt{1-\cos(\mu-\theta)}}
\]
for all $\theta\in(\omega,\mu)$.

We have now shown that $\|f_\lambda (T)\| \lesssim (C_\mu/\sqrt{1-\cos(\mu-\omega)} )|\lambda| ^{-1}$, whilst the $\mathcal{F} (S_{\theta}^o)$-functional calculus for $T$  implies that $f_\lambda(T)=(\lambda - [T])^{-1}$, hence $\lambda\in\rho([T])$. It follows that $\sigma([T])\subseteq S_{\omega+}$, since $\mu\in(\omega,\pi]$ was arbitrary, and thus $[T]$ is sectorial of type $S_{\omega+}$ on $\Hil$. 

To prove the analytic semigroup properties, note that by Theorem~\ref{thm:FfcDef}, for each $\mu\in(\omega,\pi]$, the $\mathcal{F}(S_{\mu+}^o)$-functional calculus for $[T]$ exists. In particular, if $\lambda\in\C\setminus S_{\mu+}$ and $r_\lambda(z)\coloneqq1/(\lambda-z)$ for all $z\in S_{\mu+}^o$, then $r_\lambda([T])=(\lambda - [T])^{-1}$. Therefore, since $f_\lambda = r_\lambda\circ[\cdot]$ in the preceding paragraph, we have actually proven the composition rule  $(r_\lambda\circ[\cdot])(T)=r_\lambda([T])$, relating the functional calculus for $T$ with the functional calculus for $[T]$.

To complete the proof, we can essentially follow the proof of (a)-(d) in
\cite[Proposition~3.4.1]{Haase06}, except the arguments which rely on Cauchy integral representations for the functional calculus in a neighbourhood of the origin must be modified to instead treat the function $z\mapsto e^{-t[z]}$. This can be achieved by using the composition rule just noted to write $(e^{-t[\cdot]})(T) = \psi(T) + (i+[T])^{-1}$, where $\psi(z) \coloneqq e^{-t[z]} - (i+[z])^{-1}$ for all $z\in S_\theta^o$. The Cauchy integral representation for $\psi(T)$ in \eqref{eq:CIF} can then be substituted for those used in the aforementioned reference.
\end{proof}

The following remark characterises the spectrum of $[T]$, although we do not use this result.

\begin{rem}
More generally, it is straightforward to adapt the proof of the Spectral Mapping Theorem obtained by Haase in \cite[Theorem~2.7.8]{Haase06} (cf. \cite[Theorem~6.4]{Haase05}) to bisectorial operators. Therefore, since $[\cdot]$ is holomorphic on any bisector with polynomial limits at 0 and $\infty$ in $\C_\infty$ (equal to $0$ and $\infty$, respectively, in the sense of \cite[Section~2.2]{Haase06}), this shows that
\[
\sigma([T])=[\sigma(T)]
\coloneqq \{[z]: z\in\sigma(T)\},
\]
where $[0]\coloneqq 0$ and $[\infty]\coloneqq \infty$, whenever $T$ is bisectorial.
\end{rem}

\subsection{Quadratic Estimates}
An injective bisectorial operator $T$ of type $S_{\omega}$ is said to have a \textit{bounded $H^{\infty}(S_\omega)$-functional calculus} when there exists $\mu\in(\omega,\tfrac{\pi}{2}]$ and $C\in(0,\infty)$ such that $\|f(T)\| \leq C \|f\|_{\infty}$ for all $f \in H^{\infty} (S_{\mu}^o)$. This property is crucial in the first-order approach to boundary value problems, as it implies that the spectral projections $\chi^\pm(T)$ associated with the sectors $\pm S_{\omega+}$ are bounded (see Section~\ref{ssect:GWP}). The equivalence with \textit{quadratic estimates} below, and many other properties, is well-known (see \cite[Corollary~E and Theorem~F]{ADMc96}).

\begin{thm} \label{thm:Quadratic estimates and functional calculus}
Suppose that $0 \leq \omega < \tfrac{\pi}{2}$ and $T$ is an injective bisectorial operator of type $S_{\omega}$ on $\Hil$. If $\mu\in(\omega,\tfrac{\pi}{2}]$, $\sup_{\lambda\in\C\setminus S_\mu}\norm{\lambda(\lambda I-T)^{-1}}{}{}\leq C_\mu<\infty$ and there exists $C_0\in[1,\infty)$ such that
\[
\int_0^{\infty} \| t T (I + (tT)^2)^{-1} u\|_\Hil^2 \frac{\d t}{t} +\int_0^{\infty} \| t T^* (I + (tT^*)^2)^{-1} u\|_\Hil^2 \frac{\d t}{t} 
\leq C_0 \|u\|_\Hil^2
\ \text{ for all } u \in \Hil,
\]
then there exists $\widetilde{C}_{\mu} \in[1,\infty)$, depending only on $C_0$ and $C_\mu$, such that $\|f(T)\| \leq \widetilde{C}_{\mu} \|f\|_{\infty}$ and
\[
\widetilde{C}_\mu^{-1} \|u\|_\Hil^2
\leq \int_0^{\infty} \|\psi(tT)u \|_\Hil^2 \frac{\d t}{t}
\leq \widetilde{C}_\mu \|u\|_\Hil^2
\ \text{ for all }\ u \in \Hil,
\]
whenever $f \in H^{\infty} (S_{\mu}^o)$, $\psi\in\Psi(S_{\mu}^o)$ and $\psi$ is not identically zero on $S_{\mu+}^o$ nor on $-S_{\mu+}^o$.
\end{thm}


\printbibliography

\end{document}